\documentclass[12pt]{article}
\usepackage{amsmath,amscd,amsbsy,amssymb,latexsym,url,bm,amsthm}
\usepackage[vlined,boxed,commentsnumbered,linesnumbered,ruled]{algorithm2e}
\usepackage{epsfig,graphicx,subfigure}
\usepackage{enumitem,balance,mathtools}
\usepackage{wrapfig}
\usepackage{mathrsfs, euscript}
\usepackage[usenames]{xcolor}
\usepackage{hyperref}
\usepackage{epstopdf}
\usepackage{float}
\usepackage{comment}
\author{Fei Wang\thanks{School of Mathematical Sciences, CMA-Shanghai, Shanghai Jiao Tong University, Shanghai,  China \href{mailto:fwang256@sjtu.edu.cn}{\texttt{fwang256@sjtu.edu.cn}}} \and Zeren Zhang\thanks{School of Mathematical Sciences, Shanghai Jiao Tong University, Shanghai, China \href{mailto:zhangzr0018@sjtu.edu.cn}{\texttt{zhangzr0018@sjtu.edu.cn}}}}
\title{The stability threshold for 2D MHD equations around Couette with general  viscosity and magnetic resistivity}
\date{}
\allowdisplaybreaks[4]
\newtheorem{theorem}{Theorem}[section]
\newtheorem{lemma}{Lemma}[section]
\newtheorem{proposition}{Proposition}[section]

\newtheorem{remark}{Remark}[section]

\numberwithin{equation}{section}

\begin{document}
	\maketitle
	
	\begin{abstract}
		We address a threshold problem of the Couette flow $(y,0)$ in a uniform magnetic field $(\beta,0)$ for the 2D MHD equation on $\mathbb{T}\times\mathbb{R}$ with fluid viscosity $\nu$ and magnetic resistivity $\mu$. The nonlinear enhanced dissipation and inviscid damping are also established. In particularly, when $0<\nu\leq\mu^3\leq1$, we get a threshold $\nu^{\frac{1}{2}}\mu^{\frac{1}{3}}$ in $H^N(N\geq4)$. When $0<\mu^3\leq\nu\leq1$, we obtain a threshold $\min\{\nu^{\frac{1}{2}},\mu^{\frac{1}{2}}\}\min\{1,\nu^{-1}\mu^{\frac{1}{3}}\}$, hence improving the results in \cite{D24,ZZ23b,K24}.
	\end{abstract}
	\section{Introduction}
	In this paper, we study the nonlinear stability of the Couette flow $(y,0)$ in a uniform magnetic field $(\beta,0)$ for the 2-D MHD equations on the domain $(x,y)\in\mathbb{T}\times\mathbb{R}$:
	\begin{equation*}
		\left\{\begin{array}{l}
			\partial_t\tilde{u}+\tilde{u}\cdot\nabla\tilde{u}-\tilde{b}\cdot\nabla\tilde{b}-\nu\Delta\tilde{u}+\nabla\tilde{p}=0,\\
			\partial_t\tilde{b}+\tilde{u}\cdot\nabla\tilde{b}-\tilde{b}\cdot\nabla\tilde{u}-\mu\Delta\tilde{b}=0,\\
			\nabla\cdot\tilde{u}=\nabla\cdot\tilde{b}=0,\\
			\tilde{u}|_{t=0}=\tilde{u}^{in},\quad\tilde{b}|_{t=0}=\tilde{b}^{in},
		\end{array}\right.
	\end{equation*}
	where $\tilde{u}$ is the velocity, $\tilde{b}$ is the magnetic field, $\tilde{p}$ is the pressure, $\nu$ is the fluid viscosity, and $\mu$ is the magnetic resistivity.
	Let $u=\tilde{u}-(y,0),b=\tilde{b}-(\beta,0)$ be the perturbation of velocity and magnetic field. We introduce the vorticity and the current density
	\begin{equation*}
		\omega=\nabla^\bot\cdot u,\ j=\nabla^\bot\cdot b,
	\end{equation*}
	and then $(\omega,j)$ satisfies
	\begin{equation} \label{mhd}
		\left\{\begin{array}{l}
			\partial_t\omega+y\partial_x\omega-\beta\partial_xj-\nu\Delta\omega=\mathrm{NL}^\omega,\\
			\partial_tj+y\partial_xj-\beta\partial_x\omega-\mu\Delta j+2\partial_{xy}\phi=\mathrm{NL}^j,\\
			\omega|_{t=0}=\omega^{in},\quad j|_{t=0}=j^{in},\\
			u=\nabla^\bot\psi,\quad b=\nabla^\bot\phi,\\
			\psi=\Delta^{-1}\omega,\quad \phi=\Delta^{-1}j,\\
		\end{array}\right.
	\end{equation}
	where 
	\begin{align*}
		\mathrm{NL}^{\omega}:&=-u\cdot \nabla\omega+b\cdot\nabla j,\\
		\mathrm{NL}^j:&=-u\cdot \nabla j+b\cdot\nabla\omega+2\partial_{xy}\phi(\omega-2\partial_{xx}\psi)-2\partial_{xy}\psi(j-2\partial_{xx}\phi).
	\end{align*}
	
		We are concerned with the stability threshold problem in the Sobolev space and formulate it as \cite{BGM17}:
	
		\textbf{Stability threshold}: Given $N\geq0$, determine $\gamma_1,\gamma_2\in\mathbb{R}$ such that
		\begin{align*}
			&\|(\omega^{in},j^{in})\|_{H^N}\ll\nu^{\gamma_1}\mu^{\gamma_2}\ \Rightarrow\ \mathrm{stability},\\
			&\|(\omega^{in},j^{in})\|_{H^N}\gg\nu^{\gamma_1}\mu^{\gamma_2}\ \Rightarrow\ \mathrm{possible \ instability}.
		\end{align*}
		
		To quantify the stability threshold is an active area which is initiated first for fluid dynamics, i.e., for the Navier-Stokes (NS) equations. For the NS equations, magnetic resistivity $\mu$ is missing, and hence the above threshold for the stability becomes $\nu^\gamma$ for some $\gamma\geq0$. Bedrossian, Germain, and Masmoudi proved the threshold  $\gamma=\frac{3}{2}$ \cite{BGM17} if the perturbation for velocity belongs to Sobolev space $H^{9/2+}$ in $\mathbb{T}\times\mathbb{R}\times\mathbb{T}$. Then $\gamma=1$ in Sobolev space $H^2$ had been confirmed by Wei and Zhang \cite{WZ21}. A lower $\gamma=\frac{1}{2}$ \cite{BVW18} was obtained in  $H^2$ for the 2D flow near Couette, which had been improved to $\gamma=\frac{1}{3}$ \cite{MZ19,WZ23}. In \cite{BGM20,BGM22,BMV16,LMZ22}, the authors proved that $\gamma=1$ on $\mathbb{T}\times\mathbb{R}\times\mathbb{T}$ and $\gamma=0$ on $\mathbb{T}\times\mathbb{R}$ for the perturbation in Gevrey class. For the threshold problem in a bounded domain, Chen, Wei, and Zhang \cite{QWZ20} proved $\gamma=1$ in a finite channel $\mathbb{T}\times[-1,1]\times\mathbb{T}$ with non-slip boundary condition and slip boundary condition in $H^2$. The result $\gamma=\frac{1}{2}$ held in a finite channel $\mathbb{T}\times[-1,1]$ with non-slip boundary condition\cite{CLWZ20}. 
		For slip boundary condition in $\mathbb{T}\times[-1,1]$, the authors in \cite{CLWZ20,BHIW23} showed $\gamma=\frac{1}{2}$, which has been improved to $\gamma=\frac{1}{3}$ \cite{WZ23b}. Very recently, the stability on $\mathbb{T}\times[-1,1]$ in the Gevrey class had been proved in \cite{BHIW24a,BHIW24b}. For stability results of general shear flow, we refer to \cite{LZ23,LWZ20,WZZ20,CMEW20,DL20,D23,C23,AB24,ADM21,OK80}.
		
		The problem becomes more complicated for MHD equations, due to the presence of magnetic fields. On the one hand, the magnetic field weakens the mixing mechanism of Couette flow. Especially if viscosity is larger than resistivity, inviscid damping gets counteracted by algebraic growth for specific time regimes \cite{K24,KZ23}. On the other hand, a strong magnetic field has a stabilizing effect \cite{GBM09,L20,ZZZ21}.  Liss \cite{L20} proved the first stability threshold $\gamma=1$ in Sobolev spaces for $\nu=\mu>0$ in $\mathbb{T}\times\mathbb{R}\times\mathbb{T}$. Recently, Rao, Zhang, and Zi \cite{RZZ23} extended this result to the case $\mu\neq\nu$. For 2D case in $\mathbb{T}\times\mathbb{R}$, the nonlinear stability in Gevrey-$\frac{1}{2}$- space when $\nu=0,\mu>0$ was obtained by Zhao and Zi \cite{ZZ23}. Chen and Zi \cite{ZZ23b} showed that $\gamma=\frac{5}{6}+$ for shear flow close to Couette in Sobolev space when $\nu=\mu$. Dolce \cite{D24} proved $(\gamma_1,\gamma_2)=(\frac{2}{3},0)$ in a more general regime $0<\mu^3\lesssim\nu\leq\mu$ where the fluid effects dominate. In the case $\mu\leq\nu$, Knobel \cite{K24} showed the stability if $\nu^3\lesssim\mu\leq\nu$ and instability with inflation of size $\nu\mu^{-\frac{1}{3}}$ if $\mu\lesssim\nu^3$. This leaves the case $0<\nu\lesssim\mu^3\leq1$ open, which is essentially different from the symmetric counterpart region $\mu\lesssim\nu^3$.
		
		The purpose of this paper is two-fold. First, we obtain the threshold $\nu^{\frac{1}{2}}\mu^{\frac{1}{3}}$ in regime $0<\nu\leq\mu^3\leq1$ (see Theorem \ref{coro2}). Our stability estimates \eqref{11} is consistent with the linear results when $\nu=0,\mu>0$ in \cite{ZZ23}. Second, by further exploring the structure of the equations, we improve the existing results in \cite{D24,ZZ23b,K24} when $0<\mu^3\leq\nu\leq1$ (see Theorem \ref{main}, \ref{coro}). Our results show that fluid and magnetic field have a subtle interaction rather than a simple superposition of each other. We next state the main theorems.
		\begin{theorem}\label{coro2}
			Assume $0<\nu\leq\mu^{3}\leq1$, $|\beta|>1/2$, and $N\geq4$. Let $(\omega^{in},j^{in})$ be the initial datum of \eqref{mhd}. There exist constant $0<\delta_0<1$ and $\epsilon_0=\epsilon_0(N,\beta,\nu,\mu)>0$ such that if $(\omega^{in},j^{in})$ satisfies
			\begin{equation*}
				\lVert(\omega^{in},j^{in})\rVert_{H^N}=\epsilon\leq\epsilon_0\nu^{\frac{1}{2}}\mu^{\frac{1}{3}},
			\end{equation*}
			then the following stability estimates hold:
			\begin{subequations}\label{11}
				\begin{align}
					&\|(u,b)(t,x+yt,y)\|_{H^{N-1}}\lesssim\epsilon,\\
					&\langle t
					\rangle\lVert(u^1_{\neq},b^1_{\neq})(t)\rVert_{L^2}+\langle t\rangle^2\lVert(u^2_{\neq},b^2_{\neq})(t)\rVert_{L^2}\lesssim\min\{\mu^{-\frac{1}{3}},\langle t\rangle\}\epsilon e^{-\delta_0\nu^{\frac{1}{3}}t},\\
					&\lVert(\omega_{\neq},j_{\neq})(t,x+yt,y)\rVert_{H^N}\lesssim\min\{\mu^{-\frac{1}{3}},\langle t\rangle\}\epsilon e^{-\delta_0\nu^{\frac{1}{3}}t}.\label{coro23}
				\end{align}
			\end{subequations}
		\end{theorem}
		\begin{remark}
			 The above results can be extended to $0<\nu\leq\mu\leq1$ by exactly the same argument.
		\end{remark}
		\begin{theorem}\label{main}
		Assume $0<\mu^3\leq\nu\leq\mu^{\frac{1}{3}}\leq1$, $|\beta|>1/2$, and $N\geq4$. Let $(\omega^{in},j^{in})$ be the initial datum of \eqref{mhd}. There exist constant $0<\delta_0<1$ and $\epsilon_0=\epsilon_0(N,\beta,\nu,\mu)>0$ such that if $(\omega^{in},j^{in})$ satisfies
		\begin{equation*}
			\lVert(\omega^{in},j^{in})\rVert_{H^N}=\epsilon\leq\epsilon_0\min\{\nu^{\frac{1}{2}},\mu^{\frac{1}{2}}\},
		\end{equation*}
		then the following stability estimates hold:
		\begin{align}
			&\|(u,b)(t,x+yt,y)\|_{H^N}\lesssim\epsilon,\\
				 &\lVert(u^1_{\neq},b^1_{\neq})(t)\rVert_{L^2}+\langle t\rangle\lVert(u^2_{\neq},b^2_{\neq})(t)\rVert_{L^2}\lesssim\epsilon e^{-\delta_0\min\{\nu^{\frac{1}{3}},\mu^{\frac{1}{3}}\}t},\label{indp}\\
			 &\lVert(\omega_{\neq},j_{\neq})(t,x+yt,y)\rVert_{H^N}\lesssim\epsilon\langle t\rangle e^{-\delta_0\min\{\nu^{\frac{1}{3}},\mu^{\frac{1}{3}}\}t}.
		\end{align}
	\end{theorem}
	\begin{remark}
		Compared to the Navier-Stokes equations, the inviscid damping and enhanced dissipation have a transient amplification of order $\max\{\nu^{-\frac{1}{3}},\mu^{-\frac{1}{3}}\}$ due to the background magnetic field. We can not expect the Lyapunov stability of $(\omega,j)$, i.e., $\|(\omega,j)(t,x+yt,y)\|_{H^N}\lesssim\epsilon$. In fact, the zero mode $(\omega_0,j_0)$ is order $\epsilon\langle t\rangle$ from \cite{D24}. 
	\end{remark}
	\begin{theorem}\label{coro}
		Assume $0<\mu^{\frac{1}{3}}\leq\nu\leq1$, $|\beta|>1/2$, and $N\geq4$. Let $(\omega^{in},j^{in})$ be the initial datum of \eqref{mhd}. There exist constant $0<\delta_0<1$ and $\epsilon_0=\epsilon_0(N,\beta,\nu,\mu)>0$ such that if $(\omega^{in},j^{in})$ satisfies
		\begin{equation}\label{init}
			\lVert(\omega^{in},j^{in})\rVert_{H^N}=\epsilon\leq\epsilon_0\mu^{\frac{1}{2}}(\nu\mu^{-\frac{1}{3}})^{-1}=\epsilon_0\nu^{-1}\mu^{\frac{5}{6}},
		\end{equation}
		then the following stability estimates hold:
		\begin{subequations}\label{corol}
			\begin{align}
				&\|(u,b)(t,x+yt,y)\|_{H^N}\lesssim\nu\mu^{-\frac{1}{3}}\epsilon,\\
				&\lVert(u^1_{\neq},b^1_{\neq})(t)\rVert_{L^2}+\langle t\rangle\lVert(u^2_{\neq},b^2_{\neq})(t)\rVert_{L^2}\lesssim\nu\mu^{-\frac{1}{3}}\epsilon e^{-\delta_0\mu^{\frac{1}{3}}t},\\
				&\lVert(\omega_{\neq},j_{\neq})(t,x+yt,y)\rVert_{H^N}\lesssim\nu\mu^{-\frac{1}{3}}\langle t\rangle \epsilon e^{-\delta_0\mu^{\frac{1}{3}}t}.
			\end{align}
		\end{subequations}
	\end{theorem}
	\noindent\textbf{Idea of the proof}: As in \cite{D24,K24}, it is natural to use symmetric variables
	\begin{equation*}
		z=\sqrt{\partial_{xx}\Delta^{-1}}\omega,\quad q=\sqrt{\partial_{xx}\Delta^{-1}}j,
	\end{equation*}
	when $0<\nu\leq\mu^3\leq1$. In fact, by a similar argument as \cite{K24}, we get
	\begin{equation*}
		\lVert(\omega_{\neq},j_{\neq})(t,x+yt,y)\rVert_{H^N}\lesssim\langle t\rangle\lVert(z,q)(t,x+yt,y)\rVert_{H^N}\lesssim\mu\nu^{-\frac{1}{3}}\langle t\rangle\epsilon e^{-\delta_0\nu^{\frac{1}{3}}t}.
	\end{equation*}
	However, considering the linear estimates in \cite{ZZ23}, we expect that $(\omega_{\neq},j_{\neq})$ has an amplification no larger than $\langle t\rangle$, which is much smaller than the above prediction. Therefore, we have to give up the symmetric variables, and treat the system of $(\omega_{\neq},j_{\neq})$ directly.
	
	For the nonlinear problem, the most challenging part is of the form 
	\begin{equation*}
		\langle M(F\cdot\nabla_LG),MH\rangle,
	\end{equation*}
	where $M$ is the multiplier defined in \eqref{m} and $FGH\in\{U\Omega \Omega, UJJ,BJ\Omega,B\Omega J\}$, where we use capital letters $U,\Omega,B,J$ to represent quantities in the new coordinates. We need to treat $F^1\partial_{X}G$ and $F^2\partial^L_YG$ separately. For $F^1\partial_{X}G$, we use commutator $
		\langle[M,F^1\partial_X]G,MH\rangle$
	in the specific frequency region which is effective to facilitate cancellations and derivative distribution. These estimates need the strong enhanced dissipation $\nu^{\frac{1}{6}}\||\partial_X|^{\frac{1}{3}}M(\Omega_{\neq},J_{\neq})\|_{L^2L^2}$ to reduce derivative losses as much as possible (see section \ref{estt} for details). Next turning to $F^2\partial_YG_0$, due to the nonlinear interaction, $(\Omega_0,J_0)$ undergoes a linear growth $\langle t\rangle$, and this growth is crucially counter-balanced by the inviscid damping of $(U^2,B^2)$. These ideas can also apply to nonlinear estimates on the system of $(z,q)$ to improve the threshold when $0<\mu^3\leq\nu\leq1$. In the regime $\mu^{\frac{1}{3}}\leq\nu$, since $(\omega_{\neq},j_{\neq})$ has a growth rate larger than $\langle t\rangle$, the estimates highly depend on the good structure of the nonlinearity, especially for the term $\langle M(\partial_{XX}\Delta_L^{-1})^{\frac{1}{2}}(B\cdot\nabla_LJ),MZ\rangle$ (see section \ref{bsset2} for details).
	
	\begin{remark}
			Considering the three results  above, we give two comments. First, compared to $(\omega_{\neq},j_{\neq})$, due to the  stretching term $-\partial_{xy}{\Delta}^{-1}z$, the symmetric quantity $(z,q)$ grows in an additional time interval $[\frac{\eta}{k}-\nu^{-\frac{1}{3}},\frac{\eta}{k}]$, where $k,\eta$ are denoted the frequency of $x,y$ respectively. If $\mu^3\leq\nu\leq\mu^{\frac{1}{3}}$, this growth is controlled by enhanced dissipation and diffusion. If $\mu^{\frac{1}{3}}\leq\nu$, it leads to an amplification of order $\nu\mu^{-\frac{1}{3}}$, as shown in \eqref{corol}, which is consistent with the results when $\mu=0,\nu>0$ in \cite{KZ23}. Second, both Theorem \ref{coro2} and Theorem \ref{main} hold on the regime $\mu^3\leq\nu\leq\mu$, but neither is better than the other. In addition, theorem \ref{coro2} cannot be extended to the case $\nu>\mu$. 
	\end{remark}

	\section{Preliminary}
	\subsection{Notation}
	Throughout this paper, for $r,s \in\mathbb{R}$, we define
	\begin{equation*}
		|r,  s |:=|r|+|s|,\quad\langle r \rangle:=\sqrt{1+r^2}.
	\end{equation*}
	We use the notation $r\lesssim s$ to express $r\leq Cs$ for some constant $C>0$ independent of the parameters of interest such as $\nu,\mu$. The Fourier transform of a function $f$ is denoted by
	\begin{equation*}
		\mathcal{F}(f)(k,\eta)=\hat{f}_{k,\eta}=\hat{f}(k,\eta)=\iint_{\mathbb{T}\times\mathbb{R}}e^{-i(kx+\eta y)}f(x,y)\mathrm{d}x\mathrm{d}y.
	\end{equation*}
	For a Fourier multiplier $A$, we define $Af=\mathcal{F}^{-1}(A(k,\eta)\hat{f}(k,\eta))$. 
	We denote
	\begin{equation*}
		f_0(y):=\int_{\mathbb{T}}f(x,y)\mathrm{d}x, \quad f_{\neq}(x,y):=f(x,y)-f_0(y).
	\end{equation*}
	
	\subsection{Reformulation of the equations}
	\textbf{Change of independent variables}: First, we change the coordinates to mod out by the fast mixing of the Couette flow:
	\begin{equation*}
		X=x-yt,\quad Y=y,
	\end{equation*}
	and denote
	\begin{gather*}
		\Omega(t,X,Y):=\omega(t,x,y),\quad J(t,X,Y):=j(t,x,y),\\
		\nabla_L=(\partial_X^L,\partial_Y^L):=(\partial_X,\partial_Y-t\partial_X),\\
		\Delta_L:=\partial_{XX}+(\partial_Y-t\partial_X)^2.
	\end{gather*}
	Then the system satisfied by $(\Omega,J)$ reads
		\begin{equation*}
		\left\{\begin{array}{l}
			\partial_t\Omega-\beta\partial_XJ-\nu\Delta_L\Omega=\mathrm{NL}^\Omega,\\
			\partial_tJ-\beta\partial_X\Omega-\mu\Delta_L J+2\partial_{XY}^L\Phi=\mathrm{NL}^J,\\
			\Omega|_{t=0}=\Omega^{in},\quad J|_{t=0}=J^{in},\\
			U=\nabla_L^\bot\Psi,\quad B=\nabla_L^\bot\Phi,\\
			\Psi=\Delta_L^{-1}\Omega,\quad \Phi=\Delta_{L}^{-1}J,\\
		\end{array}\right.
	\end{equation*}
	where 
	\begin{equation}\label{nlterm}
			\begin{aligned}
			\mathrm{NL}^{\Omega}:&=-U\cdot \nabla_L\Omega+B\cdot\nabla_L J,\\
			\mathrm{NL}^J:&=-U\cdot \nabla_L J+B\cdot\nabla_L\Omega+2\partial_{XY}^L\Phi(\Omega-2\partial_{XX}\Psi)-2\partial_{XY}^L\Psi(J-2\partial_{XX}\Phi).
		\end{aligned}
	\end{equation}
	\textbf{Change of dependent variables}: When $\nu\leq\mu^3$, we treat the above system of $(\Omega,J)$ directly. While for $\mu^3\leq\nu$, following \cite{D24}, we introduce good unknowns to symmetrize the system. Let
	\begin{equation}\label{defgamma}
		\Gamma=\sqrt{\frac{\partial_{XX}}{\Delta_L}}.
	\end{equation} 
	We define
	\begin{equation*}
		Z:=\Gamma\Omega,\quad Q:=\Gamma J.
	\end{equation*}
	Then $(Z,Q)$ satisfies the equations
	\begin{equation}\label{ZQeq}
		\left\{
		\begin{array}{l}
			\displaystyle\partial_tZ-\beta \partial_XQ-\nu \Delta_L Z-\frac{\partial_{XY}^L}{\Delta_L}Z=\Gamma\mathrm{NL}^\Omega,\\
			\displaystyle\partial_tQ-\beta \partial_XZ-\mu \Delta_LQ+\frac{\partial_{XY}^L}{\Delta_L}Q=\Gamma\mathrm{NL}^J,\\
		\end{array}\right.
	\end{equation}
	where $\mathrm{NL}^{\Omega},\mathrm{NL}^J$ are defined as \eqref{nlterm}. The associated energy functional is given by 
	\begin{equation*}
		\mathrm{E}(t)=\frac{1}{2}\left(\lVert Z\rVert^2_{L^2}+\left\lVert Q\right\rVert^2_{L^2}+\frac{1}{\beta}Re\left\langle\frac{\partial_Y^L}{ \Delta_L}Z,Q\right\rangle \right).
	\end{equation*}

	\subsection{The Fourier multipliers} 
	 We introduce the following multipliers to capture the stability mechanisms of the above systems. The first multiplier $M^1$ is standard, which helps us to obtain inviscid damping of the MHD equations. Let $$\lambda:=\min\{\nu,\mu\},$$
	 then $M^1$ satisfies the following equations
	\begin{equation}\label{m1}
		\left\{\begin{array}{l}
			\displaystyle\frac{\partial_tM^1_k}{M^1_k}=-\frac{C_1}{1+(\eta/k-t)^2},\ \mathrm{for}\ 0<|k|\leq C_2\lambda^{-1/2},\\
			M^1_k(t,\eta)=1,\ \mathrm{for}\ |k|\in(0,C_2\lambda^{-1/2}]^c,\\
			M^1_k(0,\eta)=1,
		\end{array} 
		\right.
	\end{equation}
	where $C_1>1000C_2$ and $C_2>3000$ are sufficiently large.	Motivated by \cite{DWZ21}, we construct the multiplier $M^2$ to capture a slightly stronger enhanced dissipation $e^{-(\lambda k^2)^{1/3}t}$, compared with the classical enhanced dissipation $e^{-\lambda^{1/3}t}$. It is determined by the equations
	\begin{equation}\label{m2}
		\left\{\begin{array}{l}
			\displaystyle\frac{\partial_tM^{2}_k}{M^{2}_k}=-\frac{ 2(C_2)^2(\lambda k^2)^{\frac{1}{3}}}{1+(\lambda k^2)^{\frac{2}{3}}(\eta/k-t)^2},\ \mathrm{for}\ k\neq0,\\
			M^2_0(t,\eta)=M^2_k(0,\eta)=1.
		\end{array}\right.
	\end{equation}
	$M^3$ is given in \cite{D24}, defined as
	\begin{equation}\label{m3}
		\left\{\begin{array}{l}
			\displaystyle\frac{\partial_tM^3_k}{M^3_k}=-\frac{C_3}{(1+(\eta/k-t)^2)^{\frac{3}{2}}},\ \mathrm{for}\ 0<|k|\leq C_2\lambda^{-1/2},\\
			M^3_k(t,\eta)=1,\ \mathrm{for}\ |k|\in(0,C_2\lambda^{-1/2}]^c,\\
			M^3_k(0,\eta)=1,
		\end{array} 
		\right.
	\end{equation} 
	with $\displaystyle C_3>\frac{C_1}{|\beta|-1/2}$.
	This multiplier could help us to get the sharp assumption $|\beta|>\frac{1}{2}$.
	
	If $0<\mu^{\frac{1}{3}}\leq\nu\leq1$, we need two extra multipliers $M^4$ and $M^5$ from \cite{K24}. They are designed to control the growth caused by the stretching terms $-\frac{\partial^L_{XY}}{\Delta_L}Z$ with $\frac{\partial^L_{XY}}{\Delta_L}Q$ in time intervals $[\frac{\eta}{k}-4\nu^{-1},\frac{\eta}{k}+4\nu^{-1}]$ and $[\frac{\eta}{k}+4\nu^{-1},\frac{\eta}{k}+4\mu^{-\frac{1}{3}}]$ respectively, given by
	\begin{equation}\label{m4}
		\left\{\begin{array}{l}
			\displaystyle\frac{\partial_tM^4_k}{M^4_k}=-\frac{C_4\nu}{1+\nu^2(\eta/k-t)^2},\ \mathrm{for}\ 0<|k|\leq C_2\mu^{-1/2},\\
			M^4_k(t,\eta)=1,\ \mathrm{for}\ |k|\in(0,C_2\mu^{-1/2}]^c,\\
			M^4_k(0,\eta)=1,
		\end{array} 
		\right.
	\end{equation} 
	and
	\begin{equation}\label{m5}
		\left\{\begin{array}{l}
			\displaystyle\frac{\partial_tM^5_k}{M^5_k}=-\frac{t-\eta/k}{1+(\eta/k-t)^2},\ \mathrm{for}\ 4\nu^{-1}\leq t-\frac{\eta}{k}\leq4\mu^{-\frac{1}{3}} \ \mathrm{and}\ 0<|k|\leq C_2\mu^{-1/2}\\
			M^5_k(t,\eta)=1,\ \mathrm{otherwise},\\
			M^5_k(0,\eta)=1.
		\end{array} 
		\right.
	\end{equation} 
	
	If $0<\nu\leq\mu^3\leq1$, we use $M^6$ from \cite{ZZ23,L20} to control the stretching term $2\partial^L_{XY}\Phi$ in time interval $[\frac{\eta}{k},\frac{\eta}{k}+4\mu^{-\frac{1}{3}}]$. The multiplier $M^6$ satisfies 
	\begin{equation}\label{m6}
		\left\{\begin{array}{l}
			\displaystyle\frac{\partial_tM^6_k}{M^6_k}=-\frac{t-\eta/k}{1+(\eta/k-t)^2},\ \mathrm{for}\ 0\leq t-\frac{\eta}{k}\leq4\mu^{-\frac{1}{3}} \ \mathrm{and}\ 0<|k|\leq C_2\nu^{-1/2}\\
			M^6_k(t,\eta)=1,\ \mathrm{otherwise},\\
			M^6_k(0,\eta)=1.
		\end{array} 
		\right.
	\end{equation} 
	The constants $C_1,C_2,$ and $C_3$ above are to be chosen in the proof of Proposition \ref{linearest}. Then we have the following properties for these multipliers. 
	\begin{lemma}
		It holds that
			\begin{subequations} \label{propm}
				\begin{align}
				&M^i_k\approx1,\ {\rm for\ all}\ t,k,\eta,\ \mathrm{and}\  i=1,2,3,4,\label{211}\\
				&\lambda (k^2+(\eta-kt)^2)-\frac{\partial_tM^2_k}{M^2_k}\geq C_2\lambda^{\frac{1}{3}}|k|^{\frac{2}{3}},\ {\rm for\ all}\ t,k\neq0,\eta,\label{213}\\
				&\frac{\partial_{k}M^2_k}{M^2_k}\lesssim \lambda^{\frac{1}{3}}|k|^{-\frac{4}{3}}|\eta|+|k|^{-1},\ {\rm for\ all}\ |k|>C_2\lambda^{-1/2},\label{215}\\
				&\frac{\partial_{\eta}M^i_k}{M^i_k}\lesssim \lambda^{\frac{1}{3}}|k|^{-\frac{1}{3}}+|k|^{-1},\ {\rm for\ all}\ k\neq0,\ \mathrm{and}\ i=1,2,3,4,5,6\label{214},\\
				&\left|\frac{1}{M^5_k}\right|\lesssim1+\min\left\{\nu\left\langle t-\frac{\eta}{k}\right\rangle,\alpha\right\},\ {\rm for\ all}\ k\neq0,\ \mathrm{and}\ \alpha=\nu\mu^{-\frac{1}{3}}.\label{216}\\
				&\max\left\{\Gamma_k,\frac{1}{\langle t\rangle},\mu^{\frac{1}{3}}\right\}\leq|M^6_k|.\label{217}
				\end{align}
			\end{subequations}
	\end{lemma}
	\begin{proof}
		The proof is elementary and can be found in previous relevant papers. However, for the sake of completeness, we give a brief sketch. For $M^1$, a direct computation shows that
		\begin{equation*}
			M^1(t,k,\eta)=\exp\left(-\int_{0}^{t}\frac{C_1}{1+(\eta/k-\tau)^2}\mathrm{d}\tau\right),
		\end{equation*}
		which implies \eqref{211} for $i=1$. Taking the derivative with respect to $\eta$ gives
		\begin{equation*}
			\left|\frac{\partial_{\eta}M^1_k}{M^1_k}\right|\leq\left|\frac{1}{k}\int_{0}^{t}\frac{C_1(\eta/k-\tau)}{(1+(\eta/k-\tau)^2)^2}\mathrm{d}\tau\right|\leq\frac{1}{|k|}\left(\int_{0}^{t}\frac{C_1}{1+(\eta/k-\tau)^2}\mathrm{d}\tau\right),
		\end{equation*}
		which proves that \eqref{214} holds for $M^1$. Next turning to $M^2$, it holds that
		\begin{equation*}
			M^2(t,k,\eta)=\exp\left(-\int_0^t\frac{2(C_2)^2(\lambda k^2)^{\frac{1}{3}}}{1+(\lambda k^2)^{\frac{2}{3}}(\eta/k-\tau)^2}\mathrm{d}\tau\right).
		\end{equation*}
		Furthermore, we have
		\begin{equation*}
				\left|\frac{\partial_\eta M^{2}_k}{M^{2}_k}\right|\lesssim\left|\left(\int_0^t\frac{(\lambda k^2)(\eta/k-\tau)k^{-1}}{(1+(\lambda k^2)^{\frac{2}{3}}(\eta/k-\tau)^2)^2}\mathrm{d}\tau\right)\right|\lesssim\lambda^{\frac{1}{3}}|k|^{-\frac{1}{3}}\left(\int_0^t\frac{(\lambda k^2)^{\frac{1}{3}}}{1+(\lambda k^2)^{\frac{2}{3}}(\eta/k-\tau)^2}\mathrm{d}\tau\right),
		\end{equation*}
		and
		\begin{align*}
			\left|\frac{\partial_k M^{2}_k}{M^{2}_k}\right|=&\left|\int_0^t\partial_k\left(\frac{2(C_2)^2\lambda^{\frac{1}{3}}k^{-\frac{2}{3}}}{k^{-\frac{4}{3}}+\lambda ^{\frac{2}{3}}(\eta/k-\tau)^2}\right)\mathrm{d}\tau\right|\lesssim\int_0^t\left|\left(\frac{\lambda^{\frac{1}{3}}k^{-\frac{5}{3}}}{k^{-\frac{4}{3}}+\lambda ^{\frac{2}{3}}(\eta/k-\tau)^2}\right)\right|\\
			&+\left|\left(\frac{\lambda^{\frac{1}{3}}k^{-3}}{(k^{-\frac{4}{3}}+\lambda ^{\frac{2}{3}}(\eta/k-\tau)^2)^2}\right)\right|+\left|\left(\frac{\lambda k^{-\frac{2}{3}}(\eta/k-\tau)\eta k^{-2}}{(k^{-\frac{4}{3}}+\lambda ^{\frac{2}{3}}(\eta/k-\tau)^2)^2}\right)\right|\mathrm{d}\tau\\
			\lesssim&|k|^{-1}+\lambda^{\frac{1}{3}}|k|^{-\frac{4}{3}}|\eta|.
		\end{align*}
		Hence we deduce that \eqref{211}, \eqref{214}, and \eqref{215} hold for $M^2$.
		For \eqref{213}, it can be easily checked by considering $|\eta/k-t|\leq(\lambda k^2)^{-\frac{1}{3}}|C_2|^{\frac{1}{2}}$ or $|\eta/k-t|\geq(\lambda k^2)^{-\frac{1}{3}}|C_2|^{\frac{1}{2}}$ separately. The treatment of $M^3,M^4$ is similar to $M^1$ and hence is omitted. Finally, $M^5$ and $M^6$ are given by the exact formulas:
		\begin{enumerate}
			\item if $k=0$ or $|k|>C_2\mu^{-\frac{1}{2}}:\ M^5(t,k,\eta)=1$;
			\item if $k\neq0$, $\frac{\eta}{k}<-4\mu^{-\frac{1}{3}},$ and $|k|\leq C_2\mu^{-\frac{1}{2}}:\ M^5(t,k,\eta)=1$;
			\item if $k\neq0,\ -4\mu^{-\frac{1}{3}}\leq\frac{\eta}{k}\leq-4\nu^{-1}$, and $|k|\leq C_2\mu^{-\frac{1}{2}}$:
			\begin{equation*}
				M^5(t,k,\eta)=\left\{\begin{array}{ll}
					\left(\frac{k^2+\eta^2}{k^2+(\eta-kt)^2}\right)^{\frac{1}{2}},\ & \textrm{if}\ t-\frac{\eta}{k}\leq4\mu^{-\frac{1}{3}},\\
					\left(\frac{k^2+\eta^2}{k^2+16\mu^{-\frac{2}{3}}k^2}\right)^{\frac{1}{2}},\ &\textrm{if}\ t-\frac{\eta}{k}>4\mu^{-\frac{1}{3}},
				\end{array}\right.
			\end{equation*} 
			\item if $k\neq0,\ \frac{\eta}{k}>-4\nu^{-1}$, and $|k|\leq C_2\mu^{-\frac{1}{2}}$:
			\begin{equation*}
				M^5(t,k,\eta)=\left\{\begin{array}{ll}
					1,&\textrm{if}\ t-\frac{\eta}{k}\leq4\nu^{-1},\\
					\left(\frac{k^2+16\nu^{-2}k^2}{k^2+(\eta-kt)^2}\right)^{\frac{1}{2}},\ &\textrm{if}\ 4\nu^{-1}<t-\frac{\eta}{k}\leq4\mu^{-\frac{1}{3}},\\
					\left(\frac{k^2+16\nu^{-2}k^2}{k^2+16\mu^{-\frac{2}{3}}k^2}\right)^{\frac{1}{2}},\ &\textrm{if}\ t-\frac{\eta}{k}>4\mu^{-\frac{1}{3}},
				\end{array}\right.
			\end{equation*}
		\end{enumerate}
		and 
		\begin{enumerate}
			\item if $k=0$ or $|k|>C_2\mu^{-\frac{1}{2}}:\ M^6(t,k,\eta)=1$;
			\item if $k\neq0$, $\frac{\eta}{k}<-4\mu^{-\frac{1}{3}},$ and $|k|\leq C_2\nu^{-\frac{1}{2}}:\ M^6(t,k,\eta)=1$;
			\item if $k\neq0,\ -4\mu^{-\frac{1}{3}}\leq\frac{\eta}{k}\leq0$, and $|k|\leq C_2\nu^{-\frac{1}{2}}$:
			\begin{equation*}
				M^6(t,k,\eta)=\left\{\begin{array}{ll}
					\left(\frac{k^2+\eta^2}{k^2+(\eta-kt)^2}\right)^{\frac{1}{2}},\ &\textrm{if}\ t-\frac{\eta}{k}\leq4\mu^{-\frac{1}{3}},\\
					\left(\frac{k^2+\eta^2}{k^2+16\mu^{-\frac{2}{3}}k^2}\right)^{\frac{1}{2}},\ &\textrm{if}\ t-\frac{\eta}{k}>4\mu^{-\frac{1}{3}},
				\end{array}\right.
			\end{equation*} 
			\item if $k\neq0,\ \frac{\eta}{k}>0$, and $|k|\leq C_2\nu^{-\frac{1}{2}}$:
			\begin{equation*}
				M^6(t,k,\eta)=\left\{\begin{array}{ll}
					1,&\textrm{if}\ t-\frac{\eta}{k}\leq0,\\
					\left(\frac{k^2}{k^2+(\eta-kt)^2}\right)^{\frac{1}{2}},\ &\textrm{if}\ 0<t-\frac{\eta}{k}\leq4\mu^{-\frac{1}{3}},\\
					\left(\frac{k^2}{k^2+16\mu^{-\frac{2}{3}}k^2}\right)^{\frac{1}{2}},\ &\textrm{if}\ t-\frac{\eta}{k}>4\mu^{-\frac{1}{3}}.
				\end{array}\right.
			\end{equation*}
		\end{enumerate}
		Thus by direct computation, \eqref{214}-\eqref{217} hold for $M^5,M^6$, completing the proof.
	\end{proof}
	
	Now we define the main multiplier
	\begin{equation}\label{m}
		M:=\left\{\begin{array}{ll}
			A\langle|\nabla|\rangle^NM^1M^2 M^3&\mathrm{if}\ 0<\mu^{3}\leq\nu\leq\mu^{\frac{1}{3}}\leq1,\\
			A\langle|\nabla|\rangle^NM^1M^2 M^3M^4M^5&\mathrm{if}\ 0<\mu^{\frac{1}{3}}\leq\nu\leq1,\\
			A\langle|\nabla|\rangle^NM^1M^2 M^3M^6&\mathrm{if}\ 0<\nu\leq\mu^{3}\leq1,
		\end{array}\right.
	\end{equation}
	where $A$ is the enhanced dissipation time-decay multiplier given by
	\begin{equation*}
		\left\{\begin{array}{l}
			A_k(t,\eta)=e^{\delta_0\lambda^{\frac{1}{3}}t},\ \mathrm{for}\ k\neq0,\\
			A_0(t,\eta)=1,\ k=0,
		\end{array}\right.
	\end{equation*}
	with $\displaystyle \delta_0\leq\frac{1}{100|\beta|}$.

	\section{Linear Problem}
	The treatment of the linear problem was essentially already obtained in \cite{D24,K24,ZZ23} with slight modification. For the sake of completeness, we list the linear results in this section and give a sketch of the proof in the appendix. When $0<\nu\leq\mu^3\leq1$, we consider the linearized equations of $(\Omega,J)$ 
	\begin{equation*}
		\left\{\begin{array}{l}
			\partial_t\Omega-\beta\partial_XJ-\nu\Delta_L\Omega=0,\\
			\partial_tJ-\beta\partial_X\Omega-\mu\Delta_L J+2\partial_{XY}^L\Phi=0.\\
		\end{array}\right.
	\end{equation*} 
	Note that $\Omega_0,J_0$ solve the heat equation and the estimates are standard, for which we omit further details. For the nonzero-modes, we define the associated energy functional as
	\begin{equation*}
		E(t):=\frac{1}{2}\left(\lVert M\Omega_{\neq}\rVert^2_{L^2}+\left\lVert MJ_{\neq}\right\rVert^2_{L^2}+\frac{2}{\beta}Re\left\langle\frac{\partial_Y^L\chi}{ \Delta_L}M\Omega_{\neq},MJ_{\neq}\right\rangle\right),
	\end{equation*}
	where the multiplier $\chi$ is such that
		\begin{equation}\label{chii}
		\chi(t,k,\eta)=\left\{\begin{array}{ll}
			1,\ &\mathrm{if}\ 0<t-\frac{\eta}{k}\leq4\mu^{-\frac{1}{3}}\ \mathrm{and}\ k\neq0,\\
			0,\ &\mathrm{otherwise}.\
		\end{array}\right.
		\end{equation}
	Since
		\begin{equation*}
			\left|\frac{1}{\beta}\left\langle\frac{\partial_Y^L\chi}{ \Delta_L}M\Omega_{\neq},MJ_{\neq}\right\rangle\right|\leq\frac{1}{|\beta|}\|M\Omega_{\neq}\|_{L^2}\|MJ_{\neq}\|_{L^2},
		\end{equation*}
		we get that $E(t)$ is coercive, with
		\begin{equation*}
			\frac{1}{2}\left(1-\frac{1}{2|\beta|}\right)\left(\|M\Omega_{\neq}\|^2_{L^2}+\|MJ_{\neq}\|^2_{L^2}\right)\leq E(t)\leq\frac{1}{2}\left(1+\frac{1}{2|\beta|}\right)\left(\|M\Omega_{\neq}\|^2_{L^2}+\|MJ_{\neq}\|^2_{L^2}\right),
		\end{equation*}
		for $|\beta|>1/2$. Then we have the following linear result.
	\begin{proposition}\label{linearest3}
		Let $0<\nu\leq\mu^{3}\leq1$, $|\beta|>1/2$, and $N\geq4$. It holds that
		\begin{equation*}
			E(t)+\frac{1}{100|\beta|}\int_{0}^{t}D(\tau)+CK(\tau)\mathrm{d}\tau\leq E(0),
		\end{equation*}
		where
		\begin{equation*}
			\begin{aligned}
				&E(t)=\frac{1}{2}\left(\lVert M\Omega_{\neq}\rVert^2_{L^2}+\left\lVert MJ_{\neq}\right\rVert^2_{L^2}+\frac{2}{\beta}Re\left\langle\frac{\partial_Y^L\chi}{ \Delta_L}M\Omega_{\neq},MJ_{\neq}\right\rangle\right),\\
				&D(t)=\nu \lVert \nabla_LM\Omega_{\neq}\rVert^2_{L^2}+\mu \lVert \nabla_L MJ_{\neq}\rVert^2_{L^2},\\
				&CK(t)=\sum_{F\in\{\Omega_{\neq},J_{\neq}\}}\left(\left\lVert\Gamma MF\right\rVert^2_{L^2}+\lambda^{\frac{1}{3}}\left\lVert |\partial_X|^{\frac{1}{3}}MF\right\rVert^2_{L^2}\right),
			\end{aligned}
		\end{equation*}
		with $\Gamma,\chi$ given by \eqref{defgamma} and \eqref{chii} respectively. Furthermore, we have
		\begin{equation*}
			\lVert \Omega_{\neq}(t)\rVert_{H^N}+\lVert J_{\neq}(t)\rVert_{H^N}\lesssim \min\{\mu^{-\frac{1}{3}},\langle t\rangle\}e^{-\delta_0\nu^{1/3}t}(\lVert \omega_{\neq}^{in}\rVert_{H^N}+\lVert j_{\neq}^{in}\rVert_{H^N}).
		\end{equation*}
	\end{proposition}

	For $0<\mu^3\leq\nu\leq\mu^{\frac{1}{3}}\leq1$, following \cite{D24}, we consider the linearized problem of \eqref{ZQeq}
	\begin{equation*}
		\left\{
		\begin{array}{l}
			\displaystyle\partial_tZ-\beta \partial_XQ-\nu \Delta_L Z-\frac{\partial_{XY}^L}{\Delta_L}Z=0,\\
			\displaystyle\partial_tQ-\beta \partial_XZ-\mu \Delta_LQ+\frac{\partial_{XY}^L}{\Delta_L}Q=0.\\
		\end{array}\right.
	\end{equation*}
	Then the equations for $(MZ,MQ)$ read
	\begin{equation}\label{mequation}
		\left\{
		\begin{array}{l}
			\displaystyle\partial_t(MZ)-\frac{\partial_tM}{M}MZ-\beta \partial_XMQ-\nu \Delta_L MZ-\frac{\partial_{XY}^L}{\Delta_L}MZ=0,\\
			\displaystyle\partial_t(MQ)-\frac{\partial_tM}{M}MQ-\beta \partial_XMZ-\mu \Delta_LMQ+\frac{\partial_{XY}^L}{\Delta_L}MQ=0.\\
		\end{array}\right.
	\end{equation}
	The associated energy functional is given by
	\begin{equation*}
		\overline{E}(t):=\frac{1}{2}\left(\lVert MZ\rVert^2_{L^2}+\left\lVert MQ\right\rVert^2_{L^2}+\frac{2}{\beta}Re\left\langle\frac{\partial_Y^L\overline{\chi}}{ \Delta_L}MZ,MQ\right\rangle\right),
	\end{equation*}
	where $\overline{\chi}$ satisfies
	\begin{equation}\label{defchi}
		\overline{\chi}(t,k,\eta)=\left\{\begin{array}{ll}
			1,\ &\mathrm{if}\ |t-\frac{\eta}{k}|\leq4\nu^{-1}\ \mathrm{and}\ k\neq0,\\
			0,\ &\mathrm{otherswise}.\
		\end{array}\right.
	\end{equation}
	Then we have the following linear estimates.
	\begin{proposition}\label{linearest}
		Let $0<\mu^3\leq\nu\leq\mu^{\frac{1}{3}}\leq1$, $|\beta|>1/2$, and $N\geq4$. It holds that
		\begin{equation}\label{linest}
			\overline{E}(t)+\frac{1}{100|\beta|}\int_{0}^{t}\overline{D}(\tau)+\overline{CK}(\tau)\mathrm{d}\tau\leq \overline{E}(0),
		\end{equation}
		where
		\begin{equation}\label{defchieng}
			\begin{aligned}
				&\overline{E}(t):=\frac{1}{2}\left(\lVert MZ\rVert^2_{L^2}+\left\lVert MQ\right\rVert^2_{L^2}+\frac{2}{\beta}Re\left\langle\frac{\partial_Y^L\overline{\chi}}{ \Delta_L}MZ,MQ\right\rangle\right),\\
				&\overline{D}(t)=\nu \lVert \nabla_LMZ\rVert^2_{L^2_{X,Y}}+\mu \lVert \nabla_L MQ\rVert^2_{L^2_{X,Y}},\\
				&\overline{CK}(t)=\sum_{F\in\{Z,Q\}}\left(\left\lVert\Gamma MF\right\rVert^2_{L^2_{X,Y}}+\lambda^{\frac{1}{3}}\left\lVert |\partial_X|^{\frac{1}{3}}MF\right\rVert^2_{L^2_{X,Y}}\right),
			\end{aligned}
		\end{equation}
		with $\Gamma,\overline{\chi}$ given by \eqref{defgamma} and \eqref{defchi} respectively. Furthermore, we have
		\begin{align}
			&\lVert Z(t)\rVert_{H^N}+\lVert Q(t)\rVert_{H^N}\lesssim e^{-\delta_0\lambda^{1/3}t}(\lVert Z^{in}\rVert_{H^N}+\lVert Q^{in}\rVert_{H^N}),\label{zbd}\\
			&\lVert \Omega_{\neq}(t)\rVert_{H^N}+\lVert J_{\neq}(t)\rVert_{H^N}\lesssim \langle t\rangle e^{-\delta_0\lambda^{1/3}t}(\lVert \omega_{\neq}^{in}\rVert_{H^N}+\lVert j_{\neq}^{in}\rVert_{H^N}),\label{omgbd}\\
			&\lVert U^1_{\neq}(t)\rVert_{H^N}+\langle t\rangle\lVert U^2(t)\rVert_{H^{N-1}}\lesssim  e^{-\delta_0\lambda^{1/3}t}(\lVert \omega_{\neq}^{in}\rVert_{H^N}+\lVert j_{\neq}^{in}\rVert_{H^N}).\label{ubd}
		\end{align}
	\end{proposition}
	When $0<\mu^{\frac{1}{3}}\leq\nu\leq1$, we need to use $M^4,M^5$ to control the growth caused by $\frac{\partial^L_{XY}}{\Delta_L}Q$ and get the following result.
	\begin{proposition}\label{linearest2}
		Let $0<\mu^{\frac{1}{3}}\leq\nu\leq1$, $|\beta|>1/2$, and $N\geq4$. It holds that
		\begin{equation*}
			E(t)+\frac{1}{100|\beta|}\int_{0}^{t}D(\tau)+CK(\tau)\mathrm{d}\tau\leq E(0),
		\end{equation*}
		with $\overline{E},\overline{D},\overline{CK}$ defined in \eqref{defchieng}. Furthermore, we obtain
			\begin{align*}
			&\lVert Z(t)\rVert_{H^N}+\lVert Q(t)\rVert_{H^N}\lesssim \nu\mu^{-\frac{1}{3}}e^{-\delta_0\mu^{1/3}t}(\lVert Z^{in}\rVert_{H^N}+\lVert Q^{in}\rVert_{H^N}),\\
			&\lVert \Omega_{\neq}(t)\rVert_{H^N}+\lVert J_{\neq}(t)\rVert_{H^N}\lesssim \nu\mu^{-\frac{1}{3}}\langle t\rangle e^{-\delta_0\mu^{1/3}t}(\lVert \omega_{\neq}^{in}\rVert_{H^N}+\lVert j_{\neq}^{in}\rVert_{H^N}).
		\end{align*}
	\end{proposition}
	
	For the proof of Proposition \ref{linearest3}--\ref{linearest2}, see the appendix in Section \ref{append}.
	\section{Nonlinear Estimates}
	In this section, we consider the nonlinear problem. Recalling the linear results, we use the multipliers $\Gamma,M_5,M_6$ to control the growth caused by the stretching term $\partial_{XY}^L\Phi$. However, since these multipliers are not bounded below by any constant, they could cause trouble in the nonlinear analysis. To overcome this difficulty, we use commutator estimates for $\Gamma$ and $M$ to reduce the loss of derivatives (see Lemma \ref{lemcom} for details). Also we need to trade regularity with decay in time among $(U_{\neq},B_{\neq}), (Z,Q)$, and $(\Omega_{\neq},J_{\neq})$ (see more details in Lemma \ref{indap}). First, we introduce the energy functionals.
	\subsection{Energy functionals and the bootstrap argument}
		In the region $0<\nu\leq\mu^3\leq1$, we define the energy functionals of $(\Omega_{\neq},J_{\neq})$ and the zero-mode respectively
		
		\begin{align}
			&E(t)=\frac{1}{2}\left(\lVert M\Omega_{\neq}\rVert^2_{L^2}+\left\lVert MJ_{\neq}\right\rVert^2_{L^2}+\frac{2}{\beta}Re\left\langle\frac{\partial_Y^L\chi}{ \Delta_L}M\Omega_{\neq},MJ_{\neq}\right\rangle\right),\label{nonef4}\\
			&E_{0}(t)=\frac{1}{2}\left(\left\lVert U^1_0\right\rVert^2_{H^N_{Y}}+\left\lVert B^1_0\right\rVert^2_{H^N_{Y}}+\frac{1}{\langle t\rangle^2}\big(\left\lVert \Omega_0\right\rVert^2_{H^N_{Y}}+\left\lVert J_0\right\rVert^2_{H^N_{Y}}\big)\right),\label{nonef3}
		\end{align}
		the associated dissipation functionals
		\begin{align}
			&D(t)=\nu \lVert \nabla_LM\Omega_{\neq}\rVert^2_{L^2}+\mu \lVert \nabla_L MJ_{\neq}\rVert^2_{L^2},\label{nondf4}\\
			&D_0(t)=\nu\left\lVert \partial_YU^1_0\right\rVert^2_{H^N_{Y}}+\mu\left\lVert \partial_YB^1_0\right\rVert^2_{H^N_{Y}}+\frac{1}{\langle t\rangle^2}\big(\nu\left\lVert \partial_Y\Omega_0\right\rVert^2_{H^N_{Y}}+\mu\left\lVert \partial_YJ_0\right\rVert^2_{H^N_{Y}}\big),\label{nondf3}
		\end{align}
		and the Cauchy-Kovalevskaya term
		\begin{align}
			&CK(t)=\sum_{F\in\{\Omega_{\neq},J_{\neq}\}}\left(\left\lVert\Gamma MF\right\rVert^2_{L^2}+\lambda^{\frac{1}{3}}\left\lVert |\partial_X|^{\frac{1}{3}}MF\right\rVert^2_{L^2}\right),\label{nonck4}
		\end{align}
		where $\chi,\Gamma$ given by \eqref{chii} and \eqref{defgamma}. A direct computation gives the following lemma.
	\begin{lemma}
		Assume $0<\nu\leq\mu^3\leq1$, $|\beta|>1/2$, and $N\geq4$. Let $E,E_0,D,D_0,$ and $CK$ be defined in \eqref{nonef4}-\eqref{nonck4} respectively. Then we have
		\begin{align*}
			&\frac{\mathrm{d}}{\mathrm{d}t}E+\frac{1}{100|\beta|}(D+CK)\leq T+N+S,
			\\&\frac{\mathrm{d}}{\mathrm{d}t}E_{0}+\frac{1}{100|\beta|}D_{0}\leq R,
		\end{align*}
		where the transport terms $T$ with $N$, and stretching term $S$ for the $(\Omega_{\neq},J_{\neq})$ system are given by
		\begin{equation*}
			\begin{aligned}
				T=&Re\left\langle M(U\cdot\nabla_{L}\Omega)_{\neq},M\Omega_{\neq}\right\rangle+Re\left\langle M(U\cdot\nabla_{L}J)_{\neq},MJ_{\neq}\right\rangle\\
				&+Re\big(\left\langle M(B\cdot\nabla_{L}J)_{\neq},M\Omega_{\neq}\right\rangle+\left\langle M(B\cdot\nabla_{L}\Omega)_{\neq},MJ_{\neq}\right\rangle\big),
			\end{aligned}
		\end{equation*}
		\begin{equation*}
			\begin{aligned}
				N=&\left|\left\langle M(U\cdot\nabla_{L}\Omega)_{\neq},\frac{2\partial_Y^L}{\beta \Delta_L}MJ_{\neq}\right\rangle\right|+\left|\left\langle M(U\cdot\nabla_{L}J)_{\neq},\frac{2\partial_Y^L}{\beta \Delta_L}M\Omega_{\neq}\right\rangle\right|\\
				&+\left|\left\langle M(B\cdot\nabla_{L}J)_{\neq},\frac{2\partial_Y^L}{\beta \Delta_L}MJ_{\neq}\right\rangle\right|
				+\left|\left\langle M(B\cdot\nabla_{L}\Omega)_{\neq},\frac{2\partial_Y^L}{\beta \Delta_L}M\Omega_{\neq}\right\rangle\right|,
			\end{aligned}
		\end{equation*}
		and
		\begin{equation*}
			\begin{aligned} 
				S= &\left|\left\langle M\left(\left(\frac{2\partial_{XY}^L}{\Delta_L} J\right) \left(\Omega-\frac{2\partial_{XX}}{\Delta_L} \Omega\right)\right)_{\neq}, \left(M J+\frac{2\partial_Y^L}{\beta\Delta_L}M\Omega\right)_{\neq}\right\rangle\right|\\
				&+\left|\left\langle M\left(\left(\frac{2\partial_{XY}^L}{\Delta_L} \Omega\right) \left(J-\frac{2\partial_{XX}}{\Delta_L} J\right)\right)_{\neq},\left(M J+\frac{2\partial_Y^L}{\beta\Delta_L}M\Omega\right)_{\neq}\right\rangle\right|
				.\end{aligned}
		\end{equation*}
		The nonlinear term $R$ for the zero-mode functional is given by
		\begin{equation}\label{zeror}
			\begin{aligned}
				R&=\left|\left\langle\left\langle\partial_Y\right\rangle^N\left(U_{\neq}^2 U_{\neq}^1\right)_0,\left\langle\partial_Y\right\rangle^N\left(\partial_Y U_0^1\right)\right\rangle_{L^2_Y}\right|+\left|\left\langle\left\langle\partial_Y\right\rangle^N\left(B_{\neq}^2 B_{\neq}^1\right)_0,\left\langle\partial_Y\right\rangle^N\left(\partial_Y U_0^1\right)\right\rangle_{L^2_Y}\right| \\
				&+\left|\left\langle\left\langle\partial_Y\right\rangle^N\left(U_{\neq}^2 B_{\neq}^1\right)_0,\left\langle\partial_Y\right\rangle^N\left(\partial_Y B_0^1\right)\right\rangle_{L^2_Y}\right|+\left|\left\langle\left\langle\partial_Y\right\rangle^N\left(B_{\neq}^2 U_{\neq}^1\right)_0,\left\langle\partial_Y\right\rangle^N\left(\partial_Y B_0^1\right)\right\rangle_{L^2_Y}\right| \\
				&\begin{aligned}
					+\frac{1}{\langle t\rangle^2}&\left(\left|\left\langle\left\langle\partial_Y\right\rangle^N\left(U_{\neq}^2 \Omega_{\neq}\right)_0,\left\langle\partial_Y\right\rangle^N\left(\partial_Y \Omega_0\right)\right\rangle_{L^2_Y}\right|+\left|\left\langle\left\langle\partial_Y\right\rangle^N\left(B_{\neq}^2 J_{\neq}\right)_0,\left\langle\partial_Y\right\rangle^N\left(\partial_Y \Omega_0\right)\right\rangle_{L^2_Y} \right| \right.\\
					&+\left|\left\langle\left\langle\partial_Y\right\rangle^N\left(U_{\neq}^2 J_{\neq}\right)_0,\left\langle\partial_Y\right\rangle^N\left(\partial_Y J_0\right)\right\rangle_{L^2_Y}\right|+\left|\left\langle\left\langle\partial_Y\right\rangle^N\left(B_{\neq}^2 \Omega_{\neq}\right)_0,\left\langle\partial_Y\right\rangle^N\left(\partial_Y J_0\right)\right\rangle_{L^2_Y}\right| \\
					&+\left|\left\langle\langle\partial_Y\rangle^N\left(2 \partial_X B_{\neq}^1\left(\Omega_{\neq}-2 \partial_{X X} \Delta_L^{-1} J_{\neq}\right)\right)_0,\langle\partial_Y\rangle^N J_0\right\rangle_{L^2_Y}\right| \\
					&\left.+\left|\left\langle\langle\partial_Y\rangle^N\left(2 \partial_X U_{\neq}^1\left(J_{\neq}-2 \partial_{X X} \Delta_L^{-1} J_{\neq}\right)\right)_0,\langle\partial_Y\rangle^N J_0\right\rangle_{L^2_Y}\right|\right).\end{aligned}
			\end{aligned}
		\end{equation}
	\end{lemma}
	
	In the region $0<\mu^3\leq\nu\leq1$, the energy functionals of $(Z,Q)$ and $(\Omega,J)$ are given by	
		\begin{align}
		&\overline{E}(t)=\frac{1}{2}\left(\lVert MZ\rVert^2_{L^2_{X,Y}}+\left\lVert MQ\right\rVert^2_{L^2_{X,Y}}+\frac{2}{\beta}Re\left\langle\frac{\partial_Y^L\chi{\chi}}{ \Delta_L}MZ,MQ\right\rangle_{L^2_{X,Y}} \right),\label{nonef1}\\
		&\widetilde{E}(t)=\frac{1}{2}\left(\lVert M\Omega\rVert^2_{L^2_{X,Y}}+\left\lVert MJ\right\rVert^2_{L^2_{X,Y}}\right),\label{nonef2}
	\end{align}
	where $\overline{\chi}$ satisfies \eqref{defchi}. The dissipation and the CK terms are defined as
	\begin{align}
		&\overline D(t)=\nu \lVert \nabla_LMZ\rVert^2_{L^2_{X,Y}}+\mu \lVert \nabla_L MQ\rVert^2_{L^2_{X,Y}},\label{nondf1}\\
		&\widetilde{D}(t)=\nu \lVert \nabla_LM\Omega\rVert^2_{L^2_{X,Y}}+\mu \lVert \nabla_L MJ\rVert^2_{L^2_{X,Y}},\label{nondf2}\\
		&\overline{CK}(t)=\sum_{F\in\{Z,Q\}}\left(\left\lVert\Gamma MF\right\rVert_{L^2_{X,Y}}+\lambda^{\frac{1}{3}}\left\lVert |\partial_X|^{\frac{1}{3}}MF\right\rVert^2_{L^2_{X,Y}}\right),\label{nonckf1}\\
		&\widetilde{CK}(t)=\sum_{G\in\{\Omega,J\}}\left(\left\lVert\Gamma MG\right\rVert_{L^2_{X,Y}}+\lambda^{\frac{1}{3}}\left\lVert |\partial_X|^{\frac{1}{3}}MG\right\rVert^2_{L^2_{X,Y}}\right).\label{nonckf2}
	\end{align}
For the region $0<\mu^3\leq\nu\leq1$, we recall a lemma from \cite{D24}.  
	\begin{lemma}[\cite{D24}]
		Assume $0<\mu^3\leq\nu\leq1$, $|\beta|>1/2$, and $N\geq4$. Let $\overline{E},\widetilde{E},\overline{D},\widetilde{D},\overline{CK},\widetilde{CK},E_0$, and $D_0$ be defined in \eqref{nonef1}--\eqref{nonckf2}, \eqref{nonef3}, and \eqref{nondf3} respectively. Then we have
		\begin{align}
			&\frac{\mathrm{d}}{\mathrm{d}t}\overline{E}+\frac{1}{100|\beta|}(\overline{D}+\overline{CK})\leq \overline{T}+\overline{N}+\overline{S},\label{sym}\\
			&\frac{\mathrm{d}}{\mathrm{d}t}\widetilde{E}+\frac{1}{100|\beta|}(\widetilde{D}+\widetilde{CK})\leq 4\sqrt{\overline{E}}\sqrt{\widetilde{E}}+\widetilde{T}+\widetilde{S},\label{ho}
			\\&\frac{\mathrm{d}}{\mathrm{d}t}E_{0}+\frac{1}{100|\beta|}D_{0}\leq R,\label{zero}
		\end{align}
		where the transport terms $\overline{T}$ with $\overline{N}$, and stretching term $\overline{S}$ for $(Z,Q)$ system are given by
		\begin{equation}\label{tran}
			\begin{aligned}
				\overline{T}=&Re\left\langle\Gamma M(U\cdot\nabla_{L}\Omega),MZ\right\rangle_{L^2_{X,Y}}+Re\left\langle\Gamma M(U\cdot\nabla_{L}J),MQ\right\rangle_{L^2_{X,Y}}\\
				&+Re\left(\left\langle\Gamma M(B\cdot\nabla_{L}J),MZ\right\rangle_{L^2_{X,Y}}+\left\langle\Gamma M(B\cdot\nabla_{L}\Omega),MQ\right\rangle_{L^2_{X,Y}}\right),
			\end{aligned}
		\end{equation}
		\begin{equation}\label{trann}
			\begin{aligned}
				\overline{N}=&\left|\left\langle\Gamma M(U\cdot\nabla_{L}\Omega),\frac{2\partial_Y^L}{\beta \Delta_L}MQ\right\rangle_{L^2_{X,Y}}\right|+\left|\left\langle\Gamma M(U\cdot\nabla_{L}J),\frac{2\partial_Y^L}{\beta \Delta_L}MZ\right\rangle_{L^2_{X,Y}}\right|\\
				&+\left|\left\langle\Gamma M(B\cdot\nabla_{L}J),\frac{2\partial_Y^L}{\beta \Delta_L}MQ\right\rangle_{L^2_{X,Y}}\right|
				+\left|\left\langle\Gamma M(B\cdot\nabla_{L}\Omega),\frac{2\partial_Y^L}{\beta \Delta_L}MZ\right\rangle_{L^2_{X,Y}}\right|,
			\end{aligned}
		\end{equation}
		and
		\begin{equation}\label{stechs}
			\begin{aligned} 
				\overline{S}= & \left|\left\langle\Gamma M\left(\left(\frac{2\partial_{XY}^L}{\Delta_L} J\right) \left(\Omega- \frac{2\partial_{XX}}{\Delta_L}\Omega\right)\right), MQ+\frac{2\partial_Y^L}{\beta \Delta_L} MZ\right\rangle_{L^2_{X,Y}}\right| \\ 
				& +\left|\left\langle\Gamma M\left(\left(\frac{2\partial^L_{XY}}{\Delta_L} \Omega\right) \left(J- \frac{2\partial_{XX}}{\Delta_L} J\right)\right), MQ+ \frac{2\partial_Y^L}{\beta \Delta_L} M Z\right\rangle_{L^2_{X,Y}}\right|.
			\end{aligned}
		\end{equation}
		While, the transport $\widetilde{T}$ and stretching term $\widetilde{S}$ for $(\Omega,J)$ system are
			\begin{equation}\label{omegtran}
			\begin{aligned}
				\widetilde{T}=&Re\left\langle M(U\cdot\nabla_{L}\Omega),M\Omega\right\rangle_{L^2_{X,Y}}+Re\left\langle M(U\cdot\nabla_{L}J),MJ\right\rangle_{L^2_{X,Y}}\\
				&+Re\big(\left\langle M(B\cdot\nabla_{L}J),M\Omega\right\rangle_{L^2_{X,Y}}+\left\langle M(B\cdot\nabla_{L}\Omega),MJ\right\rangle_{L^2_{X,Y}}\big),
			\end{aligned}
		\end{equation}
		and
		\begin{equation}\label{omegstr}
			\begin{aligned} 
				\widetilde{S}= \left|\left\langle M\left(\left(\frac{2\partial_{XY}^L}{\Delta_L} J\right) \left(\Omega-\frac{2\partial_{XX}}{\Delta_L} \Omega\right)\right), M J\right\rangle_{L^2_{X,Y}}\right| 
				+\left|\left\langle M\left(\left(\frac{2\partial_{XY}^L}{\Delta_L} \Omega\right) \left(J-\frac{2\partial_{XX}}{\Delta_L} J\right)\right), M J\right\rangle_{L^2_{X,Y}}\right|
			.\end{aligned}
		\end{equation}
		Finally, the nonlinear term $R$ for the zero-mode functional is given by \eqref{zeror}.
	\end{lemma}
	\begin{remark}
		Actually the lemma in \cite{D24} prove the above result only for $0<\mu^3\leq\nu\leq\mu\leq1$. By exactly the same argument, we may extend it to $0<\mu^3\leq\nu\leq1$.
	\end{remark}
	
	Now we establish the bootstrap argument. By a standard local well-posedness argument, we state the following lemma without showing more details. 
	\begin{lemma}\label{st}
		Assume $|\beta|>1/2$ and $N\geq4$. There exists $t_0\ll1$ independent of $\nu,\mu$ such that if $\|(\omega^{in},j^{in})\|_{H^N}\leq\epsilon$, then there holds
		\begin{equation*}
			\sup_{t\in[0,2t_0]}\max\{E(t),\overline{E}(t),\widetilde{E}(t),E_0(t)\}\leq2\epsilon^2.
		\end{equation*}
	\end{lemma}
	
	\noindent\textbf{Bootstrap hypotheses}: 
		If $0<\nu\leq\mu^3\leq1$, let $t> t_0$ be the maximal time such that the following estimates hold on $[t_0,t]$:
	\begin{subequations}\label{bs2}
		\begin{align}
			&E(t)+\frac{1}{100|\beta|}\int_{t_0}^{t}D(\tau)+CK(\tau)\mathrm{d}\tau\leq 10\epsilon^2,\label{bsho2}\\
			&E_{0}(t)+\frac{1}{100|\beta|}\int_{t_0}^{t}D_{0}(\tau)\mathrm{d}\tau\leq 10\epsilon^2.\label{bszero2}
		\end{align}
	\end{subequations}
	
	If $0<\mu^3\leq\nu\leq1$, let $t> t_0$ be the maximal time such that the following estimates hold on $[t_0,t]$:
	\begin{subequations}\label{bs}
		\begin{align}
			&\overline{E}(t)+\frac{1}{100|\beta|}\int_{t_0}^{t}\overline{D}(\tau)+\overline{CK}(\tau)\mathrm{d}\tau\leq 10\epsilon^2,\label{bssym}\\
			&\widetilde{E}(t)+\frac{1}{100|\beta|}\int_{t_0}^{t}\widetilde{D}(\tau)+\widetilde{CK}(\tau)\mathrm{d}\tau\leq 1000\epsilon^2\langle t\rangle^2,\label{bsho}\\
			&E_{0}(t)+\frac{1}{100|\beta|}\int_{t_0}^{t}D_{0}(\tau)\mathrm{d}\tau\leq 10\epsilon^2.\label{bszero}
		\end{align}
	\end{subequations}

	Next, we prove that $t=+\infty$ in the three propositions below.
	\begin{proposition}\label{bsup3}
		When $0<\nu\leq\mu^3\leq1$, under the assumption of Theorem \ref{coro2} and bootstrap hypotheses, there exists $0<\epsilon_0=\epsilon_0(N,\beta)$ such that if $\epsilon<\epsilon_0\nu^{\frac{1}{2}}\mu^{\frac{1}{3}}$, then estimates \eqref{bs2} hold on $[t_0,t]$ with all the occurrences of 10 on the right-hand side replaced by 5. The bootstrap argument implies that $t=+\infty$.
	\end{proposition}
	\begin{proposition}\label{bsup1}
		When $0<\mu^3\leq\nu\leq\mu^{\frac{1}{3}}\leq1$, under the assumption of Theorem \ref{main} and bootstrap hypotheses, there exists $0<\epsilon_0=\epsilon_0(N,\beta)$ such that if $\epsilon<\epsilon_0\lambda^{\frac{1}{2}}$, then estimates \eqref{bs} hold with all the constants on the right-hand side divided by 2 on $[t_0,t]$. It follows by continuity that $t=+\infty$.
	\end{proposition}
	\begin{proposition}\label{bsup2}
		When $0<\mu^{\frac{1}{3}}\leq\nu\leq1$, under the assumption of Theorem \ref{coro} and bootstrap hypotheses, there exists $0<\epsilon_0=\epsilon_0(N,\beta)$ such that if $\epsilon<\epsilon_0\nu^{-1}\mu^{\frac{5}{6}}$, then estimates \eqref{bs} hold with all the constants on the right-hand side divided by 2 on $[t_0,t]$. Therefore, we have that $t=+\infty$.
	\end{proposition}
	
	\subsection{Preliminary lemmas}
	The proof of Proposition \ref{bsup3}--\ref{bsup2} constitutes the majority of our work. Before proving them, we need to introduce some preliminary lemmas. The first result gives commutator estimates for the multipliers $\Gamma$ and $M$.
	\begin{lemma}\label{lemcom}
		For $k,l\neq 0$, we have
		\begin{equation}\label{comsym}
			\left|\Gamma(k,\eta)-\Gamma(l,\xi)\right|\lesssim\frac{\Gamma(l,\xi)}{|l|}\Big(\Gamma(k,\eta)\big|\eta-\xi-(k-l)t\big|+\left|k-l\right|\Big).
		\end{equation}
		If $k=l\neq0$, it holds that
		\begin{equation}\label{commeq}
			|M(k,\eta)-M(k,\xi)|\lesssim e^{\delta_0\lambda^{\frac{1}{3}}t}\left( \big(|k|^{-1}+\lambda^{\frac{1}{3}}|k|^{-\frac{1}{3}}\big)|\eta-\xi|\langle|k,\xi|\rangle^N+\langle|k,\xi|\rangle^{N-1}|\eta-\xi|+\langle|\eta-\xi|\rangle^{N}\right).
		\end{equation}
		Similarly, if $|k-l|<|l|/2,\ \min(|k|,|l|)>C_2\lambda^{-\frac{1}{2}}$, the following estimate is true
		\begin{equation}\label{comm}
			|M(k,\eta)-M(l,\xi)|\lesssim e^{\delta_0\lambda^{\frac{1}{3}}t} \big(\langle|k-l,\eta-\xi|\rangle^N+\langle|l,\xi|\rangle^N\big)\big(\lambda^{\frac{1}{3}}|k|^{-\frac{1}{3}}|\eta-\xi|+|l|^{-1}(\lambda^{\frac{1}{2}}|\eta|+1)|k-l|\big).
		\end{equation}
	\end{lemma}
	\begin{proof}
		A direct computation shows
		\begin{align*}
			\left|\Gamma(k,\eta)-\Gamma(l,\xi)\right|=&\left|\frac{\left(\frac{\xi}{l}-t\right)^2-\left(\frac{\eta}{k}-t\right)^2}{\sqrt{1+\left(\frac{\eta}{k}-t\right)^2}\sqrt{1+\left(\frac{\xi}{l}-t\right)^2}\left(\sqrt{1+\left(\frac{\eta}{k}-t\right)^2}+\sqrt{1+\left(\frac{\xi}{l}-t\right)^2}\right)}\right|\\			
			\lesssim&\frac{\left|\frac{\eta}{k}-\frac{\xi}{l}\right|}{\left(1+\left|\frac{\eta}{k}-t\right|\right)\left(1+\left|\frac{\xi}{l}-t\right|\right)}\\
			\leq&\frac{\Gamma(l,\xi)}{|l|}\frac{\left|[(\eta-\xi)-(k-l)t]k-(k-l)(\eta-kt)\right|}{|k|+\left|\eta-kt\right|}\\
			\lesssim&\frac{\Gamma(l,\xi)}{|l|}\Big(\Gamma(k,\eta)\big|(\eta-\xi)-(k-l)t\big|+\left|k-l\right|\Big),
		\end{align*}
		proving \eqref{comsym}.
		For \eqref{commeq}, let $\tilde{M}(k,\eta)=M(k,\eta)/\langle|k,\eta|\rangle^N$. We have
		\begin{align*}
			|M(k,\eta)-M(k,\xi)|=&\left|\tilde{M}(k,\eta)\langle|k,\eta|\rangle^N-\tilde{M}(k,\eta)\langle|k,\xi|\rangle^N+\tilde{M}(k,\eta)\langle|k,\xi|\rangle^N-\tilde{M}(k,\xi)\langle|k,\xi|\rangle^N\right|\\
			\leq& \tilde{M}(k,\eta)|\langle|k,\eta|\rangle^N-\langle|k,\xi|\rangle^N|+|\tilde{M}(k,\eta)-\tilde{M}(k,\xi)|\langle|k,\xi|\rangle^N\\
			\lesssim&\tilde{M}(k,\eta)\left|\int_0^1\langle|k,\xi+\theta(\eta-\xi)\rangle^{N-1}|\eta-\xi|\mathrm{d}\theta\right|+e^{\delta_0\lambda^{\frac{1}{3}}t}\big(|k|^{-1}+\lambda^{\frac{1}{3}}|k|^{-\frac{1}{3}}\big)|\eta-\xi|\langle|k,\xi|\rangle^N\\
			\lesssim&e^{\delta_0\lambda^{\frac{1}{3}}t}\left(\langle|k,\xi|\rangle^{N-1}|\eta-\xi|+\langle|\eta-\xi|\rangle^{N}+\big(|k|^{-1}+\lambda^{\frac{1}{3}}|k|^{-\frac{1}{3}}\big)|\eta-\xi|\langle|k,\xi|\rangle^N\right).
		\end{align*}
		From \cite{WZ23}, we get \eqref{comm}, completing the proof.
	\end{proof}
	The next lemma gives the estimates of $U$. \eqref{bdu2t} enables us to capture the inviscid damping of $U^2_{\neq}$. Obviously, the same results hold for the magnetic field $B$. 
	\begin{lemma}\label{indap}
		We have the following estimates
		\begin{align}
			&|\hat{U}^1_{\neq}(k,\eta)|\leq\frac{1}{|k|}|\hat{Z}(k,\eta)|=\frac{1}{|k|}|\Gamma(k,\eta)\hat{\Omega}_{\neq}(k,\eta)|,\label{bdu1}\\
			\label{bdu2}
			&|\hat{U}^2_{\neq}(k,\eta)|\leq \left|\frac{\hat{Z}(k,\eta)}{\sqrt{p(k,\eta)}}\right|=\frac{1}{|k|}|\Gamma(k,\eta) \hat{Z}(k,\eta)|=\frac{1}{|k|}|\Gamma^2(k,\eta)\hat{\Omega}_{\neq}(k,\eta)|,\\
			&|\hat{U}^2_{\neq}(k,\eta)|\leq\frac{1}{\langle t\rangle}\left|\langle|k,\eta|\rangle\Gamma(k,\eta)\hat{\Omega}_{\neq}(k,\eta)\right|=\frac{1}{\langle t\rangle}\left|\langle|k,\eta|\rangle\hat{Z}(k,\eta)\right|\label{bdu2t},
		\end{align}
		where $p(k,\eta)=k^2+(\eta-kt)^2$.
	\end{lemma}
	\begin{proof}
		By 
		\begin{equation*}
			|\hat{U}_{\neq}(k,\eta)|=\left|\frac{(\eta-kt,k)}{k^2+(\eta-kt)^2}\hat{\Omega}_{\neq}(k,\eta)\right|
		\end{equation*}
		and the definition of $Z$, we deduce \eqref{bdu2} and \eqref{bdu1}. The inequality $\langle|k,\eta-kt|\rangle\langle|k,\eta|\rangle\gtrsim\langle|kt|\rangle$ implies \eqref{bdu2t}, completing the proof.
	\end{proof}
	\subsection{Proof of bootstrap in the region $\mu^3\leq\nu\leq\mu^{\frac{1}{3}}$}\label{case1}
	First, we consider $0<\mu^3\leq\nu\leq\mu\leq1$ and it holds $\lambda=\nu$. The bound of $\overline{S},R$ has been obtained in \cite{D24}, therefore, we focus on $\overline{T},\overline{N},\widetilde{T},\widetilde{S}$.
	\subsubsection{Estimates on $\overline{T}$}\label{estt}
	Recall \begin{equation*}
		\begin{aligned}
			\overline{T}=&Re\left\langle\Gamma M(U\cdot\nabla_{L}\Omega),MZ\right\rangle_{L^2_{X,Y}}+Re\left\langle\Gamma M(U\cdot\nabla_{L}J),MQ\right\rangle_{L^2_{X,Y}}\\
			&+Re\left(\left\langle\Gamma M(B\cdot\nabla_{L}J),MZ\right\rangle_{L^2_{X,Y}}+\left\langle\Gamma M(B\cdot\nabla_{L}\Omega),MQ\right\rangle_{L^2_{X,Y}}\right).
		\end{aligned}
	\end{equation*}
	Here we only focus on the first term 
	 $$I^1=Re\left\langle\Gamma M(U\cdot\nabla_{L}\Omega),MZ\right\rangle,$$ and the other terms are treated similarly. We need to get an estimate for the commutator $\left[M\Gamma,U^1\partial_{X}\right]$. Since $\nabla_L\cdot U=0$, we have 
	\begin{equation*}
		Re\left\langle(U\cdot\nabla_{L}MZ),MZ\right\rangle=0.
	\end{equation*}
	Using $U=U_{\neq}+U_0$ and $\Omega=\Omega_{\neq}+\Omega_0$, $I^1$ can be decomposed into four terms
	\begin{align*}
		I^1=&\left\langle\left[M\Gamma,U\cdot\nabla_L\right]\Omega,MZ\right\rangle\\
		=&\left\langle\left[M\Gamma,U^1_{\neq}\partial_{X}\right]\Omega_{\neq},MZ\right\rangle
		+\left\langle\left[M\Gamma,U^1_0\partial_{X}\right]\Omega_{\neq},MZ\right\rangle\\
		&+\left\langle\left[M\Gamma,U^2_{\neq}\partial_Y^L\right]\Omega_{\neq},MZ\right\rangle
		+\left\langle\left[M\Gamma,U^2_{\neq}\partial_Y^L\right]\Omega_{0},MZ\right\rangle\\
		=&:I^1_1+I^1_2+I^1_3+I^1_4.
	\end{align*}
	\textbf{Treament of $I^1_1$}
	
	\noindent Considering the effects of $\Gamma$ and $M$, we split $I^1_1$ into three parts:
	\begin{align*}
		I^1_1=&\left|\sum_{k,l\neq0}\iint\left[M(k,\eta)\Gamma(k,\eta)-M(l,\xi)\Gamma(l,\xi)\right]\hat{U}^1_{\neq}(k-l,\eta-\xi) il\hat{\Omega}(l,\xi) M(k,\eta)\overline{\hat{Z}}(k,\eta)\mathrm{d}\xi\mathrm{d}\eta\right|\\
		\lesssim&  \sum_{k,l\neq0}\iint\left|\Gamma(k,\eta)-\Gamma(l,\xi)\right|\left|M(k-l,\eta-\xi)\hat{U}^1_{\neq}(k-l,\eta-\xi)\right|\left|l\hat{\Omega}(l,\xi)\right|\left| M(k,\eta)\hat{Z}(k,\eta)\right|\mathrm{d}\xi\mathrm{d}\eta\\
		&+\sum_{k,l\neq0}\iint\left|\Gamma(k,\eta)-\Gamma(l,\xi)\right|\left|\hat{U}^1_{\neq}(k-l,\eta-\xi)\right| \left|lM(l,\xi)\hat{\Omega}(l,\xi)\right| \left|M(k,\eta)\hat{Z}(k,\eta)\right|\mathrm{d}\xi\mathrm{d}\eta\\
		&+\sum_{k,l\neq0}\iint\left|M(k,\eta)-M(l,\xi)\right|\left|\hat{U}^1_{\neq}(k-l,\eta-\xi)\right|\left|l\hat{Z}(l,\xi)\right|\left|M(k,\eta)\hat{Z}(k,\eta)\right|\mathrm{d}\xi\mathrm{d}\eta\\
		=&:I^1_{11}+I^1_{12}+I^1_{13},
	\end{align*}
	where we used triangle inequality and the fact
	\begin{equation}\label{bdm}
		M(k,\eta)\lesssim A(k-l,\eta-\xi)\langle k-l,\eta-\xi\rangle^N+A(l,\xi)\langle l,\xi\rangle^N.
	\end{equation}
	For $I^1_{11}$, by \eqref{bdu1}, Cauchy-Schwartz, Young's convolution inequality, and the definition of $M$ \eqref{m}, we have
	\begin{align*}
		I^1_{11}\lesssim&\sum_{k,l\neq0}\iint \left|M(k-l,\eta-\xi)\hat{U}^1_{\neq}(k-l,\eta-\xi)\right|\left|l\hat{\Omega}(l,\xi)\right|\left| \Gamma(k,\eta) M(k,\eta)\hat{Z}(k,\eta)\right|\mathrm{d}\xi\mathrm{d}\eta\\
		&+\sum_{k,l\neq0}\iint\left|M(k-l,\eta-\xi)\hat{U}^1_{\neq}(k-l,\eta-\xi)\right|\left|l\hat{Z}(l,\xi)\right|\left|M(k,\eta)\hat{Z}(k,\eta)\right|\mathrm{d}\xi\mathrm{d}\eta\\
		\lesssim&\sum_{k,l\neq0}\iint|M(k-l,\eta-\xi)\hat{Z}(k-l,\eta-\xi)||l\hat{\Omega}(l,\xi)|\left|\Gamma(k,\eta) M(k,\eta)\hat{Z}(k,\eta)\right|\mathrm{d}\xi\mathrm{d}\eta\\
		&+\sum_{k,l\neq0}\iint|M(k-l,\eta-\xi)\hat{Z}(k-l,\eta-\xi)||l\hat{Z}(l,\xi)|\left|M(k,\eta)\hat{Z}(k,\eta)\right|\mathrm{d}\xi\mathrm{d}\eta\\
		\lesssim&\ e^{-\delta_0\nu ^{\frac{1}{3}}t}\langle t\rangle\|MZ_{\neq}\|_{L^2}\frac{1}{\langle t\rangle}\|M\Omega\|_{L^2}\|\Gamma MZ\|_{L^2}+\|MZ\|^3_{L^2}\\
		\lesssim&\ \nu^{-\frac{1}{2}}\frac{1}{\langle t\rangle}\|M\Omega\|_{L^2}\nu^{\frac{1}{6}}\||\partial_X|^{\frac{1}{3}}MZ\|_{L^2}\|\Gamma MZ\|_{L^2}+\nu^{-\frac{1}{3}}\|MZ\|_{L^2}\nu^{\frac{1}{3}}\||\partial_X|^{\frac{1}{3}}MZ\|^2_{L^2}.
	\end{align*}
	Then by the bootstrap hypotheses, we obtain
	\begin{equation*}
		\int_{t_0}^{t}I_{11}^1\mathrm{d}\tau\lesssim\nu^{-\frac{1}{2}}\epsilon^3+\nu^{-\frac{1}{3}}\epsilon^3.
	\end{equation*}
	Next consider $I_{12}^1$. \eqref{comsym} implies
	\begin{align*}
		I^1_{12}=&\sum_{k,l\neq0}\iint\left|\Gamma(k,\eta)-\Gamma(l,\xi)\right|\left|\hat{U}^1_{\neq}(k-l,\eta-\xi)\right| \left|lM(l,\xi)\hat{\Omega}(l,\xi)\right| \left|M(k,\eta)\hat{Z}(k,\eta)\right|\mathrm{d}\xi\mathrm{d}\eta\\
		\lesssim&\sum_{k,l\neq0}\iint|(\eta-\xi)\hat{U}^1_{\neq}(k-l,\eta-\xi)||M(l,\xi)\hat{Z}(l,\xi)||\Gamma(k,\eta) M(k,\eta)\hat{Z}(k,\eta)|\mathrm{d}\xi\mathrm{d}\eta\\
		&+\sum_{k,l\neq0}\iint t|(k-l)\hat{U}^1_{\neq}(k-l,\eta-\xi)||M(l,\xi)\hat{Z}(l,\xi)||\Gamma(k,\eta) M(k,\eta)\hat{Z}(k,\eta)|\mathrm{d}\xi\mathrm{d}\eta\\
		&+\sum_{k,l\neq0}\iint|(k-l)\hat{U}^1_{\neq}(k-l,\eta-\xi)||M(l,\xi)\hat{Z}(l,\xi)||M(k,\eta)\hat{Z}(k,\eta)|\mathrm{d}\xi\mathrm{d}\eta\\
		\lesssim&\|MZ\|^2_{L^2}\|\Gamma MZ\|_{L^2}+e^{-\delta_0\nu ^{\frac{1}{3}}t}t\|MZ\|^2_{L^2}\|\Gamma MZ\|_{L^2}+\|MZ\|^3_{L^2}.
	\end{align*}
	By the bootstrap hypotheses, we have
	\begin{equation*}
		\int_{t_0}^{t}I_{12}^1\mathrm{d}\tau\lesssim\nu^{-\frac{1}{6}}\epsilon^3+\nu^{-\frac{1}{2}}\epsilon^3+\nu^{-\frac{1}{3}}\epsilon^3.
	\end{equation*}
	Focusing on $I^1_{13}$ now, we decompose $I^1_{13}=I^1_{131}+I^1_{132}$ with
	\begin{equation*}
		I^1_{13j}=\sum_{(k,l)\in R_j}\iint\left|M(k,\eta)-M(l,\xi)\right|\left|\hat{U}^1_{\neq}(k-l,\eta-\xi)\right|\left|l\hat{Z}(l,\xi)\right|\left|M(k,\eta)\hat{Z}(k,\eta)\right|\mathrm{d}\xi\mathrm{d}\eta,\ j=1,2,
	\end{equation*}
	where
	\begin{equation}\label{part}
		\begin{aligned} 
			&R_1=\{(k,l)\in \mathbb{Z}^2\big||k-l|\geq|l|/2\ \mathrm{or}\ \min(|k|,|l|)\leq C_2\lambda^{-\frac{1}{2}}\},\\
			&R_2=\{(k,l)\in \mathbb{Z}^2\big||k-l|<|l|/2,\  \min(|k|,|l|)>C_2\lambda^{-\frac{1}{2}}\} .
		\end{aligned}
	\end{equation}
	If $(k,l)\in R_1$, $Z(l,\xi)$ has low $x$ modes and it holds that
	\begin{equation}\label{estl}
		|l|\lesssim|k-l|+\lambda^{-\frac{1}{6}}|l|^{\frac{1}{3}}|k|^{\frac{1}{3}},
	\end{equation}
	which can be easily checked by considering $|k-l|\leq|l|/2$ and $\min(|k|,|l|)\leq C_2\lambda^{-\frac{1}{2}}$, or $|k-l|\geq|l|/2$ separately. Combining it with Cauchy-Schwartz, Young's convolution inequality, \eqref{bdu1}, and \eqref{bdm}, we get
	\begin{align*}
		I^1_{131}=&\sum_{(k,l)\in R_1}\iint\left|M(k,\eta)-M(l,\xi)\right|\left|\hat{U}^1_{\neq}(k-l,\eta-\xi)\right|\left|l\hat{Z}(l,\xi)\right|\left|M(k,\eta)\hat{Z}(k,\eta)\right|\mathrm{d}\xi\mathrm{d}\eta\\
		\lesssim&\sum_{(k,l)\in R_1}\iint\left|M(k-l,\eta-\xi)\hat{U}^1_{\neq}(k-l,\eta-\xi)\right|\left|l\hat{Z}(l,\xi)\right|\left|M(k,\eta)\hat{Z}(k,\eta)\right|\mathrm{d}\xi\mathrm{d}\eta\\
		&+\sum_{k,l\neq0}\iint\left|\hat{U}^1_{\neq}(k-l,\eta-\xi)\right|\left(|k-l|+\nu^{-\frac{1}{6}}|l|^{\frac{1}{3}}|k|^{\frac{1}{3}}\right)\left|M(l,\xi)\hat{Z}(l,\xi)\right|\left|M(k,\eta)\hat{Z}(k,\eta)\right|\mathrm{d}\xi\mathrm{d}\eta\\
		\lesssim&\ \|MZ\|^3_{L^2}+\nu^{-\frac{1}{6}}\|MZ\|_{L^2}\||\partial_{X}^{\frac{1}{3}}|MZ\|^2_{L^2}.
	\end{align*}
 	Thus by the bootstrap hypotheses, we obtain
	\begin{equation*}
		\int_{t_0}^{t}I_{131}^1\mathrm{d}\tau\lesssim\nu^{-\frac{1}{3}}\epsilon^3+\nu^{-\frac{1}{2}}\epsilon^3.
	\end{equation*}
	To control $I^1_{132}$, using \eqref{comm}, we get
	\begin{align*}
		I^1_{132}=&\sum_{(k,l)\in R_2}\iint\left|M(k,\eta)-M(l,\xi)\right|\left|\hat{U}^1_{\neq}(k-l,\eta-\xi)\right|\left|l\hat{Z}(l,\xi)\right|\left|M(k,\eta)\hat{Z}(k,\eta)\right|\mathrm{d}\xi\mathrm{d}\eta\\
		\lesssim&\sum_{(k,l)\in R_2}\iint\big(M(k-l,\eta-\xi)+M(l,\xi)\big)\big(\nu^{\frac{1}{3}}|k|^{-\frac{1}{3}}|\eta-\xi|+|l|^{-1}(\nu^{\frac{1}{2}}|\eta|+1)|k-l|\big)\\
		&\times\left|\hat{U}^1_{\neq}(k-l,\eta-\xi)\right|\left|l\hat{Z}(l,\xi)\right|\left|M(k,\eta)\hat{Z}(k,\eta)\right|\mathrm{d}\xi\mathrm{d}\eta=:\sum_{i=1}^{6}I^1_{132i},
	\end{align*}
	where $I^1_{132i}$ are obtained by expanding $$\big(M(k-l,\eta-\xi)+M(l,\xi)\big)\big(\nu^{\frac{1}{3}}|k|^{-\frac{1}{3}}|\eta-\xi|+|l|^{-1}(\nu^{\frac{1}{2}}|\eta|+1)|k-l|\big).$$ Now we treat each term $I^1_{132i}$ for $i=1,2,3,4,5,6$. First, by the fact $\min(|k|,|l|)>C_2\nu^{-\frac{1}{2}}$ and \eqref{bdu1}, we obtain
	\begin{align*}
		I^1_{1321}=&\sum_{(k,l)\in R_2}\iint\nu^{\frac{1}{3}}|k|^{-\frac{1}{3}}|\eta-\xi|\left|M\hat{U}^1_{\neq}(k-l,\eta-\xi)\right|\left|l\hat{Z}(l,\xi)\right|\left|M(k,\eta)\hat{Z}(k,\eta)\right|\mathrm{d}\xi\mathrm{d}\eta\\
		\lesssim &\sum_{(k,l)\in R_2}\iint\nu^{\frac{1}{2}}\big(|\eta-\xi-(k-l)t|+t|k-l|\big)\left|M\hat{U}^1_{\neq}(k-l,\eta-\xi)\right|\left|l\hat{Z}(l,\xi)\right|\left|M(k,\eta)\hat{Z}(k,\eta)\right|\mathrm{d}\xi\mathrm{d}\eta\\
		\lesssim&\ \langle t\rangle e^{-\delta_0\nu ^{\frac{1}{3}}t}\nu^{\frac{1}{2}}\|M\nabla_LZ\|_{L^2}\|MZ\|^2_{L^2}.
	\end{align*}
	Next, we use the fact $|k-l|\leq|l|/2$ and \eqref{bdu1} to arrive at
	\begin{align*}
		I^1_{1322}=&\sum_{(k,l)\in R_2}\iint|l|^{-1}\nu^{\frac{1}{2}}|\eta||k-l|\left|\widehat{(MU^1_{\neq})}(k-l,\eta-\xi)\right|\left|l\hat{Z}(l,\xi)\right|\left|M(k,\eta)\hat{Z}(k,\eta)\right|\mathrm{d}\xi\mathrm{d}\eta\\
		\lesssim &\sum_{(k,l)\in R_2}\iint\nu^{\frac{1}{2}}\left|\widehat{(MU^1_{\neq})}(k-l,\eta-\xi)\right|\left|l\hat{Z}(l,\xi)\right|\big(|\eta-kt|+t|k|\big)\left|M(k,\eta)\hat{Z}(k,\eta)\right|\mathrm{d}\xi\mathrm{d}\eta\\
		\lesssim&\ \langle t\rangle e^{-\delta_0\nu ^{\frac{1}{3}}t}\nu^{\frac{1}{2}}\|MZ\|^2_{L^2}\|M\nabla_LZ\|_{L^2},
	\end{align*}
	and
	\begin{align*}
		I^1_{1323}=&\sum_{(k,l)\in R_2}\iint|l|^{-1}|k-l|\left|\widehat{(MU^1_{\neq})}(k-l,\eta-\xi)\right|\left|l\hat{Z}(l,\xi)\right|\left|M(k,\eta)\hat{Z}(k,\eta)\right|\mathrm{d}\xi\mathrm{d}\eta\\
		\lesssim&\ \|MZ\|^3_{L^2}.
	\end{align*}
	Then by $|k|^{-\frac{1}{3}}|l|\lesssim|l|^{\frac{1}{3}}|k|^{\frac{1}{3}}$, we have
	\begin{align*}
		I^1_{1324}=&\sum_{(k,l)\in R_2}\iint\nu^{\frac{1}{3}}|k|^{-\frac{1}{3}}|\eta-\xi|\left|\hat{U}^1_{\neq}(k-l,\eta-\xi)\right|\left|lM(l,\xi)\hat{Z}(l,\xi)\right|\left|M(k,\eta)\hat{Z}(k,\eta)\right|\mathrm{d}\xi\mathrm{d}\eta\\
		\lesssim &\sum_{(k,l)\in R_2}\iint\nu^{\frac{1}{3}}|l|^{\frac{1}{3}}|k|^{\frac{1}{3}}\left|(\eta-\xi)\hat{U}^1_{\neq}(k-l,\eta-\xi)\right|\left|(\widehat{MZ})(l,\xi)\right|\left|(\widehat{MZ})(k,\eta)\right|\mathrm{d}\xi\mathrm{d}\eta\\
		\lesssim&\ \nu^{\frac{1}{3}}\|MZ\|_{L^2}\||\partial_X|^{\frac{1}{3}}MZ\|^2_{L^2}.
	\end{align*}
	Similar to $I^1_{1322}$ and $I^1_{1323}$, it holds
	\begin{align*}
		I^1_{1325}+I_{1326}^1=&\sum_{(k,l)\in R_2}\iint\big(|l|^{-1}(\nu^{\frac{1}{2}}|\eta|+1)|k-l|\big)\left|\hat{U}^1_{\neq}(k-l,\eta-\xi)\right|\left|l(\widehat{MZ})(l,\xi)\right|\left|(\widehat{MZ})(k,\eta)\right|\mathrm{d}\xi\mathrm{d}\eta\\
		\lesssim&\ \langle t\rangle e^{-\delta_0\nu ^{\frac{1}{3}}t}\nu^{\frac{1}{2}}\|MZ\|^2_{L^2}\|M\nabla_LZ\|_{L^2}+\|MZ\|^3_{L^2}.
	\end{align*}
	Therefore, by the bootstrap hypotheses, we obtain
	\begin{equation*}
		\int_{t_0}^{t}I_{132}^1\mathrm{d}\tau\lesssim\nu^{-\frac{1}{2}}\epsilon^3+\nu^{-\frac{1}{2}}\epsilon^3+\nu^{-\frac{1}{3}}\epsilon^3+\epsilon^3+\nu^{-\frac{1}{2}}\epsilon^3+\nu^{-\frac{1}{3}}\epsilon^3\lesssim\nu^{-\frac{1}{2}}\epsilon^3.
	\end{equation*}
	Putting together the estimates above, we conclude
	\begin{equation}\label{estT11}
		\int_{t_0}^{t}I_{1}^1\mathrm{d}\tau\lesssim\nu^{-\frac{1}{2}}\epsilon^3.
	\end{equation}
	\textbf{Treament of $I^1_2$}
	
	\noindent By triangle inequality, we get
	\begin{align*}
		I^1_2=&\left|\left\langle\left[M\Gamma,U^1_0\partial_{X}\right]\Omega_{\neq},MZ\right\rangle\right|\\
		=&\left|\sum_{k\neq0}\iint\big[M(k,\eta)\Gamma(k,\eta)-M(k,\xi)\Gamma(k,\xi)\big]\hat{U}^1_{0}(\eta-\xi) ik\hat{\Omega}(k,\xi) M(k,\eta)\overline{\hat{Z}}(k,\eta)\mathrm{d}\xi\mathrm{d}\eta\right|\\
		\lesssim&\sum_{k\neq0}\iint|\Gamma(k,\eta)-\Gamma(k,\xi)||\hat{U}^1_{0}(\eta-\xi)||k M(k,\xi)\hat{\Omega}(k,\xi)|| M(k,\eta)\hat{Z}(k,\eta)|\mathrm{d}\xi\mathrm{d}\eta\\
		&+\sum_{k\neq0}\iint|M(k,\eta)-M(k,\xi)|\Gamma(k,\eta)|\hat{U}^1_{0}(\eta-\xi)||k \hat{\Omega}(k,\xi)|| M(k,\eta)\hat{Z}(k,\eta)|\mathrm{d}\xi\mathrm{d}\eta=:I^1_{21}+I^1_{22}.
	\end{align*}
	For $I^1_{21}$, applying \eqref{comsym} with $k=l$, Cauchy-Schwartz, and Young's convolution inequality, we have
	\begin{align*}
		I^1_{21}=&\sum_{k\neq0}\iint|\Gamma(k,\eta)-\Gamma(k,\xi)||\hat{U}^1_{0}(\eta-\xi)||k M(k,\xi)\hat{\Omega}(k,\xi)|| M(k,\eta)\hat{Z}(k,\eta)|\mathrm{d}\xi\mathrm{d}\eta\\
		\lesssim&\sum_{k\neq0}\iint|(\eta-\xi)\hat{U}^1_{0}(\eta-\xi)|| M(k,\xi)\hat{Z}(k,\xi)|| \Gamma(k,\eta) M(k,\eta)\hat{Z}(k,\eta)|\mathrm{d}\xi\mathrm{d}\eta\\
		\lesssim&\|U^1_0\|_{H^N}\|MZ\|_{L^2}\|\Gamma MZ\|_{L^2}.
	\end{align*}
	Using the bootstrap hypotheses gives
	\begin{equation*}
		\int_{t_0}^{t}I_{21}^1\mathrm{d}\tau\lesssim\nu^{-\frac{1}{6}}\epsilon^3.
	\end{equation*}
	Turning to $I^1_{22}$, by \eqref{commeq} and 
	\begin{equation*}
	\Gamma(k,\xi)^{-1}\lesssim1+\left|\frac{\xi}{k}-t\right|\leq\left(1+\left|\frac{\eta}{k}-t\right|\right)+\frac{|\eta-\xi|}{|k|},
	\end{equation*}
	 we get
	\begin{align*}
		I^1_{22}=&\sum_{k\neq0}\iint|M(k,\eta)-M(k,\xi)|\Gamma(k,\eta)|\hat{U}^1_{0}(\eta-\xi)||k \hat{\Omega}(k,\xi)|| M(k,\eta)\hat{Z}(k,\eta)|\mathrm{d}\xi\mathrm{d}\eta\\
		\lesssim&\sum_{k\neq0}\iint\big[\big(|k|^{-1}+\nu^{\frac{1}{3}}|k|^{-\frac{1}{3}}\big)|\eta-\xi|M(k,\xi)+e^{\delta_0\nu^{\frac{1}{3}}t}\langle|k,\xi|\rangle^{N-1}|\eta-\xi|+e^{\delta_0\nu^{\frac{1}{3}}t}\langle|\eta-\xi|\rangle^{N}\big]\\
		&\times(1+|\eta/k-t|)\Gamma(k,\xi)\Gamma(k,\eta) |\hat{U}^1_{0}(\eta-\xi)||k \hat{\Omega}(k,\xi)|| (\widehat{MZ})(k,\eta)|\mathrm{d}\xi\mathrm{d}\eta\\
		&+\sum_{k\neq0}\iint\big[\big(|k|^{-1}+\nu^{\frac{1}{3}}|k|^{-\frac{1}{3}}\big)|\eta-\xi|M(k,\xi)+e^{\delta_0\nu^{\frac{1}{3}}t}\langle|k,\xi|\rangle^{N-1}|\eta-\xi|+e^{\delta_0\nu^{\frac{1}{3}}t}\langle|\eta-\xi|\rangle^{N}\big]\\
		&\times|\eta-\xi||k|^{-1}\Gamma(k,\xi)\Gamma(k,\eta) |\hat{U}^1_{0}(\eta-\xi)||k \hat{\Omega}(k,\xi)|| (\widehat{MZ})(k,\eta)|\mathrm{d}\xi\mathrm{d}\eta\\
		=&:I^1_{221}+I^1_{222}.
	\end{align*}
	For $I^1_{221}$, we have
	\begin{align*}
		I^1_{221}=&\sum_{k\neq0}\iint\big[\big(|k|^{-1}+\nu^{\frac{1}{3}}|k|^{-\frac{1}{3}}\big)|\eta-\xi|M(k,\xi)+e^{\delta_0\nu^{\frac{1}{3}}t}\langle|k,\xi|\rangle^{N-1}|\eta-\xi|+e^{\delta_0\nu^{\frac{1}{3}}t}\langle|\eta-\xi|\rangle^{N}\big]\\
		&\times(1+|\eta/k-t|)\Gamma(k,\xi)\Gamma(k,\eta) |\hat{U}^1_{0}(\eta-\xi)||k \hat{\Omega}(k,\xi)|| (\widehat{MZ})(k,\eta)|\mathrm{d}\xi\mathrm{d}\eta\\
		=&\sum_{k\neq0}\iint\big[\big(|k|^{-1}+\nu^{\frac{1}{3}}|k|^{-\frac{1}{3}}\big)|\eta-\xi|M(k,\xi)+e^{\delta_0\nu^{\frac{1}{3}}t}\langle|k,\xi|\rangle^{N-1}|\eta-\xi|+e^{\delta_0\nu^{\frac{1}{3}}t}\langle|\eta-\xi|\rangle^{N}\big]\\ &\times|\hat{U}^1_{0}(\eta-\xi)||k \hat{Z}(k,\xi)|| (\widehat{MZ})(k,\eta)|\mathrm{d}\xi\mathrm{d}\eta\\
		\lesssim&\ \|U^1_0\|_{H^N}\|MZ\|^2_{L^2}+\nu^{\frac{1}{3}}\|U^1_0\|_{H^N}\||\partial_X|^{\frac{1}{3}}MZ\|^2_{L^2}.
	\end{align*}
	Similarly, it holds
	\begin{align*}
		I^1_{222}=&\sum_{k\neq0}\iint\big[\big(|k|^{-1}+\nu^{\frac{1}{3}}|k|^{-\frac{1}{3}}\big)|\eta-\xi|M(k,\xi)+e^{\delta_0\nu^{\frac{1}{3}}t}\langle|k,\xi|\rangle^{N-1}|\eta-\xi|+e^{\delta_0\nu^{\frac{1}{3}}t}\langle|\eta-\xi|\rangle^{N}\big]\\
		&\times|\eta-\xi||k|^{-1}\Gamma(k,\xi)\Gamma(k,\eta) |\hat{U}^1_{0}(\eta-\xi)||k \hat{\Omega}(k,\xi)|| (\widehat{MZ})(k,\eta)|\mathrm{d}\xi\mathrm{d}\eta\\
		=&\sum_{k\neq0}\iint\big[\big(|k|^{-1}+\nu^{\frac{1}{3}}|k|^{-\frac{1}{3}}\big)|\eta-\xi|M(k,\xi)+e^{\delta_0\nu^{\frac{1}{3}}t}\langle|k,\xi|\rangle^{N-1}|\eta-\xi|+e^{\delta_0\nu^{\frac{1}{3}}t}\langle|\eta-\xi|\rangle^{N}\big]\\
		&\times|(\eta-\xi)\hat{U}^1_{0}(\eta-\xi)||\hat{Z}(k,\xi)|| (\widehat{\Gamma MZ})(k,\eta)|\mathrm{d}\xi\mathrm{d}\eta\\
		\lesssim&\ \|U^1_0\|_{H^N}\|MZ\|^2_{L^2}+\nu^{\frac{1}{3}}\|U^1_0\|_{H^N}\||\partial_X|^{\frac{1}{3}}MZ\|^2_{L^2}+\|\partial_YU^1_0\|_{H^N}\|MZ\|_{L^2}\|\Gamma MZ\|_{L^2}.
	\end{align*}
	Then by the bootstrap hypotheses, we obtain
	\begin{equation*}
		\int_{t_0}^{t}I_{22}^1\mathrm{d}\tau\lesssim\nu^{-\frac{1}{3}}\epsilon^3+\epsilon^3+\nu^{-\frac{1}{2}}\epsilon^3.
	\end{equation*}
	Therefore, it holds that
	\begin{equation}\label{estT12}
		\int_{t_0}^{t}I_{2}^1\mathrm{d}\tau\lesssim\nu^{-\frac{1}{2}}\epsilon^3.
	\end{equation}
	\textbf{Treament of $I^1_3$}
	
	\noindent In this case, the commutator can not gain us anything, and hence we use \eqref{bdm} to get
	\begin{align*}
		I^1_3=&|\left\langle\left[M\Gamma,U^2_{\neq}\partial_Y^L\right]\Omega_{\neq},MZ\right\rangle|\\
		=&\left|\sum_{k,l\neq0}\iint\left[M(k,\eta)\Gamma(k,\eta)-M(l,\xi)\Gamma(l,\xi)\right]\hat{U}^2_{\neq}(k-l,\eta-\xi) i(\xi-lt)\hat{\Omega}_{\neq}(l,\xi) M(k,\eta)\overline{\hat{Z}}(k,\eta)\mathrm{d}\xi\mathrm{d}\eta\right|\\
		\lesssim&  \sum_{k,l\neq0}\iint\left|M(k-l,\eta-\xi)\hat{U}^2_{\neq}(k-l,\eta-\xi)\right|\left|(\xi-lt)\hat{\Omega}_{\neq}(l,\xi)\right|\left| \Gamma(k,\eta) M(k,\eta)\hat{Z}(k,\eta)\right|\mathrm{d}\xi\mathrm{d}\eta\\
		&+\sum_{k,l\neq0}\iint\left|\hat{U}^2_{\neq}(k-l,\eta-\xi)\right| \left|(\xi-lt)M(l,\xi)\hat{\Omega}_{\neq}(l,\xi)\right| \left|\Gamma(k,\eta) M(k,\eta)\hat{Z}(k,\eta)\right|\mathrm{d}\xi\mathrm{d}\eta\\
		&+\sum_{k,l\neq0}\iint\left|\hat{U}^2_{\neq}(k-l,\eta-\xi)\right|\left|(\xi-lt)M(l,\xi)\hat{Z}(l,\xi)\right|\left|M(k,\eta)\hat{Z}(k,\eta)\right|\mathrm{d}\xi\mathrm{d}\eta=:I^1_{31}+I^1_{32}+I^1_{33}.
	\end{align*}
	First focus on $I^1_{31}$. By
	\eqref{bdu2}, the fact $\Gamma(k,\eta)=|k|/\sqrt{p(k,\eta)}$, and
 	\begin{equation*}\sqrt{p(l,\xi)}\lesssim\sqrt{p(k-l,\eta-\xi)}+\sqrt{p(k,\eta)},\end{equation*} we have
	\begin{align*}
		I^1_{31}=&\sum_{k,l\neq0}\iint\left|M(k-l,\eta-\xi)\hat{U}^2_{\neq}(k-l,\eta-\xi)\right|\left|(\xi-lt)\hat{\Omega}_{\neq}(l,\xi)\right|\left| \Gamma(k,\eta) M(k,\eta)\hat{Z}(k,\eta)\right|\mathrm{d}\xi\mathrm{d}\eta\\
		\lesssim&\sum_{k,l\neq0}\iint\left|\frac{(\widehat{MZ})(k-l,\eta-\xi)}{\sqrt{p(k-l,\eta-\xi)}}\right|\left|(\xi-lt)\sqrt{p(l,\xi)}\hat{Z}_{\neq}(l,\xi)\right|\left|  \frac{k(\widehat{MZ})(k,\eta)}{\sqrt{p(k,\eta)}}\right|\mathrm{d}\xi\mathrm{d}\eta\\
		\lesssim&\sum_{k,l\neq0}\iint\left|(\widehat{MZ})(k-l,\eta-\xi)\right|\left|(\xi-lt)\hat{Z}_{\neq}(l,\xi)\right|\left|  (\widehat{\Gamma MZ})(k,\eta)\right|\mathrm{d}\xi\mathrm{d}\eta\\
		&+\sum_{k,l\neq0}\iint\left|\frac{|k-l|}{\sqrt{p(k-l,\eta-\xi)}}(\widehat{MZ})(k-l,\eta-\xi)\right|\left|(\xi-lt)\hat{Z}_{\neq}(l,\xi)\right|\left|  (\widehat{MZ})(k,\eta)\right|\mathrm{d}\xi\mathrm{d}\eta\\
		&+\sum_{k,l\neq0}\iint\left|\frac{(\widehat{MZ})(k-l,\eta-\xi)}{\sqrt{p(k-l,\eta-\xi)}}\right|\left|(\xi-lt)l\hat{Z}_{\neq}(l,\xi)\right|\left|  (\widehat{MZ})(k,\eta)\right|\mathrm{d}\xi\mathrm{d}\eta\\
		\lesssim&\ \|MZ\|_{L^2}\|\nabla_LMZ\|_{L^2}\|\Gamma MZ\|_{L^2}.
	\end{align*}
	Due to the bootstrap hypotheses, we obtain
	\begin{equation*}
		\int_{t_0}^{t}I_{31}^1\mathrm{d}\tau\lesssim\nu^{-\frac{1}{2}}\epsilon^3.
	\end{equation*}
	For $I^1_{32}$, using \eqref{comsym}, we deduces
	\begin{align*}
		I^1_{32}=&\sum_{k,l\neq0}\iint\left|\hat{U}^2_{\neq}(k-l,\eta-\xi)\right| \left|(\xi-lt)M(l,\xi)\hat{\Omega}_{\neq}(l,\xi)\right| \left|\Gamma(k,\eta) M(k,\eta)\hat{Z}(k,\eta)\right|\mathrm{d}\xi\mathrm{d}\eta\\
		\lesssim&\sum_{k,l\neq0}\iint|\Gamma(k,\eta)-\Gamma(l,\xi)|\left|\hat{U}^2_{\neq}(k-l,\eta-\xi)\right| \left|(\xi-lt)(\widehat{M\Omega_{\neq}})(l,\xi)\right| \left| (\widehat{MZ})(k,\eta)\right|\mathrm{d}\xi\mathrm{d}\eta\\
		&+\sum_{k,l\neq0}\iint\left|\hat{U}^2_{\neq}(k-l,\eta-\xi)\right| \left|(\xi-lt)(\widehat{MZ})(l,\xi)\right| \left| (\widehat{MZ})(k,\eta)\right|\mathrm{d}\xi\mathrm{d}\eta\\
		\lesssim&\sum_{k,l\neq0}\iint\frac{\Gamma(l,\xi)}{|l|}\Gamma(k,\eta)\left|(\eta-\xi)-(k-l)t\right|\left|\hat{U}^2_{\neq}(k-l,\eta-\xi)\right| \left|(\xi-lt)(\widehat{M\Omega_{\neq}})(l,\xi)\right|\left| (\widehat{MZ})(k,\eta)\right|\mathrm{d}\xi\mathrm{d}\eta\\
		&+\sum_{k,l\neq0}\iint\frac{\Gamma(l,\xi)}{|l|}\left|k-l\right|\left|\hat{U}^2_{\neq}(k-l,\eta-\xi)\right| \left|(\xi-lt)(\widehat{M\Omega_{\neq}})(l,\xi)\right|\left| (\widehat{MZ})(k,\eta)\right|\mathrm{d}\xi\mathrm{d}\eta\\
		&+\sum_{k,l\neq0}\iint\left|\hat{U}^2_{\neq}(k-l,\eta-\xi)\right| \left|(\xi-lt)(\widehat{MZ})(l,\xi)\right| \left| (\widehat{MZ})(k,\eta)\right|\mathrm{d}\xi\mathrm{d}\eta=\sum_{j=1}^{3}I^1_{32j}.
	\end{align*}
	To control $I^1_{321}$, by \eqref{bdu2}, we obtain
	\begin{align*}
		I^1_{321}=&\sum_{k,l\neq0}\iint\frac{\Gamma(l,\xi)}{|l|}\Gamma(k,\eta)\left|(\eta-\xi)-(k-l)t\right|\left|\hat{U}^2_{\neq}(k-l,\eta-\xi)\right| \left|(\xi-lt)(\widehat{M\Omega_{\neq}})(l,\xi)\right|\left| (\widehat{MZ})(k,\eta)\right|\mathrm{d}\xi\mathrm{d}\eta\\
		\lesssim&\sum_{k,l\neq0}\iint\left|\left|(\eta-\xi)-(k-l)t\right|\hat{U}^2_{\neq}(k-l,\eta-\xi)\right| \left|(\xi-lt)(\widehat{MZ})(l,\xi)\right|\left| (\widehat{\Gamma MZ})(k,\eta)\right|\mathrm{d}\xi\mathrm{d}\eta\\
		\lesssim&\ \|MZ\|_{L^2}\|\nabla_LMZ\|_{L^2}\|\Gamma MZ\|_{L^2}.
	\end{align*}
	We next turn to $I^1_{322}$ and $I^1_{323}$, and use \eqref{bdu2} to arrive at
	\begin{align*}
		I^1_{322}+I^1_{323}=&\sum_{k,l\neq0}\iint\frac{\Gamma(l,\xi)}{|l|}|k-l|\left|\hat{U}^2_{\neq}(k-l,\eta-\xi)\right| \left|(\xi-lt)(\widehat{M\Omega_{\neq}})(l,\xi)\right|\left| (\widehat{MZ})(k,\eta)\right|\mathrm{d}\xi\mathrm{d}\eta\\
		&+\sum_{k,l\neq0}\iint\left|\hat{U}^2_{\neq}(k-l,\eta-\xi)\right|\left|(\xi-lt)(\widehat{MZ})(l,\xi)\right|\left|(\widehat{MZ})(k,\eta)\right|\mathrm{d}\xi\mathrm{d}\eta\\
		\lesssim&\sum_{k,l\neq0}\iint\left|\Gamma (k-l,\eta-\xi)\hat{Z}(k-l,\eta-\xi)\right| \left|(\xi-lt)(\widehat{MZ})(l,\xi)\right|\left| (\widehat{MZ})(k,\eta)\right|\mathrm{d}\xi\mathrm{d}\eta\\
		\lesssim&\ \|\Gamma MZ\|_{L^2}\|\nabla_LMZ\|_{L^2}\|MZ\|_{L^2}.
	\end{align*}
	Noting $I^1_{33}=I^1_{323}$, by bootstrap hypotheses, we obtain
	\begin{equation*}
		\int_{t_0}^{t}I_{31}^1+I_{32}^1+I_{33}^1\mathrm{d}\tau\lesssim\nu^{-\frac{1}{2}}\epsilon^3.
	\end{equation*}
	Furthermore, it holds that
	\begin{equation}\label{estT13}
		\int_{t_0}^{t}I_{3}^1\mathrm{d}\tau\lesssim\nu^{-\frac{1}{2}}\epsilon^3.
	\end{equation}
	\textbf{Treament of $I^1_4$}
	
	\noindent By \eqref{bdu2}, $\hat{\Omega}_0=i\xi\hat{U^1_0}$, and \eqref{bdu2t}, we get
	\begin{align*}
		I^1_4=&\left|\left\langle\left[M\Gamma,U^2_{\neq}\partial_Y^L\right]\Omega_{0},MZ\right\rangle\right|\\
		\leq&\sum_{k\neq0}\iint|M(k,\eta)\Gamma(k,\eta)||\hat{U}^2(k,\eta-\xi)||\xi \hat{\Omega}_0(\xi)|| M(k,\eta)\hat{Z}(k,\eta)|\mathrm{d}\xi\mathrm{d}\eta\\
		\lesssim&\sum_{k\neq0}\iint|M(k,\eta-\xi)\hat{U}^2(k,\eta-\xi)||\xi \hat{\Omega}_0(\xi)||\Gamma(k,\eta) M(k,\eta)\hat{Z}(k,\eta)|\mathrm{d}\xi\mathrm{d}\eta
		\\&+\sum_{k\neq0}\iint e^{\delta_0\nu^{\frac{1}{3}}t}|\hat{U}^2(k,\eta-\xi)||\xi \langle \xi\rangle^N\hat{\Omega}_0(\xi)||\Gamma(k,\eta) M(k,\eta)\hat{Z}(k,\eta)|\mathrm{d}\xi\mathrm{d}\eta\\
		\lesssim&\sum_{k\neq0}\iint|\Gamma(k,\eta) M(k,\eta)\hat{Z}(k,\eta-\xi)||\xi^2 \hat{U}^1_0(\xi)||\Gamma(k,\eta) M(k,\eta)\hat{Z}(k,\eta)|\mathrm{d}\xi\mathrm{d}\eta\\
		&+\sum_{k\neq0}\iint|e^{\delta_0\nu^{\frac{1}{3}}t}\langle|k,\eta-\xi|\rangle\hat{Z}(k,\eta-\xi)|\frac{1}{\langle t\rangle}\left| \langle \xi\rangle^N\xi\hat{\Omega}_0(\xi)\right||\Gamma(k,\eta) M(k,\eta)\hat{Z}(k,\eta)|\mathrm{d}\xi\mathrm{d}\eta\\
		\lesssim&\ \|U^1_0\|_{H^N}\|\Gamma MZ\|^2_{L^2}+\|MZ\|_{L^2}\frac{1}{\langle t\rangle}\|\partial_Y\Omega_0\|_{H^N}\|\Gamma MZ\|_{L^2}.
	\end{align*}
	Thus the bootstrap hypotheses imply that
	\begin{equation}\label{estT14}
		\int_{t_0}^{t}I_{4}^1\mathrm{d}\tau\lesssim\nu^{-\frac{1}{2}}\epsilon^3.
	\end{equation}
	Therefore, from the bounds \eqref{estT11}-\eqref{estT14}, we conclude that
	\begin{equation}\label{estTsym}
		\int_{t_0}^{t}T\mathrm{d}\tau\lesssim\nu^{-\frac{1}{2}}\epsilon^3.
	\end{equation}
	\subsubsection{Estimates on $\overline{N}$}
	Recall
	\begin{equation*}
		\begin{aligned}
			\overline{N}=&\left|\left\langle\Gamma M(U\cdot\nabla_{L}\Omega),\frac{2\partial_Y^L}{\beta \Delta_L}MQ\right\rangle_{L^2_{X,Y}}\right|+\left|\left\langle\Gamma M(U\cdot\nabla_{L}J),\frac{2\partial_Y^L}{\beta \Delta_L}MZ\right\rangle_{L^2_{X,Y}}\right|\\
			&+\left|\left\langle\Gamma M(B\cdot\nabla_{L}J),\frac{2\partial_Y^L}{\beta \Delta_L}MQ\right\rangle_{L^2_{X,Y}}\right|
			+\left|\left\langle\Gamma M(B\cdot\nabla_{L}\Omega),\frac{2\partial_Y^L}{\beta \Delta_L}MZ\right\rangle_{L^2_{X,Y}}\right|.
		\end{aligned}
	\end{equation*}
	Here we only treat the first term
	$$ I^2=\left|\left\langle\Gamma M(U\cdot\nabla_{L}\Omega),\frac{2\partial_Y^L}{\beta \Delta_L}MQ\right\rangle\right|,$$ and the other terms are treated similarly. In contrast to $I^1$, we do not need to use commutator estimate, but to extract a factor $\Gamma$ from the symbol $\frac{2\partial_Y^L}{\beta \Delta_L}$,
	\begin{equation*}
		\left|\frac{2(\eta-kt)}{i\beta p(k,\eta)}\right|\lesssim\frac{\Gamma(k,\eta)}{|k|}.
	\end{equation*}
	Similar as $I^1$, we decompose $I^2$ into four parts
	\begin{align*}
		I^2=&\left|\sum_{k,l}\iint (\Gamma M)(k,\eta)\left(\hat{U}(k-l,\eta-\xi)\cdot(l,\xi-lt)\hat{\Omega}(l,\xi)\right)\frac{2(\eta-kt)}{-i\beta p}M(k,\eta)\bar{\hat{Q}}(k,\eta)\mathrm{d}\xi\mathrm{d}\eta\right|\\
		\lesssim&\sum_{k,l\neq0}\iint |(\Gamma M)(k,\eta)|\left|\hat{U}^1_{\neq}(k-l,\eta-\xi)\right||l\hat{\Omega}(l,\xi)||(\widehat{\Gamma MQ})(k,\eta)|\mathrm{d}\xi\mathrm{d}\eta\\
		&+\sum_{k\neq0}\iint |(\Gamma M)(k,\eta)|\left|\hat{U}^1_0(\eta-\xi)\right||k\hat{\Omega}(k,\xi)||(\widehat{\Gamma MQ})(k,\eta)|\mathrm{d}\xi\mathrm{d}\eta\\
		&+\sum_{k,l\neq0}\iint |(\Gamma M)(k,\eta)|\left|\hat{U}^2(k-l,\eta-\xi)\right||(\xi-lt)\hat{\Omega}_{\neq}(l,\xi)||(\widehat{\Gamma MQ})(k,\eta)|\mathrm{d}\xi\mathrm{d}\eta\\
		&+\sum_{k\neq0}\iint |(\Gamma M)(k,\eta)|\left|\hat{U}^2(k,\eta-\xi)\right||\xi\hat{\Omega}_{0}(\xi)||(\widehat{\Gamma MQ})(k,\eta)|\mathrm{d}\xi\mathrm{d}\eta=\sum_{i=1}^{4}I^2_i.
	\end{align*}
	For $I^2_1$, using \eqref{bdm} and \eqref{bdu1}, and the fact
	\begin{equation*}
		|l\hat{\Omega}(l,\xi)|\lesssim|l(1+|\xi/l-t|)\hat{Z}(l,\xi)|\lesssim|(\widehat{\nabla_LZ})(l,\xi)|,
	\end{equation*}
	we deduce
	\begin{align*}
		I^2_1=&\sum_{k,l\neq0}\iint |\Gamma(k,\eta) M(k,\eta)|\left|\hat{U}^1_{\neq}(k-l,\eta-\xi)\right||l\hat{\Omega}(l,\xi)||(\widehat{\Gamma MQ})(k,\eta)|\mathrm{d}\xi\mathrm{d}\eta\\
		\lesssim&\sum_{k,l\neq0}\iint \left|\widehat{MZ}(k-l,\eta-\xi)\right||(\widehat{\nabla_LZ})(l,\xi)||(\widehat{\Gamma MQ})(k,\eta)|\mathrm{d}\xi\mathrm{d}\eta\\
		&+\sum_{k,l\neq0}\iint\left|\hat{Z}(k-l,\eta-\xi)\right||(\widehat{\nabla_LMZ})(l,\xi)||(\widehat{\Gamma MQ})(k,\eta)|\mathrm{d}\xi\mathrm{d}\eta\\
		\lesssim&\ \|MZ\|_{L^2}\|\nabla_LMZ\|_{L^2}\|\Gamma MQ\|_{L^2}.
	\end{align*}
	Similarly, it holds
	\begin{align*}
		I^2_2=&\sum_{k\neq0}\iint |\Gamma(k,\eta) M(k,\eta)|\left|\hat{U}^1_0(\eta-\xi)\right||k\hat{\Omega}(k,\xi)||(\widehat{\Gamma MQ})(k,\eta)|\mathrm{d}\xi\mathrm{d}\eta\\
		\lesssim&\ \|U^1_0\|_{H^N}\|\nabla_LMZ\|_{L^2}\|\Gamma MQ\|_{L^2}.
	\end{align*}
	To estimate $I^2_3$ and $I^2_4$, proceeding analogously to the treatment of $I^1_3$ and $I^1_4$, we have
	\begin{align*}
		I^2_3+I^2_4\lesssim&\|MZ\|_{L^2}\|\nabla_LMZ\|_{L^2}\|\Gamma MQ\|_{L^2}+\|\Gamma MZ\|_{L^2}\|\nabla_LMZ\|_{L^2}\|MQ\|_{L^2}\\
		+&\|\Gamma MZ\|_{L^2}\|U^1_0\|_{H^N}\|\Gamma MQ\|_{L^2}+\|MZ\|\frac{1}{\langle t\rangle}\|\partial_Y\Omega_0\|_{H^N}\|\Gamma MQ\|_{L^2}.
	\end{align*}
	Then by the bootstrap hypotheses, it holds that
	\begin{equation}\label{estMsym}
		\int_{t_0}^{t}N\mathrm{d}\tau\lesssim\nu^{-\frac{1}{2}}\epsilon^3.
	\end{equation}
	\subsubsection{Estimates on $\widetilde{T}$}\label{esttt}
	Here again, we only estimate the first term of \eqref{omegtran}, i.e., $$I^3=Re\left\langle M(U\cdot\nabla_{L}\Omega),M\Omega\right\rangle.$$
	Using the cancellation
	\begin{equation*}
		Re\left\langle U\cdot\nabla_{L}M\Omega,M\Omega\right\rangle=0,
	\end{equation*}
	we divide $I^3$ into three parts:
	\begin{align*}
		I^3=&\left\langle\left[M,U\cdot\nabla_L\right]\Omega,M\Omega\right\rangle\\
		=&\left\langle\left[M,U^1_{\neq}\partial_{X}\right]\Omega_{\neq},M\Omega\right\rangle
		+\left\langle\left[M,U^1_0\partial_{X}\right]\Omega_{\neq},M\Omega\right\rangle+\left\langle\left[M,U^2_{\neq}\partial_Y^L\right]\Omega,M\Omega\right\rangle
		=:\sum_{i=1}^{3}I^3_i.
	\end{align*}
	For $I^3_1,I^3_2$, we use the commutator estimate as before. Compared to $I^1$, it is not necessary to split $I^3_3$ into the nonzero and zero frequency. In fact, 
	we now exploit $\frac{1}{\langle t\rangle}$ to control the growth of $\|\Omega\|_{L^2}$ which contains the problematic $\|\Omega_0\|_{L^2}$.
	
	\noindent\textbf{Treatment of} $I^3_1$
	
	\noindent  Similar as $I^1_{13}$, we decompose $I^3_1$ into two parts:
	\begin{align*}
		I^3_{11}:=&\sum_{(k,l)\in R_1}\iint\left|M(k,\eta)-M(l,\xi)\right|\left|\hat{U}^1_{\neq}(k-l,\eta-\xi)\right|\left|l\hat{\Omega}_{\neq}(l,\xi)\right|\left|M(k,\eta)\hat{\Omega}(k,\eta)\right|\mathrm{d}\xi\mathrm{d}\eta,\\
		I^3_{12}:=&\sum_{(k,l)\in R_2}\iint\left|M(k,\eta)-M(l,\xi)\right|\left|\hat{U}^1_{\neq}(k-l,\eta-\xi)\right|\left|l\hat{\Omega}_{\neq}(l,\xi)\right|\left|M(k,\eta)\hat{\Omega}(k,\eta)\right|\mathrm{d}\xi\mathrm{d}\eta,
	\end{align*}
	where $R_1,R_2$ are defined by \eqref{part}.
	To treat $I^3_{11}$, by \eqref{bdu1} and \eqref{estl}, we get
	\begin{align*}
		I^3_{11}=&\sum_{(k,l)\in R_1}\iint\left|M(k,\eta)-M(l,\xi)\right|\left|\hat{U}^1_{\neq}(k-l,\eta-\xi)\right|\left|l\hat{\Omega}_{\neq}(l,\xi)\right|\left|M(k,\eta)\hat{\Omega}(k,\eta)\right|\mathrm{d}\xi\mathrm{d}\eta\\
		\lesssim&\sum_{(k,l)\in R_1}\iint\left|M(k-l,\eta-\xi)\hat{U}^1_{\neq}(k-l,\eta-\xi)\right|\left|l\hat{\Omega}_{\neq}(l,\xi)\right|\left|M(k,\eta)\hat{\Omega}(k,\eta)\right|\\
		&+\left|\hat{U}^1_{\neq}(k-l,\eta-\xi)\right|\left(|k-l|+\nu^{-\frac{1}{6}}|l|^{\frac{1}{3}}|k|^{\frac{1}{3}}\right)\left|M(l,\xi)\hat{\Omega}_{\neq}(l,\xi)\right|\left|M(k,\eta)\hat{\Omega}(k,\eta)\right|\mathrm{d}\xi\mathrm{d}\eta\\
		\lesssim&\ \|M\Gamma\Omega\|_{L^2}e^{-\delta_0\nu^{\frac{1}{3}}t}\|M\Omega_{\neq}\|_{L^2}\|M\Omega\|_{L^2}+\nu^{-\frac{1}{6}}\|MZ\|_{L^2}\||\partial_{X}^{\frac{1}{3}}|M\Omega\|^2_{L^2}\\
		\lesssim&\ \|M\Gamma\Omega\|_{L^2}e^{-\delta_0\nu^{\frac{1}{3}}t}\nu^{-\frac{1}{6}}\nu^{\frac{1}{6}}\||\partial_{X}^{\frac{1}{3}}|M\Omega\|_{L^2}\|M\Omega\|_{L^2}+\nu^{-\frac{1}{6}}\|MZ\|_{L^2}\||\partial_{X}^{\frac{1}{3}}|M\Omega\|^2_{L^2}.
	\end{align*}
	By the bootstrap hypotheses, it holds that $\|M\Omega\|_{L^2}\lesssim\langle t\rangle\epsilon$.  
	Then we deduce
	\begin{equation*}
		\int_{t_0}^{t} I^3_{11}\mathrm{d}\tau\lesssim\nu^{-\frac{1}{2}}\epsilon^3\langle t\rangle^2.
	\end{equation*}
	For $I^3_{12}$, we can proceed analogously to the treatment of $I^1_{132}$ with $\hat{\Omega}$ playing the role of $\hat{Z}$. We omit further details and conclude the result
	\begin{align*}
		I^3_{12}\lesssim&e^{-\delta_0\nu^{\frac{1}{3}}t}\langle t\rangle\nu^{\frac{1}{2}}\|M\Omega_{\neq}\|^3_{L^2}+e^{-\delta_0\nu^{\frac{1}{3}}t}\langle t\rangle\nu^{\frac{1}{2}}\|MZ\|_{L^2}\|M\Omega_{\neq}\|_{L^2}\|M\nabla_L\Omega\|_{L^2}\\
		&+\|MZ\|_{L^2}\|M\Omega_{\neq}\|^2_{L^2}+\nu^{\frac{1}{3}}\|MZ\|_{L^2}\|\partial_X|^{\frac{1}{3}}M\Omega\|^2_{L^2}.
	\end{align*}
	Using bootstrap hypotheses, we have
	\begin{equation*}
		\int_{t_0}^{t} I^3_{12}\mathrm{d}\tau\lesssim\nu^{-\frac{1}{2}}\epsilon^3\langle t\rangle^2.
	\end{equation*}
	Furthermore, we conclude that
	\begin{equation}\label{estT31}
		\int_{t_0}^{t} I^3_{1}\mathrm{d}\tau\lesssim\nu^{-\frac{1}{2}}\epsilon^3\langle t\rangle^2.
	\end{equation}
	\textbf{Treatment of} $I^3_2$
	
	\noindent The estimate of $I^3_2$ is similar to $I^1_{221}$, hence we omit the details. Thanks to \eqref{commeq}, we have
	\begin{align*}
		I^3_2=&\sum_{k\neq0}\iint\left|M(k,\eta)-M(k,\xi)\right|\left|\hat{U}^1_0(\eta-\xi)\right|\left|k\hat{\Omega}_{\neq}(k,\xi)\right|\left|M\hat{\Omega}(k,\eta)\right|\mathrm{d}\xi\mathrm{d}\eta\\
		\lesssim&\ \|U^1_0\|_{H^N}\|M\Omega_{\neq}\|^2_{L^2}+\nu^{\frac{1}{3}}\|U^1_0\|_{H^N}\||\partial_X|^{\frac{1}{3}}M\Omega\|^2_{L^2}.
	\end{align*}
	By the bootstrap hypotheses, it holds that
	\begin{equation}\label{estT32}
		\int_{t_0}^{t} I^3_{2}\mathrm{d}\tau\lesssim\nu^{-\frac{1}{3}}\epsilon^3\langle t\rangle^2+\epsilon^3\langle t\rangle^2\leq\nu^{-\frac{1}{2}}\epsilon^3\langle t\rangle^2.
	\end{equation}
	\textbf{Treatment of} $I^3_3$
	
	\noindent Using \eqref{bdm}, we split $I^3_3$ into three parts:
	\begin{align*}
		I^3_3=&\left|\sum_{k,l}\iint\left[M(k,\eta)-M(l,\xi)\right]\hat{U}^2_{\neq}(k-l,\eta-\xi) i(\xi-lt)\hat{\Omega}(l,\xi) M(k,\eta)\bar{\hat{\Omega}}(k,\eta)\mathrm{d}\xi\mathrm{d}\eta\right|\\
		\lesssim& \sum_{k,l}\iint\left|M\hat{U}^2_{\neq}(k-l,\eta-\xi)\right|\left|(\xi-lt)\hat{\Omega}_{\neq}(l,\xi)\right|\left|M(k,\eta)\hat{\Omega}(k,\eta)\right|\mathrm{d}\xi\mathrm{d}\eta\\
		&+\sum_{k}\iint\left|M\hat{U}^2_{\neq}(k,\eta-\xi)\right|\left|\xi\hat{\Omega}_0(\xi)\right|\left|M(k,\eta)\hat{\Omega}(k,\eta)\right|\mathrm{d}\xi\mathrm{d}\eta\\
		&+\sum_{k,l}\iint\left|e^{\delta_0\nu^{\frac{1}{3}}t}\hat{U}^2_{\neq}(k-l,\eta-\xi)\right| \left|(\xi-lt)M(l,\xi)\hat{\Omega}(l,\xi)\right| \left|M(k,\eta)\hat{\Omega}(k,\eta)\right|\mathrm{d}\xi\mathrm{d}\eta=:\sum_{i=1}^{3}I^3_{3i}.
	\end{align*}
	We now treat $I^3_{31}$. Using Cauchy-Schwartz, Young's inequality, and \eqref{bdu2}, we have 
	\begin{align*}
		I^3_{31}=&\sum_{k,l}\iint\left|M\hat{U}^2_{\neq}(k-l,\eta-\xi)\right|\left|(\xi-lt)\hat{\Omega}_{\neq}(l,\xi)\right|\left|M(k,\eta)\hat{\Omega}(k,\eta)\right|\mathrm{d}\xi\mathrm{d}\eta\\
		\lesssim&\|M\Gamma Z\|_{L^2}e^{-\delta_0\nu^{\frac{1}{3}}t}\langle t\rangle^3\frac{\|M\Omega\|_{L^2}^2}{\langle t\rangle^2}\\
		\lesssim&\nu^{-\frac{1}{2}}\|M\Gamma Z\|_{L^2}\langle t\rangle^{\frac{3}{2}}\frac{\|M\Omega\|_{L^2}^2}{\langle t\rangle^2}.
	\end{align*}
	Similarly, by the fact $\Omega_0=\partial_YU^1_0$, we get
	\begin{align*}
		I^3_{32}=&\sum_{k}\iint\left|M\hat{U}^2_{\neq}(k,\eta-\xi)\right|\left|\xi\hat{\Omega}_0(\xi)\right|\left|M(k,\eta)\hat{\Omega}(k,\eta)\right|\mathrm{d}\xi\mathrm{d}\eta\\
		\lesssim&\|M\Gamma Z\|_{L^2}\|U^1_0\|_{H^N}\langle t\rangle \frac{\|M\Omega\|_{L^2}}{\langle t\rangle}.
	\end{align*}
	Then by the bootstrap hypotheses and H$\ddot{\mathrm{o}}$lder's inequality, it holds that
	\begin{equation*}
		\int_{t_0}^{t}I^3_{31}\lesssim\nu^{-\frac{1}{2}}\epsilon^2\left(\int_{t_0}^{t}\|M\Gamma Z\|^2_{L^2}\mathrm{d}\tau\right)^\frac{1}{2}\left(\int_{t_0}^{t}\langle \tau\rangle^3\mathrm{d}\tau\right)^\frac{1}{2}\lesssim\nu^{-\frac{1}{2}}\epsilon^3\langle t\rangle^2,
	\end{equation*}
	and 
	\begin{equation*}
		\int_{t_0}^{t}I^3_{32}\lesssim\epsilon^2\left(\int_{t_0}^{t}\|M\Gamma Z\|^2_{L^2}\mathrm{d}\tau\right)^\frac{1}{2}\left(\int_{t_0}^{t}\langle \tau\rangle^2\mathrm{d}\tau\right)^\frac{1}{2}\lesssim\nu^{-\frac{1}{2}}\epsilon^3\langle t\rangle^\frac{3}{2}.
	\end{equation*}
	For $I^3_{33}$, using \eqref{bdu2t}, we deduce
	\begin{align*}
		I^3_{33}=&\sum_{k,l}\iint\left|e^{\delta_0\nu^{\frac{1}{3}}t}\hat{U}^2_{\neq}(k-l,\eta-\xi)\right| \left|(\xi-lt)M(l,\xi)\hat{\Omega}(l,\xi)\right| \left|M(k,\eta)\hat{\Omega}(k,\eta)\right|\mathrm{d}\xi\mathrm{d}\eta\\
		\lesssim&\sum_{k,l}\iint\left|e^{\delta_0\nu^{\frac{1}{3}}t}\frac{1}{\langle t\rangle}\langle|k-l,\eta-\xi|\rangle(\widehat{\Gamma\Omega_{\neq}})(k-l,\eta-\xi)\right| \left|(\xi-lt)(\widehat{M\Omega})(l,\xi)\right| \left|(\widehat{M\Omega})(k,\eta)\right|\mathrm{d}\xi\mathrm{d}\eta\\
		\lesssim&\frac{1}{\langle t\rangle}\|\Gamma M\Omega\|_{L^2}\|\nabla_LM\Omega\|_{L^2}\|M\Omega\|_{L^2}.
	\end{align*}
	Thus the bootstrap hypotheses imply that
	\begin{equation}\label{estT33}
		\int_{t_0}^{t} I^3_{3}\mathrm{d}\tau\lesssim\nu^{-\frac{1}{2}}\epsilon^3\langle t\rangle^2.
	\end{equation}
	From the bounds \eqref{estT31}-\eqref{estT33}, we conclude that
	\begin{equation}\label{estTho}
		\int_{t_0}^{t}\widetilde{T}\mathrm{d}\tau\lesssim\nu^{-\frac{1}{2}}\epsilon^3\langle t\rangle^2.
	\end{equation}
	\subsubsection{Estimates on $\widetilde{S}$}
	 Here again, we only consider the first terms of \eqref{omegstr} and define 
	 \begin{equation*}
	 	I^4=\left|\left\langle M\left(\left(\frac{2\partial_{XY}^L}{\Delta_L} J\right) \left(\Omega-\frac{2\partial_{XX}}{\Delta_L} \Omega\right)\right), M J\right\rangle\right|.
	 \end{equation*}
	We have the following decomposition
	\begin{align*}
		I^4=&\left|\sum_{k,l}\iint M(k,\eta)\left(\frac{\partial_tp}{p}\hat{J}\right)(k-l,\eta-\xi)\left(1-\frac{2l^2}{p(l,\xi)}\right)\hat{\Omega}(l,\xi) M(k,\eta)\bar{\hat{J}}(k,\eta)\mathrm{d}\xi\mathrm{d}\eta\right|\\
		\lesssim&\sum_{k,l}\iint M(k,\eta)\left|\left(\frac{\partial_tp}{p}\hat{J}\right)(k-l,\eta-\xi)\right||\hat{\Omega}_{\neq}(l,\xi)||M(k,\eta)\hat{J}(k,\eta)|\mathrm{d}\xi\mathrm{d}\eta\\
		&+\sum_{k}\iint M(k,\eta)\left|\left(\frac{\partial_tp}{p}\hat{J}\right)(k,\eta-\xi)\right||\hat{\Omega}_0(\xi)|| M(k,\eta)\hat{J}(k,\eta)|\mathrm{d}\xi\mathrm{d}\eta=:I^4_1+I^4_2.
	\end{align*}
	By \eqref{bdm} and the fact
	\begin{equation*}
		\left|\frac{\partial_tp}{p}(k-l,\eta-\xi)\right|\lesssim\Gamma(k-l,\eta-\xi),
	\end{equation*}
	we have
	\begin{align*}
		I^4_1\lesssim e^{-\delta_0\nu^{\frac{1}{3}}t}\|\Gamma MJ\|_{L^2}\|M\Omega_{\neq}\|_{L^2}\|MJ\|_{L^2}.
	\end{align*}
	For $I^4_2$, noting the fact that
	\begin{equation*}
		\left|\frac{\partial_tp}{p}(k,\eta-\xi)\right|\lesssim\frac{|k|}{\sqrt{p(k,\eta-\xi)}}\leq\frac{1}{\langle t\rangle}\langle|k,\eta-\xi|\rangle,
	\end{equation*}
	we get
	\begin{align*}
		I^4_2=&\sum_{k\neq0}\iint M(k,\eta)\left|\left(\frac{\partial_tp}{p}\hat{J}\right)(k,\eta-\xi)\right||\hat{\Omega}_0(\xi)|| M(k,\eta)\hat{J}(k,\eta)|\mathrm{d}\xi\mathrm{d}\eta\\
		\lesssim&\sum_{k\neq0}\iint \left|\frac{|k|}{\sqrt{p(k,\eta-\xi)}}M(k,\eta-\xi)\hat{J}(k,\eta-\xi)\right||\hat{\Omega}_0(\xi)|| M(k,\eta)\hat{J}(k,\eta)|\mathrm{d}\xi\mathrm{d}\eta\\
		&+\sum_{k\neq0}\iint  e^{\delta_0\nu^{\frac{1}{3}}t}\left|\frac{|k|}{\sqrt{p(k,\eta-\xi)}}\hat{J}(k,\eta-\xi)\right||\langle|\xi|\rangle^N\hat{\Omega}_0(\xi)|| M(k,\eta)\hat{J}(k,\eta)|\mathrm{d}\xi\mathrm{d}\eta\\
		\lesssim&\sum_{k\neq0}\iint \left|\Gamma(k,\eta-\xi) M(k,\eta-\xi)\hat{J}(k,\eta-\xi)\right||\xi\hat{U}^1_0(\xi)|| M(k,\eta)\hat{J}_{\neq}(k,\eta)|\mathrm{d}\xi\mathrm{d}\eta\\
		&+\sum_{k\neq0}\iint  e^{\delta_0\nu^{\frac{1}{3}}t}\frac{1}{\langle t\rangle}\left|\langle|k,\eta-\xi|\rangle\hat{J}_{\neq}(k,\eta-\xi)\right||\langle|\xi|\rangle^N\hat{\Omega}_0(\xi)|| M(k,\eta)\hat{J}_{\neq}(k,\eta)|\mathrm{d}\xi\mathrm{d}\eta\\
		\lesssim&\ \|\Gamma MJ\|_{L^2}\|U^1_0\|_{H^N}\|MJ_{\neq}\|_{L^2}+\|MJ_{\neq}\|^2_{L^2}\frac{1}{\langle t\rangle}\|\Omega_0\|_{H^N}.
	\end{align*}
	Thus by the bootstrap hypotheses, we have
	\begin{equation}\label{estSho}
		\int_{t_0}^{t} \widetilde{S}\mathrm{d}\tau\lesssim\nu^{-\frac{1}{2}}\epsilon^3\langle t\rangle^2+\nu^{-\frac{1}{6}}\epsilon^3\langle t\rangle^2+\nu^{-\frac{1}{3}}\epsilon^3\langle t\rangle^2\lesssim\nu^{-\frac{1}{2}}\epsilon^3\langle t\rangle^2,
	\end{equation}
	and complete the estimates of nonlinear terms for $0<\mu^3\leq\nu\leq\mu\leq1$.
	\subsubsection{Conclusion of bootstrap argument}
	Now we improve the bootstrap hypotheses \eqref{bs}.
	When $0<\mu\leq\nu\leq\mu^{\frac{1}{3}}\leq1$, it holds $\lambda=\mu$, and we simply exchange the role of $\nu$ and $\mu$ in the above estimates. In fact, to get \eqref{bssym}, by \eqref{estTsym}-\eqref{estMsym} and Lemma \ref{st}, integrating the energy inequality \eqref{sym} over $[t_0,t]$, we obtain
	\begin{equation*}
		\overline{E}(t)+\frac{1}{100|\beta|}\int_{t_0}^{t}\overline{D}(\tau)+\overline{CK}(\tau)\mathrm{d}\tau\leq2\epsilon^2+ C(\epsilon\lambda^{-\frac{1}{2}})\epsilon^2\leq(2+C\epsilon_0)\epsilon^2\leq5\epsilon^2,
	\end{equation*}
	with $\epsilon_0\leq\frac{1}{C}$. For \eqref{bsho}, applying \eqref{estTho}, \eqref{estSho}, Lemma \ref{st}, and bootstrap hypotheses, integrating \eqref{ho} in time and picking $\epsilon_0$ sufficiently small, we have
		\begin{equation*}
		\widetilde{E}(t)+\frac{1}{100|\beta|}\int_{t_0}^{t}\widetilde{D}(\tau)+\widetilde{CK}(\tau)\mathrm{d}\tau\leq4\int_{t_0}^{t}\sqrt{\overline{E}}\sqrt{\widetilde{E}}\mathrm{d}\tau+2\epsilon^2\langle t\rangle^2\leq402\epsilon^2\langle t\rangle^2\leq500\epsilon^2\langle t\rangle^2.
		\end{equation*}
	Thus we improve the bootstrap hypothesis \eqref{bsho}. Finally, by the estimate of $\eqref{zeror}$, we integrate \eqref{zero} in time and choose $\epsilon_0\ll1$,
	\begin{equation*}
		E_{0}(t)+\int_{t_0}^{t}D_{0}(\tau)\mathrm{d}\tau\leq2\epsilon^2+ C(\epsilon\lambda^{-\frac{1}{2}})\epsilon^2\leq5\epsilon^2,
	\end{equation*}
	improving \eqref{bszero}. Whence we complete the proof of Proposition \ref{bsup1}. $\hfill\qedsymbol$
	
	\subsection{Proof of bootstrap in the regime $\mu^{\frac{1}{3}}\leq\nu$}\label{s44}
	In the regime $0<\mu^{\frac{1}{3}}\leq\nu\leq1$, it will be useful to define some shorthands for the various terms in \eqref{tran}-\eqref{zeror}. For concreteness, in \eqref{tran}, we denote
	\begin{equation*}
		\overline{T}[FGH]:=Re\langle\Gamma M(F\cdot\nabla_LG),MH\rangle,
	\end{equation*}
	where $FGH\in\{U\Omega Z,UJQ,BJZ,B\Omega Q\}$. We can define the shorthands for \eqref{trann}-\eqref{omegstr} and \eqref{zeror} similarly. In fact, due to the fact that $M_5$ is not bounded from below by any constant, we need to handle each term differently. For the sake of simplicity, we only show the estimate of several typical terms and one can treat the other terms in analogous ways.  
	\subsubsection{Estimates on $\overline{T}$} \label{bsset2}
	Using the fact that
	\begin{align*}
		&\nabla^{\perp}_L\cdot(B\cdot \nabla_LB)=B\cdot\nabla_LJ,\\
		&\nabla^{\perp}_L\cdot(B\cdot \nabla_LU)=B\cdot\nabla_L\Omega-\left(\frac{2\partial^L_{XY}}{\Delta_L}\Omega\right)\left(J-\frac{\partial_{XX}}{\Delta_L}J\right),\\
		&\nabla^{\perp}_L\cdot \left(MZ\right)=\nabla^{\perp}_L\cdot \left(MQ\right)=0.
	\end{align*}
	and integration by part, we obtain
	\begin{align*}
		\overline{T}[BJZ]+\overline{T}[B\Omega Q]
		=&Re\left(\left\langle\Gamma M(B\cdot\nabla_{L}J),MZ\right\rangle+\left\langle\Gamma M(B\cdot\nabla_{L}\Omega),MQ\right\rangle\right)\\
		=&Re\left(\left\langle\Gamma M\nabla^{\perp}_L\cdot(B\cdot \nabla_LB),\Gamma M\nabla_L^{\perp}\cdot U\right\rangle+\left\langle\Gamma M\nabla^{\perp}_L\cdot(B\cdot \nabla_LU),\Gamma M\nabla_L^{\perp}\cdot B\right\rangle\right)\\
		&+Re\left(\left\langle\Gamma M\left(\left(\frac{2\partial^L_{XY}}{\Delta_L} \Omega\right) \left(J- \frac{2\partial_{XX}}{\Delta_L} J\right)\right), MQ\right\rangle\right)\\
		=&\left(Re\langle\partial_XM(B\cdot\nabla_LB),\partial_XMU\rangle+Re\langle\partial_XM(B\cdot\nabla_LU),\partial_XMB\rangle\right)\\
		&+Re\left(\left\langle\Gamma M\left(\left(\frac{2\partial^L_{XY}}{\Delta_L} \Omega\right) \left(J- \frac{2\partial_{XX}}{\Delta_L} J\right)\right), MQ\right\rangle\right).
	\end{align*}
	The last term also appears in the expression for $\overline{S}$, which we will treat in subsection \ref{s443}. For the other two terms, noting the cancellation 
	\begin{equation*}
		\langle B\cdot\nabla_L(\partial_XMB\partial_XMU)\rangle=0,
	\end{equation*}
	we get the commutators
	\begin{align*}
		&\left(Re\langle\partial_XM(B\cdot\nabla_LB),\partial_XMU\rangle+Re\langle\partial_XM(B\cdot\nabla_LU),\partial_XMB\rangle\right)\\
		&=\left(Re\langle[\partial_XM,B\cdot\nabla_L]B,\partial_XMU\rangle+Re\langle[\partial_XM,B\cdot\nabla_L]U,\partial_XMB\rangle\right)=:\overline{T}[BBU]+\overline{T}[BUB].
	\end{align*}
	\textbf{Treatment of} $\overline{T}[BBU]$\\
	As before, we decompose $\overline{T}[BBU]$ into four parts
	\begin{align*}
		\overline{T}[BBU]=&Re\langle[\partial_XM,B\cdot\nabla_L]B,\partial_XMU\rangle\\
		=&\langle[\partial_XM,B^1_{\neq}\partial_X]B_{\neq},\partial_XMU\rangle+\langle[\partial_XM,B^1_0\partial_X]B_{\neq},\partial_XMU\rangle\\
		&+Re\langle[\partial_XM,B^2_{\neq}\partial^L_Y]B_{\neq},\partial_XMU\rangle+Re\langle[\partial_XM,B^2_{\neq}\partial^L_Y]B_0,\partial_XMU\rangle=:T_1+T_2+T_3+T_4.
	\end{align*}
	For $T_1$, by Lemma \ref{indap}, we deduce
	\begin{align*}
		T_1=&\left|\sum_{k,l}\iint[kM(k,\eta)-lM(l,\xi)]\hat{B}^1_{\neq}(k-l,\eta-\xi)l\hat{B}_{\neq}(l,\xi)kM(k,\eta)\bar{\hat{U}}(k,\eta)\mathrm{d}\xi\mathrm{d}\eta\right|\\
		\lesssim&\sum_{k,l}\iint A(k,\eta)|\langle|k-l,\eta-\xi|\rangle^N\hat{Q}(k-l,\eta-\xi)||\hat{Q}(l,\xi)||M(k,\eta)\hat{Z}(k,\eta)|\mathrm{d}\xi\mathrm{d}\eta\\
		&+\sum_{k,l}\iint A(k,\eta)|\hat{Q}(k-l,\eta-\xi)||\langle|l,\xi|\rangle^N\hat{Q}(l,\xi)||M(k,\eta)\hat{Z}(k,\eta)|\mathrm{d}\xi\mathrm{d}\eta\\
		&+\sum_{k,l}\iint|M(k,\eta)-M(l,\xi)||\hat{Q}(k-l,\eta-\xi)||l\hat{Q}(l,\xi)||M(k,\eta)\hat{Z}(k,\eta)|\mathrm{d}\xi\mathrm{d}\eta\\
		=&:T_{11}+T_{12}+T_{13}.
	\end{align*}
	By \eqref{216}, we get
	\begin{align}
		T_{11}\lesssim&\sum_{k,l}\iint \left|\frac{M(k-l,\eta-\xi)}{M_5(k-l,\eta-\xi)}\hat{Q}(k-l,\eta-\xi)\right|\left|\frac{1}{M_5(l,\xi)}M_5(l,\xi)\hat{Q}(l,\xi)\right||M(k,\eta)\hat{Z}(k,\eta)|\mathrm{d}\xi\mathrm{d}\eta\notag\\
		\lesssim&\sum_{k,l}\iint \left(1+\min\left\{\frac{\nu}{\Gamma(k-l,\eta-\xi)},\alpha\right\}\right)|(M\hat{Q})(k-l,\eta-\xi)|\left(1+\min\left\{\frac{\nu}{\Gamma(l,\xi)},\alpha\right\}\right)\notag\\
		&\times|M_5(l,\xi)\hat{Q}(l,\xi)||(M\hat{Z})(k,\eta)|\mathrm{d}\xi\mathrm{d}\eta=:T_{11}(1,1)+T_{11}(1,\alpha)+T_{11}(\alpha,1)+T_{11}(\alpha,\alpha),\label{alphal} 
	\end{align}
	where $T_{11}(\cdot,\cdot)$ are obtained by expanding 
	\begin{equation*}
		\left(1+\min\left\{\frac{\nu}{\Gamma(k-l,\eta-\xi)},\alpha\right\}\right)\left(1+\min\left\{\frac{\nu}{\Gamma(l,\xi)},\alpha\right\}\right).
	\end{equation*}
	Obviously we only need to consider $T_{11}(1,1)$ and $T_{11}(\alpha,\alpha)$. In fact, we have
	\begin{equation*}
		T_{11}(1,\alpha)+T_{11}(\alpha,1)\lesssim\alpha\,T_{11}(1,1).
	\end{equation*}
	Then by	$\Gamma(l,\xi)^{-1}\leq t|l,\xi|$,
	we get
	\begin{equation*}
		T_{11}(1,1)\leq\|MQ\|^2_{L^2}\|MZ\|_{L^2},
	\end{equation*}
	and
	\begin{align*}
		T_{11}(\alpha,\alpha)\lesssim& \ \alpha\|MQ\|_{L^2}e^{-\delta_0\mu^\frac{1}{3}t}t\|MQ\|_{L^2}\nu\|MZ\|_{L^2}\\
		\lesssim&\ \alpha\mu^{-\frac{1}{3}}\|MQ\|^2_{L^2}\nu\|MZ\|_{L^2}.
	\end{align*}
	Therefore, the bootstrap hypotheses imply that
	\begin{equation*}
		\int_{t_0}^{t}T_{11}\mathrm{d}\tau\lesssim\mu^{-\frac{1}{3}}\epsilon^3+\alpha\mu^{-\frac{1}{3}}\epsilon^3+\alpha\mu^{-\frac{1}{3}}\epsilon^3+\alpha\mu^{-\frac{1}{2}}\epsilon^3\lesssim\alpha\mu^{-\frac{1}{2}}\epsilon^3.
	\end{equation*}
	For $T_{12}$, by a same argument, we get 
	\begin{equation*}
		\int_{t_0}^{t}T_{12}\mathrm{d}\tau\lesssim\mu^{-\frac{1}{3}}\epsilon^3+\alpha\mu^{-\frac{1}{3}}\epsilon^3+\alpha\mu^{-\frac{1}{3}}\epsilon^3+\alpha\mu^{-\frac{1}{2}}\epsilon^3\lesssim\alpha\mu^{-\frac{1}{2}}\epsilon^3.
	\end{equation*}
	To consider $T_{13}$, using similar notation as \eqref{alphal}, we have
	\begin{equation*}
		T_{13}(\alpha,\alpha)\lesssim\alpha\|MQ\|_{L^2}\|\nabla_LMQ\|_{L^2}\nu\|MZ\|_{L^2}.
	\end{equation*}
	 $T_{13}(1,1),T_{13}(1,\alpha),T_{13}(\alpha,1)$ are treated analogously as $I^1_{13}$.
	By the bootstrap hypotheses, we deduce
	\begin{equation*}
		\int_{t_0}^{t}T_1\mathrm{d}\tau\lesssim\alpha\mu^{-\frac{1}{2}}\epsilon^3.
	\end{equation*}
	For $T_2$, which is treated similarly to $I^1_{221}$, we omit further details and arrive at
	\begin{equation*}
		T_2\lesssim\alpha\|B_0^1\|_{H^N}\|MQ\|_{L^2}\|MZ\|_{L^2}+\alpha\mu^{\frac{1}{3}}\|B^1_0\|_{H^N}\||\partial_X|^{\frac{1}{3}}|MQ\|_{L^2}\||\partial_X|^{\frac{1}{3}}|MZ\|_{L^2}
	\end{equation*}
	Turning to $T_3$, by Lemma \ref{indap}, we have 
	\begin{align*}
		T_3=&\left|\sum_{k,l}\iint[kM(k,\eta)-lM(l,\xi)]\hat{B}^2_{\neq}(k-l,\eta-\xi)(\xi-lt)\hat{B}_{\neq}(l,\xi)kM(k,\eta)\bar{\hat{U}}(k,\eta)\mathrm{d}\xi\mathrm{d}\eta\right|\\
		\lesssim&\sum_{k,l}\iint A(k,\eta)|\langle|k-l,\eta-\xi|\rangle^N(\Gamma\hat{Q})(k-l,\eta-\xi)||(\xi-lt)\hat{Q}(l,\xi)||M(k,\eta)\hat{Z}(k,\eta)|\mathrm{d}\xi\mathrm{d}\eta\\
		&+\sum_{k,l}\iint A(k,\eta)|(\Gamma\hat{Q})(k-l,\eta-\xi)||\langle|l,\xi|\rangle^N(\xi-lt)\hat{Q}(l,\xi)||M(k,\eta)\hat{Z}(k,\eta)|\mathrm{d}\xi\mathrm{d}\eta\\
		\lesssim&\ \nu\|MQ\|_{L^2}\alpha\|\nabla_LMQ\|_{L^2}\|MZ\|_{L^2}.
	\end{align*}
	Similarly, for $T_4$, we use Lemma \ref{indap} to get
		\begin{align*}
		T_4=&\left|\sum_{k}\iint kM(k,\eta)\hat{B}^2_{\neq}(k,\eta-\xi)\xi\hat{B}^1(0,\xi)kM(k,\eta)\bar{\hat{U}}^1(k,\eta)\mathrm{d}\xi\mathrm{d}\eta\right|\\
		\lesssim&\ \alpha\|\Gamma MQ\|_{L^2}\|B^1_0\|_{H^N}\|MZ\|_{L^2}+\alpha\|\Gamma MQ\|_{L^2}\|\partial_YB^1_0\|_{H^N}\|MZ\|_{L^2}.
	\end{align*}
	Therefore, by the bootstrap hypotheses, we conclude
	\begin{equation*}
		\int_{t_0}^{t}T[BBU]\mathrm{d}\tau\lesssim\alpha\mu^{-\frac{1}{2}}\epsilon^3.
	\end{equation*}
	\textbf{Treatment of} $\overline{T}[UJQ]$\\
	Using similar notation as \eqref{alphal}, the estimate of $\overline{T}[UJQ](1,1)$ is exact the same as that of $I^1$. Furthermore, the remaining (1,1)-type estimates in the section \ref{s44} can be treated as the corresponding terms before line by line. We only need to consider $(\alpha,\alpha)$-type estimates. As before, we decompose $\overline{T}[UJQ]$ into four parts:
	\begin{align*}
		\overline{T}[UJQ]=&\left\langle\left[M\Gamma,U\cdot\nabla_L\right]J,MQ\right\rangle\\
			=&\left\langle\left[M\Gamma,U^1_{\neq}\partial_{X}\right]J_{\neq},MQ\right\rangle
			+\left\langle\left[M\Gamma,U^1_0\partial_{X}\right]J_{\neq},MQ\right\rangle\\
			&+\left\langle\left[M\Gamma,U^2_{\neq}\partial_Y^L\right]J_{\neq},MQ\right\rangle
			+\left\langle\left[M\Gamma,U^2_{\neq}\partial_Y^L\right]J_{0},MQ\right\rangle\\
			=&:\overline{T}[U^1_{\neq}JQ]+\overline{T}[U^1_0JQ]+\overline{T}[U^2J_{\neq}Q]+\overline{T}[U^2J_{0}Q].
	\end{align*}
	We will frequently use similar notations without further explanations. By triangle inequality, we have
	\begin{align*}
			\overline{T}[U^1_{\neq}JQ]=&\left|\sum_{k,l\neq0}\iint\left[M(k,\eta)\Gamma(k,\eta)-M(l,\xi)\Gamma(l,\xi)\right]\hat{U}^1_{\neq}(k-l,\eta-\xi) il\hat{J}(l,\xi) M(k,\eta)\bar{\hat{Q}}(k,\eta)\mathrm{d}\xi\mathrm{d}\eta\right|\\
			\leq&\left|\sum_{k,l\neq0}\iint M(k,\eta)\left[\Gamma(k,\eta)-\Gamma(l,\xi)\right]\hat{U}^1_{\neq}(k-l,\eta-\xi) il\hat{J}(l,\xi) M(k,\eta)\bar{\hat{Q}}(k,\eta)\mathrm{d}\xi\mathrm{d}\eta\right|\\
			&+\left|\sum_{k,l\neq0}\iint\left[M(k,\eta)-M(l,\xi)\right]\Gamma(l,\xi)\hat{U}^1_{\neq}(k-l,\eta-\xi) il\hat{J}(l,\xi) M(k,\eta)\bar{\hat{Q}}(k,\eta)\mathrm{d}\xi\mathrm{d}\eta\right|\\
			\lesssim&\ e^{-\delta_0\mu^{\frac{1}{3}}t}\alpha\|\nabla_LMZ\|_{L^2}\nu\|MJ\|_{L^2}\|\Gamma MQ\|_{L^2}+\nu\|\nabla_LMZ\|_{L^2}\alpha\|MQ\|_{L^2}^2\\
			&+\alpha\|MZ\|_{L^2}\nu\|\nabla_LMQ\|_{L^2}\|MQ\|_{L^2}.
	\end{align*}
	The last three terms can be treated similarly as $I^1_2,I^1_3,I^1_4$, respectively. We omit further details to deduce
	\begin{align*}
		\overline{T}[U^1_0JQ]&\lesssim\alpha\left(\|U^1_0\|_{H^N}\||\partial_X|^{\frac{1}{3}}MQ\|_{L^2}^2+\|\partial_YU^1_0\|_{H^N}\|MQ\|_{L^2}\|\Gamma MQ\|_{L^2}\right),\\
		\overline{T}[U^2J_{\neq}Q]&\lesssim\nu\|\nabla_LMZ\|_{L^2}\alpha e^{-\delta_0\mu^{\frac{1}{3}}t}t\|MQ\|_{L^2}^2+\nu\|\nabla_LMZ\|_{L^2}\alpha\|\nabla_LMQ\|_{L^2}\|MQ\|_{L^2},
	\end{align*}
	and
	\begin{equation*}
		\overline{T}[U^2J_0Q]\lesssim\alpha\left(\|\Gamma MZ\|_{L^2}\|B^1_0\|_{H^N}\|\Gamma MQ\|_{L^2}+\|MZ\|_{L^2}\frac{1}{\langle t\rangle}\|\partial_YJ_0\|_{L^2}\|\Gamma MQ\|_{L^2}\right).
	\end{equation*}
	By the bootstrap hypotheses, we conclude
	\begin{equation*}
		\int_{t_0}^{t}\overline{T}[UJQ]\mathrm{d}\tau\lesssim\alpha\mu^{-\frac{1}{2}}\epsilon^3.
	\end{equation*}
	\subsubsection{Estimates on $\overline{N}$}
	\textbf{Treatment of} $\overline{N}[BJQ]$\\
	Compared to the estimate in the regime $\mu^3\leq\nu\leq\mu^{\frac{1}{3}}$, we need to control an additional $\Gamma^{-1}$. In fact, using integration by part, Biot-Savart law, and triangle inequality, we deduce
	\begin{align*}
		\overline{N}[B_{\neq}J_{\neq}Q]=&\left\langle \Gamma M(B_{\neq}\cdot\nabla_LJ_{\neq}),\frac{2\partial^L_Y}{\beta\Delta_L}MQ\right\rangle=\left\langle \partial_X M(B_{\neq}\cdot\nabla_LB_{\neq}),\frac{2\partial^L_Y}{\beta\Delta_L}\partial_XMB\right\rangle\\
		=&\Big|\sum_{k,l}\iint ik M(k,\eta)\left(\frac{(\eta-\xi-(k-l)t,-(k-l))}{p(k-l,\eta-\xi)}\hat{J}(k-l,\eta-\xi)\cdot(l,\xi-lt)\hat{B}(l,\xi)\right)\\
		&\times\frac{2(\eta-kt)}{-i\beta p(k,\eta)}ikM(k,\eta)\bar{\hat{B}}(k,\eta)\mathrm{d}\xi\mathrm{d}\eta\Big|\\
		=&\Big|\sum_{k,l}\iint ik M(k,\eta)\frac{\hat{J}(k-l,\eta-\xi)}{p(k-l,\eta-\xi)}(l,\xi-lt)\hat{B}(l,\xi)\cdot(\eta-kt,- k)\frac{2(\eta-kt)}{-i\beta p(k,\eta)}\\
		&\times ikM(k,\eta)\bar{\hat{B}}(k,\eta)\mathrm{d}\xi\mathrm{d}\eta\Big|\\
		\lesssim&\sum_{k,l}\iint M(k,\eta)|(\Gamma \hat{Q})(k-l,\eta-\xi)||l\hat{Q}(l,\xi)|| M(k,\eta)\hat{Q}(k,\eta)|\mathrm{d}\xi\mathrm{d}\eta\\
		&+\sum_{k,l}\iint M(k,\eta)|(\Gamma \hat{Q})(k-l,\eta-\xi)||(\xi-lt)\hat{Q}(l,\xi)|| \Gamma(k,\eta)M(k,\eta)\hat{Q}(k,\eta)|\mathrm{d}\xi\mathrm{d}\eta\\
		\lesssim&\ \alpha\|MQ\|^3_{L^2}+\alpha\|\Gamma MQ\|_{L^2}\|\nabla_LMQ\|_{L^2}\|MQ\|_{L^2}.
	\end{align*}
	For $\overline{N}[B_0J_{\neq}Q]$ and $\overline{N}[B_{\neq}J_0Q]$, the estimates are similar as $I^1_2$ and $I^1_4$ separately. Therefore, the bootstrap hypotheses imply that
	\begin{equation*}
		\int_{t_0}^{t}\overline{N}[BJQ]\mathrm{d}\tau\lesssim\alpha\mu^{-\frac{1}{2}}\epsilon^3.
	\end{equation*}
	\subsubsection{Estimates on $\overline{S}$}\label{s443}
	\textbf{Treatment of} $\overline{S}[\Omega JQ]$\\
	By the fact that 
	\begin{equation*}
		\left|\frac{\partial^L_{XY}}{\Delta_L}\hat{\Omega}(k,\eta)\right|\leq |\hat{Z}(k,\eta)|,
	\end{equation*}
	we deduce
	\begin{align*}
		\overline{S}[\Omega J_{\neq}Q]\lesssim&\sum_{k,l}\iint M(k,\eta)|\hat{Z}(k-l,\eta-\xi)||\hat{J}_{\neq}(l,\xi)||(\Gamma M\hat{Q})(k,\eta)|\mathrm{d}\xi\mathrm{d}\eta\\
		\lesssim&\ \nu\|\nabla_LMZ\|_{L^2}\alpha e^{-\delta_0\mu^{\frac{1}{3}}t}\langle t\rangle \frac{\|MJ_{\neq}\|_{L^2}}{\langle t\rangle}\|\Gamma MQ\|_{L^2}.
	\end{align*}
	Using the fact $J_0=\partial_YB^1_0$, we have
	\begin{equation*}
		\overline{S}[\Omega J_0 Q]\lesssim\nu\|\nabla_LMZ\|_{L^2}\|\partial_Y B_0^1\|_{H^N}\|\Gamma MQ\|_{L^2}.
	\end{equation*}
	 Then it holds that
	\begin{equation*}
		\int_{t_0}^{t}\overline{S}[\Omega JQ]\mathrm{d}\tau\lesssim\alpha\mu^{-\frac{1}{2}}\epsilon^3.
	\end{equation*}
	\textbf{Treatment of} $\overline{S}[J\Omega Q]$\\
	We use commutator estimate of $\Gamma$ \eqref{comsym} to get
	\begin{align*}
		\overline{S}[J\Omega_{\neq}Q]\lesssim&\sum_{k,l}\iint|\Gamma(k,\eta)-\Gamma(l,\xi)| M(k,\eta)|\hat{Q}(k-l,\eta-\xi)||\hat{\Omega}_{\neq}(l,\xi)||M(k,\eta)\hat{Q}(k,\eta)|\mathrm{d}\xi\mathrm{d}\eta\\
		&+\sum_{k,l}\iint M(k,\eta)|\hat{Q}(k-l,\eta-\xi)||\hat{Z}(l,\xi)||M(k,\eta)\hat{Q}(k,\eta)|\mathrm{d}\xi\mathrm{d}\eta\\
		\lesssim&\ \alpha\|\nabla_LMQ\|_{L^2}\nu\|\nabla_LMZ\|_{L^2}\|MQ\|_{L^2}.
	\end{align*}
	For the zero-mode, we have
	\begin{equation*}
		\overline{S}[J \Omega_0 Q]\lesssim\nu\|\nabla_LMQ\|_{L^2}\|\partial_Y U_0^1\|_{H^N}\|\Gamma MQ\|_{L^2}.
	\end{equation*}
	By the bootstrap hypotheses, we conclude
	\begin{equation*}
		\int_{t_0}^{t}\overline{S}[ J\Omega Q]\mathrm{d}\tau\lesssim\alpha\mu^{-\frac{1}{2}}\epsilon^3.
	\end{equation*}
	\subsubsection{Estimates on $\widetilde{T}$}
	 The estimates of $\widetilde{T}$ rely mainly on the triangle inequality and Lemma \ref{indap}, and we omit further details to present the results for the sake of simplicity. \\
	\textbf{Treatment of} $\widetilde{T}[BJ\Omega]$\\
		It is worth noting that
	\begin{equation*}
		\nu\|M\Omega\|_{L^2}\lesssim\nu\|\partial_XM\Omega_{\neq}\|_{L^2}+\nu\|\partial_YU^1_0\|_{H^N}.
	\end{equation*}
	For $\widetilde{T}[B_{\neq}J_{\neq}\Omega]$, by Lemma \ref{indap}, we have
	\begin{align*}
		\widetilde{T}[B_{\neq}J_{\neq}\Omega]\lesssim&\sum_{k,l}\iint|M(k,\eta)|\hat{B}^1_{\neq}(k-l,\eta-\xi)||l\hat{J}_{\neq}(l,\xi)||M(k,\eta)\hat{\Omega}(k,\eta)|\mathrm{d}\xi\mathrm{d}\eta\\
		&+\sum_{k,l}\iint|M(k,\eta)|\hat{B}^2_{\neq}(k-l,\eta-\xi)|(\xi-lt)\hat{J}_{\neq}(l,\xi)||M(k,\eta)\hat{\Omega}(k,\eta)|\mathrm{d}\xi\mathrm{d}\eta\\
		\lesssim&\ \alpha\|MQ\|_{L^2}\nu\|\nabla_LMJ\|_{L^2}\|M\Omega\|_{L^2}+\ \nu\|MQ\|_{L^2}\alpha\|\nabla_LMJ\|_{L^2}\|M\Omega\|_{L^2}.
	\end{align*}
	The estimates of $\widetilde{T}[B_0J_{\neq}\Omega]$ and $\widetilde{T}[B_{\neq}J_0\Omega]$ are similar as $I^3_2$ and $I^3_4$. Therefore, the bootstrap hypotheses imply
	\begin{equation*}
		\int_{t_0}^{t}\widetilde{T}[ BJ\Omega ]\mathrm{d}\tau\lesssim\alpha\mu^{-\frac{1}{2}}\epsilon^3\langle t\rangle^2.
	\end{equation*}
	\textbf{Treatment of} $\widetilde{T}[UJJ]$\\
	For $\widetilde{T}[U^1_{\neq}J_{\neq}J]$, by triangle inequality and \eqref{bdu1}, we have
	\begin{align*}
		\widetilde{T}[U^1_{\neq}J_{\neq}J]\lesssim&\sum_{k,l}\iint M(k,\eta)|\hat{U}^1_{\neq}(k-l,\eta-\xi)||l\hat{J}_{\neq}(l,\xi)||M(k,\eta)\hat{J}(k,\eta)|\mathrm{d}\xi\mathrm{d}\eta\\
		\lesssim&\sum_{k,l}\iint A(k,\eta)|\langle |k-l,\eta-\xi|\rangle^N\hat{Z}_{\neq}(k-l,\eta-\xi)||l\hat{J}_{\neq}(l,\xi)||M(k,\eta)\hat{J}(k,\eta)|\mathrm{d}\xi\mathrm{d}\eta\\
		&+\sum_{k,l}\iint A(k,\eta)\frac{1}{\langle t\rangle}|\langle |k-l,\eta-\xi|\rangle\hat{\Omega}_{\neq}(k-l,\eta-\xi)||l\langle |l,\xi\rangle^N\hat{J}_{\neq}(l,\xi)||M(k,\eta)\hat{J}(k,\eta)|\mathrm{d}\xi\mathrm{d}\eta\\
		\lesssim&\ \nu\|\nabla_LMZ\|_{L^2}\alpha e^{-\delta_0\mu^{\frac{1}{3}}t}\langle t\rangle^2\frac{\|MJ\|^2_{L^2}}{\langle t\rangle^2}+\nu\|\nabla_LM\Omega\|_{L^2}\alpha\|\nabla_LMJ\|_{L^2}\frac{\|MJ\|_{L^2}}{\langle t\rangle}.
	\end{align*}
	Similarly, by \eqref{bdu1} and \eqref{bdu2t}, we obtain
	\begin{equation*}
		\widetilde{T}[U^2J_{\neq}J]\lesssim\ \nu\|MZ\|_{L^2}\alpha e^{-\delta_0\mu^{\frac{1}{3}}t}\langle t\rangle^3\frac{\|MJ\|^2_{L^2}}{\langle t\rangle^2}+\nu\|\nabla_LMZ\|_{L^2}\alpha\|\nabla_LMJ\|_{L^2}\frac{\|MJ\|_{L^2}}{\langle t\rangle}.
	\end{equation*}
	The estimates of $T[U_0J_{\neq}J]$ and $T[U_{\neq}J_0J]$ are analogous to $I^3_2$ and $I^3_3$. Therefore, by the bootstrap hypotheses and H$\ddot{\mathrm{o}}$lder's inequality, we deduce
	\begin{equation*}
		\int_{t_0}^{t}\widetilde{T}[ UJJ ]\mathrm{d}\tau\lesssim\alpha\mu^{-\frac{1}{2}}\epsilon^3\langle t\rangle^2.
	\end{equation*}
	\subsubsection{Estimates on $\widetilde{S}$}
	\textbf{Treatment of} $\widetilde{S}[\Omega JJ]$\\
	For $\widetilde{S}[\Omega J_{\neq}J]$, we have
	\begin{equation*}
		\widetilde{S}[\Omega J_{\neq}J]\lesssim\nu\|\nabla_LMZ\|_{L^2}\alpha e^{-\delta_0\mu^{\frac{1}{3}}t}\langle t\rangle^2\frac{\|MJ\|^2_{L^2}}{\langle t\rangle^2}+\nu\|\nabla_LM\Omega\|_{L^2}\alpha\|MJ_{\neq}\|_{L^2}\frac{\|MJ\|_{L^2}}{\langle t\rangle}.
	\end{equation*}
	The estimate of $\widetilde{S}[\Omega J_{0}J]$ is similar to $I^4_2$. Therefore, the bootstrap hypotheses and H$\ddot{\mathrm{o}}$lder's inequality imply that
	\begin{equation*}
		\int_{t_0}^{t}\widetilde{S}[ \Omega JJ ]\mathrm{d}\tau\lesssim\alpha\mu^{-\frac{1}{2}}\epsilon^3\langle t\rangle^2.
	\end{equation*}
	\subsubsection{Estimates on $R$}
	Thanks to Lemma \ref{indap}, the improvement of $R$ is standard. First we treat $R[B^2_{\neq}B^1_{\neq}U^1_0]$ and $R[U^2_{\neq}B^1_{\neq}B^1_0]$. It holds that
	\begin{equation*}
		R[B^2_{\neq}B^1_{\neq}U^1_0]\lesssim\mu^{-\frac{1}{6}}\nu\|MQ\|_{L^2}\alpha\mu^{\frac{1}{6}}\|MQ\|_{L^2}\|\partial_{Y}U^1_0\|_{H^N},
	\end{equation*}
	and
	\begin{equation*}
		R[U^2_{\neq}B^1_{\neq}B^1_0]\lesssim\mu^{-\frac{1}{2}}\nu\| MZ\|_{L^2}\alpha\|MQ\|_{L^2}\mu^{\frac{1}{2}}\|\partial_{Y}B^1_0\|_{H^N}.
	\end{equation*}
	Turing to $R[B^2_{\neq}J_{\neq}\Omega_0]$ and $R[U^2_{\neq}J_{\neq}J_0]$, we have
	\begin{equation*}
		R[B^2_{\neq}J_{\neq}\Omega_0]\lesssim\mu^{-\frac{1}{6}}\nu\mu^{\frac{1}{6}}\|MQ\|_{L^2}\alpha\frac{\|MJ\|_{L^2}}{\langle t\rangle}\frac{\|\partial_{Y}\Omega_0\|_{H^N}}{\langle t\rangle},
	\end{equation*}
	and
	\begin{equation*}
		R[U^2_{\neq}J_{\neq}J_0]\lesssim\mu^{-\frac{1}{2}}\nu\|MZ\|_{L^2}\alpha\frac{\|MJ\|_{L^2}}{\langle t\rangle}\frac{\mu^{\frac{1}{2}}\|\partial_{Y}J_0\|_{H^N}}{\langle t\rangle}.
	\end{equation*}
	For $R[B^1_{\neq}\Omega_{\neq}J_0]$, by the fact 
	\begin{equation*}
		|\hat{\Omega}_{\neq}(l,\xi)|\lesssim|(\widehat  {\nabla_LZ})(l,\xi)|,
	\end{equation*} we get
	\begin{equation*}
		R[B^1_{\neq}\Omega_{\neq}J_0]\lesssim\nu\frac{\|MJ_{\neq}\|_{L^2}}{\langle t\rangle}e^{-\delta_0\mu^{\frac{1}{3}}t}\alpha\|\nabla_LMZ\|_{L^2}\frac{\|J_0\|_{H^N}}{\langle t\rangle}.
	\end{equation*}
	Therefore, by the bootstrap hypotheses and H$\ddot{\mathrm{o}}$lder's inequality, we deduce
	\begin{equation*}
		\int_{t_0}^{t}R \mathrm{d}\tau\lesssim\alpha\mu^{-\frac{1}{2}}\epsilon^3.
	\end{equation*}
	Whence we complete the proof of Proposition \ref{bsup2}.
	\subsection{Proof of bootstrap in the regime $\nu\leq\mu^3$}
	The process is almost similar to that in section \ref{case1}, expect that we use a lower bound of $M^6$ \eqref{217}. We only give a brief proof.
	\subsubsection{Estimates on $T$}
	Recalling 
		\begin{align*}
			T=&Re\left\langle M(U\cdot\nabla_{L}\Omega)_{\neq},M\Omega_{\neq}\right\rangle+Re\left\langle M(U\cdot\nabla_{L}J)_{\neq},MJ_{\neq}\right\rangle\\
			&+Re\big(\left\langle M(B\cdot\nabla_{L}J)_{\neq},M\Omega_{\neq}\right\rangle+\left\langle M(B\cdot\nabla_{L}\Omega)_{\neq},MJ_{\neq}\right\rangle\big),
		\end{align*}
	we only estimate the first term 
	\begin{equation*}
		I^5=Re\left\langle M(U\cdot\nabla_{L}\Omega)_{\neq},M\Omega_{\neq}\right\rangle,
	\end{equation*} 
	and the other terms are treated similarly. Using the cancellation $Re\langle(U\cdot\nabla_LM\Omega_{\neq}),M\Omega_{\neq}\rangle=0$,
	we have
	\begin{align*}
		I^5=&Re\langle M(U\cdot\nabla_L\Omega)_{\neq},M\Omega_{\neq}\rangle-Re\langle(U\cdot\nabla_LM\Omega_{\neq}),M\Omega_{\neq}\rangle\\
		=&\left\langle\left[M,U^1_{\neq}\partial_{X}\right]\Omega_{\neq},M\Omega_{\neq}\right\rangle
		+\left\langle\left[M,U^1_0\partial_{X}\right]\Omega_{\neq},M\Omega_{\neq}\right\rangle\\
		&+\left\langle\left[M,U^2_{\neq}\partial_Y^L\right]\Omega_{\neq},M\Omega_{\neq}\right\rangle+\left\langle M(U^2_{\neq}\partial_Y\Omega_0),M\Omega_{\neq}\right\rangle=:I^5_1+I^5_2+I^5_3+I^5_4.
	\end{align*}
	\textbf{Treatment of} $I^5_1$\\
	By the decomposition \eqref{part}, we split $I^5_1$ into two parts $I^5_{11}$ and $I^5_{12}$. Then using \eqref{estl} and \eqref {217}, we obtain
	\begin{align*}
		I^5_{11}=&\sum_{(k,l)\in R_1}\iint\left|M(k,\eta)-M(l,\xi)\right|\left|\hat{U}^1_{\neq}(k-l,\eta-\xi)\right|\left|l\hat{\Omega}(l,\xi)\right|\left|M(k,\eta)\hat{\Omega}_{\neq}(k,\eta)\right|\mathrm{d}\xi\mathrm{d}\eta\\
		\lesssim&\sum_{k,l}\iint|\langle|k-l,\eta-\xi|\rangle^N(AM^6\hat{\Omega}_{\neq})(k-l,\eta-\xi)||l\hat{\Omega}_{\neq}(l,\xi)||M(k,\eta)\hat{\Omega}_{\neq}(k,\eta)|\mathrm{d}\xi\mathrm{d}\eta\\
		&+\sum_{k,l}\iint\left|\frac{(AM^6\hat{\Omega}_{\neq})(k-l,\eta-\xi)}{|k-l|}\right|(|k-l|+\nu^{-\frac{1}{6}}|l|^{\frac{1}{3}}|k|^{\frac{1}{3}})|\langle|l,\xi|\rangle^N\hat{\Omega}_{\neq}(l,\xi)||M(k,\eta)\hat{\Omega}_{\neq}(k,\eta)|\mathrm{d}\xi\mathrm{d}\eta\\
		\lesssim&\ \mu^{-\frac{1}{3}}\|M\Omega_{\neq}\|^3_{L^2}+\mu^{-\frac{1}{3}}\nu^{-\frac{1}{6}}\|M\Omega_{\neq}\|_{L^2}\||\partial_X|^{\frac{1}{3}}M\Omega\|^2_{L^2}.
	\end{align*}
	For $I^5_{12}$, using \eqref{comm} and \eqref{217}, we have
	\begin{align*}
		I^5_{12}=&\sum_{(k,l)\in R_2}\iint\left|M(k,\eta)-M(l,\xi)\right|\left|\hat{U}^1_{\neq}(k-l,\eta-\xi)\right|\left|l\hat{\Omega}(l,\xi)\right|\left|M(k,\eta)\hat{\Omega}_{\neq}(k,\eta)\right|\mathrm{d}\xi\mathrm{d}\eta\\
		\lesssim&\sum_{(k,l)\in R_2}\iint e^{\delta_0\nu^{\frac{1}{3}}t}\big(\langle|k-l,\eta-\xi|\rangle^N+\langle|l,\xi|\rangle^N\big)\big(\nu^{\frac{1}{3}}|k|^{-\frac{1}{3}}|\eta-\xi|+|l|^{-1}(\nu^{\frac{1}{2}}|\eta|+1)|k-l|\big)\\
		&\times\left|\hat{U}^1_{\neq}(k-l,\eta-\xi)\right|\left|l\hat{\Omega}(l,\xi)\right|\left|M(k,\eta)\hat{\Omega}_{\neq}(k,\eta)\right|\mathrm{d}\xi\mathrm{d}\eta\\
		\lesssim&\ \mu^{-\frac{1}{3}}t\nu^{\frac{1}{2}}\|\nabla_LM\Omega_{\neq}\|_{L^2}e^{-\delta_0\nu^{\frac{1}{3}}t}\|M\Omega_{\neq}\|_{L^2}^2+\mu^{-\frac{1}{3}}\|M\Omega_{\neq}\|^3_{L^2}+\mu^{-\frac{1}{3}}\|M\Omega_{\neq}\|_{L^2}\nu^{\frac{1}{3}}\||\partial_X|^{\frac{1}{3}}M\Omega\|^2_{L^2}
	\end{align*}
	Therefore, by the bootstrap hypotheses, we deduce
	\begin{equation*}
		\int_{t_0}^{t}I^5_1 \mathrm{d}\tau\lesssim\mu^{-\frac{1}{3}}\nu^{-\frac{1}{2}}\epsilon^3.
	\end{equation*}
	\textbf{Treatment of} $I^5_2$\\
	Similar to $I^3_2$, by \eqref{commeq} and \eqref{217}, we have
	\begin{align*}
		I^5_2\lesssim&\sum_{k\neq0}\iint\big[\big(|k|^{-1}+\nu^{\frac{1}{3}}|k|^{-\frac{1}{3}}\big)|\eta-\xi|\langle|k,\xi|\rangle^N+e^{\delta_0\nu^{\frac{1}{3}}t}\langle|k,\xi|\rangle^{N-1}|\eta-\xi|+e^{\delta_0\nu^{\frac{1}{3}}t}\langle|\eta-\xi|\rangle^{N}\big]\\
		&\times|\hat{U}^1_{0}(\eta-\xi)||k \hat{\Omega}(k,\xi)|| M(k,\eta)\hat{\Omega}_{\neq}(k,\eta)|\mathrm{d}\xi\mathrm{d}\eta\\
		\lesssim&\ \mu^{-\frac{1}{3}}\left(\ \|U^1_0\|_{H^N}\|M\Omega_{\neq}\|^2_{L^2}+\nu^{\frac{1}{3}}\|U^1_0\|_{H^N}\||\partial_X|^{\frac{1}{3}}M\Omega\|^2_{L^2}\right).
	\end{align*}
	\textbf{Treatment of} $I^5_3$\\
	By \eqref{bdu2} and \eqref{217}, we get
	\begin{align*}
		I^5_3\lesssim&\sum_{k,l}\iint e^{\delta_0\nu^{\frac{1}{3}}t}(\langle|k-l,\eta-\xi|\rangle^N+\langle l,\xi\rangle^N)|(\Gamma M^6\hat{\Omega}_{\neq})(k-l,\eta-\xi)|\\
		&\times|(\xi-lt)\hat{\Omega}_{\neq}(l,\xi)||M(k,\eta)\hat{\Omega}_{\neq}(k,\eta)|\mathrm{d}\xi\mathrm{d}\eta\\
		\lesssim&\ \|\Gamma M\Omega\|_{L^2}\mu^{-\frac{1}{3}}\|\nabla_LM\Omega_{\neq}\|_{L^2}\|M\Omega_{\neq}\|_{L^2}.
	\end{align*}
	\textbf{Treatment of} $I^5_4$\\
	Using \eqref{bdu2t} and \eqref{217}, we deduce
	\begin{align*}
		I^5_4=&\sum_{k}\iint M(k,\eta)\hat{U}^2(k,\eta-\xi)i\xi\hat{\Omega}(0,\xi)M(k,\eta)\hat{\Omega}_{\neq}(k,\eta)\mathrm{d}\xi\mathrm{d}\eta\\
		\lesssim&\ \|\Gamma M\Omega\|_{L^2}\|U^1_0\|_{H^N}\|M\Omega_{\neq}\|_{L^2}+\mu^{-\frac{1}{3}}\|\Gamma M\Omega\|_{L^2}\frac{\|\partial_Y\Omega_0\|_{H^N}}{\langle t\rangle}\|M\Omega_{\neq}\|_{L^2}.
	\end{align*}
	Therefore, by the bootstrap hypotheses, we conclude
	\begin{equation}\label{tiltilt}
		\int_{t_0}^{t}T \mathrm{d}\tau\lesssim\mu^{-\frac{1}{3}}\nu^{-\frac{1}{2}}\epsilon^3.
	\end{equation}
	\subsubsection{Estimates on $N$}
	By the fact
	\begin{equation*}
		\left|\frac{\partial_Y^L}{\Delta_L}\hat{J}_{\neq}(k,\eta)\right|\leq\left|\frac{1}{|k|}\Gamma(k,\eta)\hat{J}_{\neq}(k,\eta)\right|,
	\end{equation*}
	we have
	\begin{align*}
		N[U\Omega_{\neq}J_{\neq}]=&\left|\left\langle M(U\cdot\nabla_{L}\Omega_{\neq}),\frac{2\partial_Y^L}{\beta \Delta_L}MJ_{\neq}\right\rangle\right|\\
		\lesssim&\|M\Omega_{\neq}\|_{L^2}\mu^{-\frac{1}{3}}\|\nabla_LM\Omega_{\neq}\|_{L^2}\|\Gamma MJ\|_{L^2}+\|U^1_0\|_{H^N}\mu^{-\frac{1}{3}}\|M\Omega_{\neq}\|_{L^2}\|\Gamma MJ\|_{L^2}.
	\end{align*}
	Using \eqref{bdu2t} and \eqref{217}, we deduce
	\begin{align*}
		N[U\Omega_0J_{\neq}]=&\left|\left\langle M(U^2\partial_Y\Omega_0),\frac{2\partial_Y^L}{\beta \Delta_L}MJ_{\neq}\right\rangle\right|\\
		\lesssim&\|\Gamma M\Omega\|_{L^2}\|U^1_0\|_{H^N}\|\Gamma MJ\|_{L^2}+\| M\Omega\|_{L^2}\frac{\|\partial_Y\Omega_0\|_{H^N}}{\langle t\rangle}\|\Gamma MJ\|_{L^2}.
	\end{align*}
	Thus by bootstrap hypotheses, we conclude
	\begin{equation}\label{tiltiln}
		\int_{t_0}^{t}N \mathrm{d}\tau\lesssim\mu^{-\frac{1}{3}}\nu^{-\frac{1}{2}}\epsilon^3.
	\end{equation}
	\subsubsection{Estimates on $S$}
	Let 
	\begin{equation*}
		I^6:=\left|\left\langle M\left(\left(\frac{2\partial_{XY}^L}{\Delta_L} J\right)\Omega\right)_{\neq}, M J_{\neq}\right\rangle\right|,
	\end{equation*}
	by \eqref{217} and the fact that $\langle|k,\eta-kt|\rangle\langle|k,\eta|\rangle\gtrsim\langle|kt|\rangle$, we have
	\begin{equation*}
		I^6[J\Omega_{\neq}J_{\neq}]\leq\|MJ_{\neq}\|_{L^2}^2\mu^{-\frac{1}{3}}\|M\Omega_{\neq}\|_{L^2},
	\end{equation*}
	and
	\begin{equation*}
		I^6[J\Omega_0J_{\neq}]\leq\|MJ_{\neq}\|_{L^2}^2\|U^1_0\|_{H^N}+\mu^{-\frac{1}{3}}\|MJ_{\neq}\|_{L^2}^2\frac{\|\Omega_{0}\|_{H^N}}{\langle t\rangle}.
	\end{equation*}
	Thus the bootstrap hypotheses imply that
	\begin{equation}\label{tiltils}
		\int_{t_0}^{t}S \mathrm{d}\tau\lesssim\mu^{-\frac{1}{3}}\nu^{-\frac{1}{2}}\epsilon^3.
	\end{equation}
	Then we put	together \eqref{tiltilt}-\eqref{tiltils} to improve the bootstrap hypothesis \eqref{bsho2}.
	\subsubsection{Estimates on $R$}
	For $R[U^2_{\neq}U^1_{\neq}U^1_0]$, using Lemma \ref{indap} and \eqref{217}, we deduce
	\begin{equation*}
		R[U^2_{\neq}U^1_{\neq}U^1_0]\leq\|\Gamma M\Omega\|_{L^2}\|M\Omega_{\neq}\|_{L^2}\|\partial_YU^1_0\|_{H^N}.
	\end{equation*}
	Turning to $R[U^2_{\neq}\Omega_{\neq}\Omega_0]$, we have
	\begin{equation*}
		R[U^2_{\neq}\Omega_{\neq}\Omega_0]\lesssim\|\Gamma M\Omega\|_{L^2}\|M\Omega_{\neq}\|_{L^2}\frac{\|\partial_Y\Omega_0\|_{H^N}}{\langle t\rangle}.
	\end{equation*}
	Considering $R[B^1_{\neq}\Omega_{\neq}J_0]$, it holds that
	\begin{equation*}
		R[B^1_{\neq}\Omega_{\neq}J_0]\lesssim\|MJ_{\neq}\|_{L^2}\|M\Omega_{\neq}\|\frac{\|J_0\|_{H^N}}{\langle t\rangle}.
	\end{equation*}
	Then by the bootstrap hypothesis, we conclude
	\begin{equation*}
		\int_{t_0}^{t}R \mathrm{d}\tau\lesssim\nu^{-\frac{1}{2}}\epsilon^3,
	\end{equation*}
	which improve the bootstrap hypothesis \eqref{bszero2}. Whence we complete the proof of Proposition \ref{bsup3}.
	\section{Appendix}\label{append}
	In this section, we give a proof of Proposition \ref{linearest3}--\ref{linearest2}.
	\begin{proof}[Proof of Proposition \ref{linearest}]
		Due to the translation invariance, the linearized problem is best-studied mode-by-mode in $X,Y$. Taking the Fourier transform of equations \eqref{mequation} in $X,Y$ gives 
		\begin{equation*}
			\left\{
			\begin{array}{l}
				\displaystyle\partial_t(M_k\hat{Z}_k)-\beta ikM_k\hat{Q}_k+(\nu p_k-\frac{\partial_tM_k}{M_k})M_k\hat{Z}_k+\frac{1}{2}\frac{\partial_tp_k}{p_k}M_k\hat{Z}_k=0,\\
				\displaystyle\partial_t(M_k\hat{Q}_k)-\beta ikM_k\hat{Z}_k+(\mu p_k-\frac{\partial_{t}M_k}{M_k})M_k\hat{Q}_k-\frac{1}{2}\frac{\partial_tp_k}{p_k}M_k\hat{Q}_k=0,
			\end{array}\right.
		\end{equation*}
		where $p_k=k^2+(\eta-kt)^2$. Testing the above equations against $M\hat{Z}_k,M\hat{Q}_k$ separately, we obtain
		\begin{align*}
			\frac{1}{2}\frac{\rm{d}}{\rm{d}t}\lVert M_k\hat{Z}_k\rVert_{L^2}^2=&-\left(\nu \lVert \sqrt{p_k} M_k\hat{Z}_k\rVert^2+\sum_{i=1}^{4}\left\lVert\sqrt{-\frac{\partial_tM^i_k}{M^i_k}}M_k\hat{Z}_k\right\rVert_{L^2}^2\right)+\delta_0\lambda^{\frac{1}{3}}\lVert M_k\hat{Z}_k\rVert_{L^2}^2\notag\\
			&-\frac{1}{2}\left\langle\frac{\partial_tp_k}{p_k}M_k\hat{Z}_k,M_k\hat{Z}_k\right\rangle+Re\langle\beta ikM_k\hat{Q}_k,M_k\hat{Z}_k\rangle, \\
			\frac{1}{2}\frac{\rm{d}}{\rm{d}t}\lVert M_k\hat{Q}_k\rVert_{L^2}^2=&-\left(\mu \lVert \sqrt{p_k} M_k\hat{Q}_k\rVert^2+\sum_{i=1}^{4}\left\lVert\sqrt{-\frac{\partial_tM^i_k}{M^i_k}}M_k\hat{Z}_k\right\rVert_{L^2}^2\right)+\delta_0\lambda^{\frac{1}{3}}\lVert M_k\hat{Q}_k\rVert_{L^2}^2\notag\\
			&+\frac{1}{2}\left\langle\frac{\partial_tp_k}{p_k}M_k\hat{Q}_k,M_k\hat{Q}_k\right\rangle+Re\langle\beta ikM_k\hat{Z}_k,M_k\hat{Q}_k\rangle.
		\end{align*}
		For the mixed term, we have
		\begin{align*}
			-\frac{1}{2}\frac{\rm{d}}{\rm{d}t}Re\left\langle\frac{\partial_tp_k}{\beta ikp_k}\overline{\chi}_kM_k\hat{Z}_k,M_k\hat{Q}_k\right\rangle&=\frac{1}{2}\left\langle\frac{\partial_tp_k}{p_k}\overline{\chi}_kM_k\hat{Z}_k,M_k\hat{Z}_k\right\rangle-\frac{1}{2}\left\langle\frac{\partial_tp_k}{p_k}\overline{\chi}_kM_k\hat{Q}_k,M_k\hat{Q}_k\right\rangle\notag\\
			&+Re\left\langle\frac{\partial_tp_k}{2\beta ikp_k}\big((\nu+\mu)p_k-2\delta_0\lambda^{\frac{1}{3}}k^{\frac{2}{3}}-2\sum_{i=1}^{3}\frac{\partial_tM^i_k}{M^i_k}\big)\overline{\chi}_kM_k\hat{Z}_k,M_k\hat{Q}_k\right\rangle\notag\\
			&-\frac{1}{2}Re\left\langle\frac{p_k\partial_{tt}p_k-(\partial_tp_k)^2}{\beta ikp_k^2}\overline{\chi}_kM_k\hat{Z}_k,M_k\hat{Q}_k\right\rangle,
		\end{align*}
		where we used $\displaystyle \frac{\partial_tM_k}{M_k}=\delta_0\lambda^{\frac{1}{3}}+\sum_{i=1}^{3}\frac{\partial_tM^i_k}{M^i_k}$. Combining the above three energy identities, we arrive at the following inequality
		\begin{equation*}
			\frac{\rm{d}}{\mathrm{d}t}E_k+DF+CKF_1+CKF_2\leq\sum_{i=0}^{6}L_i,
		\end{equation*}
		where
		\begin{align*}
			&E_k=\frac{1}{2}\left(\lVert M_k\hat{Z}_k\rVert^2_{L^2}+\left\lVert 	M_k\hat{Q}_k\right\rVert^2_{L^2}-Re\left\langle\frac{\partial_tp_k}{|\beta|ikp_k}M_k\hat{Z}_k,M_k\hat{Q}_k\right\rangle \right),\\
			&DF=\nu \lVert \sqrt{p_k} M_k\hat{Z}_k\rVert^2_{L^2}+\mu \lVert \sqrt{p_k} M_k\hat{Q}_k\rVert^2_{L^2},\\
			&CKF_1=\sum_{F\in\{Z,Q\}}\sum_{i=1,2}\left\lVert\sqrt{-\frac{\partial_tM^i_k}{M^i_k}}M_k\hat{F}_k\right\rVert_{L^2}^2,\ CKF_2=\sum_{F\in\{Z,Q\}}\left\lVert\sqrt{-\frac{\partial_tM^3_k}{M^3_k}}M_k\hat{F}_k\right\rVert_{L^2}^2,
		\end{align*}
		and $L_0-L_6$ are given by
		\begin{equation}\label{error}
			\begin{aligned}
				&L_0:=\sum_{F\in\{Z,Q\}}\delta_0\lambda^{\frac{1}{3}}\left(\lVert M_k\hat{F}_k\rVert_{L^2}^2+\left\lVert\sqrt{\frac{\partial_tp_k}{2\beta ikp_k}}M_k\hat{F}_k\right\rVert^2_{L^2}\right),\\
				&L_1:=(\nu+\mu) Re\left\langle\frac{\partial_tp_k}{2\beta ik}\overline{\chi}_kM_k\hat{Z}_k,M_k\hat{Q}_k\right\rangle,\ 
				L_2:=\frac{1}{2}Re\left\langle\frac{p_k\partial_{tt}p_k-(\partial_tp_k)^2}{\beta ikp_k^2}\overline{\chi}_kM_k\hat{Z}_k,M_k\hat{Q}_k\right\rangle,\\
				&L_3:=Re\left\langle\frac{\partial_tp_k}{\beta ikp_k}\frac{\partial_tM^2_k}{M^2_k}\overline{\chi}_kM_k\hat{Z}_k,M_k\hat{Q}_k\right\rangle,\ \ 
				L_4:=Re\left\langle\frac{\partial_tp_k}{\beta ikp_k}\frac{\partial_tM^1_k}{M^1_k}\overline{\chi}_kM_k\hat{Z}_k,M_k\hat{Q}_k\right\rangle,\\
				&L_5:=Re\left\langle\frac{\partial_tp_k}{\beta ikp_k}\frac{\partial_tM^3_k}{M^3_k}\overline{\chi}_kM_k\hat{Z}_k,M_k\hat{Q}_k\right\rangle,\\
				&L_6:=-\frac{1}{2}\left\langle (1-\overline{\chi}_k)\frac{\partial_tp_k}{p_k}M_k\hat{Z}_k,M_k\hat{Z}_k\right\rangle+\frac{1}{2}\left\langle (1-\overline{\chi}_k)\frac{\partial_tp_k}{p_k}M_k\hat{Q}_k,M_k\hat{Q}_k\right\rangle.
			\end{aligned}
		\end{equation}
		\textbf{Case 1. $0<\mu^3\leq\nu\leq\mu\leq1$}\\ When $|k|\leq C_2\nu^{-1/2}$, using \eqref{213} and the assumption $\nu\leq\mu$, we get
		\begin{equation} \label{l0}
			L_0\leq \frac{2\delta_0}{C_2}CKF_1.
		\end{equation}
		For $L_1$, since $|\partial_tp_k|/|k|\leq2\sqrt{p_k}$, we have
		\begin{equation*}
			L_1\leq\frac{\nu}{10|\beta|}\left\lVert\sqrt{p_k}M_k\hat{Z}_k\right\rVert^2_{L^2}+\frac{10\mu}{|\beta|}\lVert M_k\hat{Z}_k\rVert^2_{L^2}+\frac{\mu}{10|\beta|}\left\lVert\sqrt{p_k}M_k\hat{Q}_k\right\rVert^2_{L^2}+\frac{10\nu}{|\beta|}\lVert M_k\hat{Q}_k\rVert^2_{L^2}.
		\end{equation*}
		Thanks to the assumption $0<\mu^3\leq\nu\leq\mu\leq1$ and \eqref{213}, it follows that
		\begin{equation}\label{l1}
			L_1\leq\left(\frac{1}{10|\beta|}+\frac{10}{C_2|\beta|}\right)(DF+CKF_1).
		\end{equation}
		We turn to $L_2$ now. Observing \begin{equation*}
			\left|\frac{p_k\partial_{tt}p_k-(\partial_tp_k)^2}{2\beta kp_k^2}\right|\leq\frac{|k|}{|\beta| p_k}\leq\frac{1}{|\beta |C_1}\left|\frac{\partial_tM_k^1}{M_k^1}\right|,
		\end{equation*}
		we deduce \begin{equation}\label{l2}
			L_2\leq\frac{1}{|\beta|C_1}CKF_1.
		\end{equation}
		For $L_3$, combining
		\begin{equation*} 
			\left|\frac{\partial_t p_k}{\beta k p_k}\right|\leq\frac{2}{|\beta|\sqrt{p_k}}\leq\frac{2}{|\beta|\sqrt{C_1}|k|}\sqrt{-\frac{\partial_t M_k^1}{M_k^1}}
		\end{equation*}
		with $|\partial_t M^2_k/M^2_k|\leq2|C_2|^2|k|^\frac{2}{3}$, we obtain
		\begin{equation}\label{l3}
			L_3\leq\frac{1}{10|\beta|}\left\lVert\sqrt{-\frac{\partial_tM^2_k}{M^2_k}}M_k\hat{Z}_k\right\rVert^2_{L^2}+\frac{20C_2}{|k|^\frac{4}{3}|\beta|C_1}\left\lVert\sqrt{-\frac{\partial_tM^1_k}{M^1_k}}M_k\hat{Q}_k\right\rVert^2_{L^2}\leq\left(\frac{1}{10|\beta|}+\frac{20C_2}{|\beta|C_1}\right)CKF_1.
		\end{equation}
		To control $L_4$, noticing
		\begin{equation*}\label{msmd}
			\sqrt{\frac{k^2}{p_k}}\frac{\partial_tM^1_k}{M^1_k}=\frac{C_1}{C_3}\frac{\partial_t M^3_k}{M^3_k},
		\end{equation*}
		we have
		\begin{equation}\label{l4}
			L_4\leq\left|\left\langle\frac{2C_1}{|\beta|C_3}\frac{\partial_tM^3_k}{M^3_k}M_k\hat{Z}_k,M_k\hat{Q}_k\right\rangle\right|\leq\frac{C_1}{|\beta|C_3}CKF_2.
		\end{equation}
		For $L_5$, it holds
		\begin{equation}\label{l5}
			L_5\leq\frac{1}{2|\beta|}CKF_2,
		\end{equation}
		where we used $|\partial_tp_k|/(|k\beta|p_k)\leq1/|\beta|$.\\
		Turning to $L_6$, by the fact
		\begin{align*}
			\left|\frac{\partial_tp_k}{2p_k}\right|\leq\frac{\nu}{50}p_k,\ \mathrm{if}\ |t-\frac{\eta}{k}|\geq4\nu^{-\frac{1}{3}},\\
			\left|\frac{\partial_tp_k}{2p_k}\right|\leq\frac{\mu}{50}p_k,\ \mathrm{if}\ |t-\frac{\eta}{k}|\geq4\mu^{-\frac{1}{3}},
		\end{align*}
		we get
		\begin{equation}\label{l6}
			L_6\leq\frac{\nu}{50}\left\lVert\sqrt{p_k}M_k\hat{Z}_k\right\rVert^2_{L^2}+\frac{\mu}{50}\left\lVert\sqrt{p_k}M_k\hat{Q}_k\right\rVert^2_{L^2}\leq\frac{1}{50}DF.
		\end{equation}
		Let 
		\begin{equation*}
			\begin{aligned}
				D_k=&\nu \lVert \sqrt{p_k} M_k\hat{Z}_k\rVert^2_{L^2}+\mu \lVert \sqrt{p_k} M_k\hat{Q}_k\rVert^2_{L^2}+\left\lVert\Gamma_kM_k\hat{Z}_k\right\rVert^2_{L^2}+\left\lVert\Gamma_kM_k\hat{Q}_k\right\rVert^2_{L^2}\\
				&+\nu^{\frac{1}{3}}|k|^{\frac{2}{3}}\lVert M_k\hat{Z}_k\rVert^2_{L^2}+\nu^{\frac{1}{3}}|k|^{\frac{2}{3}}\lVert M_k\hat{Q}_k\rVert^2_{L^2},
			\end{aligned}
		\end{equation*}
		then $D_k$ satisfies
		\begin{equation}\label{df}
			D_k\leq\left(1+\frac{1}{C_2}\right)(DF+CKF_1).
		\end{equation}
		Note $\delta_0,C_1,C_2$, and $C_3$ satisfy
		\begin{equation*}
			\delta_0\leq\frac{1}{100|\beta|},C_1>1000C_2,C_2>100,\ \mathrm{and}\ C_3>\frac{C_1}{|\beta|-\frac{1}{2}}.\ 
		\end{equation*}
		By \eqref{l0}-\eqref{df},
		we conclude
		\begin{equation}\label{klinearest}
			\frac{\rm{d}}{\mathrm{d}t}E_k+\frac{1}{100|\beta|}D_k\leq0.
		\end{equation}
		Then integrating it in time and summing over $k\in \mathbb{Z}^*$, we deduce \eqref{linest} for $|k|\leq C_2\nu^{-\frac{1}{2}}$.
		\noindent When $|k|>C_2\nu^{-\frac{1}{2}}$, \eqref{213} implies
		\begin{equation}\label{bdm2}
			\lVert M_k\hat{Z}_k\rVert^2_{L^2}+\lVert M_k\hat{Q}_k\rVert^2_{L^2}\leq\nu^{\frac{1}{3}}|k|^{\frac{2}{3}}C_2^{-\frac{2}{3}}\left(\lVert M_k\hat{Z}_k\rVert^2_{L^2}+\lVert M_k\hat{Q}_k\rVert^2_{L^2}\right)\leq C_2^{-\frac{5}{3}}(CKF_1+DF).
		\end{equation}
		Similar to \eqref{l0}-\eqref{l3}, we have
		\begin{align*}
			&L_0\leq\frac{2\delta_0}{C_2}(DF+CKF_1),\\  &L_1\leq\left(\frac{1}{10|\beta|}+\frac{50}{(C_2)^{\frac{5}{3}}|\beta|}\right)(DF+CKF_1),\\
			&L_2\leq\frac{2}{|C_2|^{\frac{5}{3}}|\beta|}(DF+CKF_1),
		\end{align*}
		and
		\begin{equation*}
			L_3\leq\left(\frac{1}{10|\beta|}+\frac{20}{|C_2|^{\frac{2}{3}}|\beta|}\right)(DF+CKF_1).
		\end{equation*}
		Note $L4=L5=0$, $L_6\leq\frac{1}{50}DF$, and 
		\begin{equation*}
			D_k\leq\left(1+\frac{2}{C_2}\right)(DF+CKF_1).
		\end{equation*}
		By choosing $C_2>3000$,
		we again arrive at \eqref{klinearest}. Then integrating it in time and summing over $k\in \mathbb{Z}^*$, we deduce \eqref{linest} for $0<\mu^3\leq\nu\leq\mu\leq1$.\\
		\textbf{Case 2. $0<\mu\leq\nu\leq\mu^{\frac{1}{3}}\leq1$}\\
		The proof is exactly the same as case 1, except exchanging the role of $\nu$ and $\mu$. For instance, when $|k|\leq C_2\mu^{-1/2}$, 
		we use $|\partial_tp_k|/|k|\leq2\sqrt{p_k}$ and the assumption $\nu\leq\mu^{\frac{1}{3}}$ to get 
		\begin{equation*}
			L_1\leq\frac{\nu}{10|\beta|}\left\lVert\sqrt{p_k}M_k\hat{Z}_k\right\rVert^2_{L^2}+\frac{20\nu}{|\beta|}\lVert M_k\hat{Q}_k\rVert^2_{L^2}\leq\left(\frac{1}{10|\beta|}+\frac{20}{C_2|\beta|}\right)(DF+CKF_1).
		\end{equation*}
		Therefore, we conclude \eqref{linest}.
		The definition of $M$ and the coercivity of $E$ yield \eqref{zbd}. For \eqref{omgbd}, we take inner product of following linearized equations with $M\Omega_{\neq},MJ_{\neq}$ separately
		\begin{equation*} 
			\left\{\begin{array}{l}
				\partial_t(M\Omega_{\neq})-(\partial_tM)\Omega_{\neq}-\beta\partial_XMJ_{\neq}-\nu\Delta_LM\Omega_{\neq}=0,\\
				\partial_t(MJ_{\neq})-(\partial_tM)J_{\neq}-\beta\partial_XM\Omega_{\neq}-\mu\Delta_L MJ_{\neq}+2\partial_{XY}^LM\Phi_{\neq}=0.
			\end{array}\right.
		\end{equation*}
		Then we get
		\begin{equation*}
			\frac{1}{2}\frac{\rm{d}}{\rm{d}t}(\left\lVert M\Omega_{\neq}\right\rVert^2_{L^2}+\left\lVert MJ_{\neq}\right\rVert^2_{L^2})\leq\left|\left\langle\frac{\partial_t\Delta_L}{\Delta_L}MJ_{\neq},MJ_{\neq}\right\rangle\right|\leq2\lVert MQ\rVert_{L^2}\lVert MJ_{\neq}\rVert_{L^2}.
		\end{equation*}
		Integrating over time implies
		\begin{align*}
			(\left\lVert M\Omega_{\neq}\right\rVert_{L^2}+\left\lVert MJ_{\neq}\right\rVert_{L^2})(t)&\lesssim(\left\lVert M\Omega_{\neq}\right\rVert_{L^2}+\left\lVert MJ_{\neq}\right\rVert_{L^2})(0)+\int_{0}^{t}\lVert MQ\rVert_{L^2}(\tau)\rm{d}\tau\\
			&\lesssim(\left\lVert M\Omega_{\neq}\right\rVert_{L^2}+\left\lVert MJ_{\neq}\right\rVert_{L^2})(0)+t(\left\lVert MZ\right\rVert_{L^2}+\left\lVert MQ\right\rVert_{L^2})(0)\\
			&\lesssim\langle t\rangle(\left\lVert M\Omega_{\neq}\right\rVert_{L^2}+\left\lVert MJ_{\neq}\right\rVert_{L^2})(0),
		\end{align*}
		where we used $\left\lVert MZ\right\rVert_{L^2}+\left\lVert MQ\right\rVert_{L^2}\leq\left\lVert M\Omega_{\neq}\right\rVert_{L^2}+\left\lVert MJ_{\neq}\right\rVert_{L^2}$. By the definition of $M$ \eqref{m}, we deduce \eqref{omgbd}.
		Finally to estimate $U$, using $\langle|k,\eta-kt|\rangle\langle|k,\eta|\rangle\gtrsim\langle|kt|\rangle$, we have
		\begin{align*}
			&\left\lVert U^1_{\neq}\right\rVert^2_{H^N}=\sum_{k\neq0}\int_{\mathbb{R}}\langle|k,\eta|\rangle^{2N}\left|\frac{|\eta-kt|}{k^2+(\eta-kt)^2}\hat{\Omega}_k\right|^2{\rm d}\eta\leq\sum_{k\neq0}\int_{\mathbb{R}}\langle|k,\eta|\rangle^{2N}\left|\hat{Z}_k\right|^2\mathrm{d}\eta=\left\lVert Z\right\rVert^2_{H^N},\\
			&\left\lVert U^2_{\neq}\right\rVert^2_{H^{N-1}}=\sum_{k\neq0}\int_{\mathbb{R}}\langle|k,\eta|\rangle^{2N-2}\left|\frac{|k|}{k^2+(\eta-kt)^2}\hat{\Omega}_k\right|^2\mathrm{d}\eta\leq\sum_{k\neq0}\int_{\mathbb{R}}\langle|k,\eta|\rangle^{2N}\left|\frac{1}{\langle t\rangle}\hat{Z}_k\right|^2\mathrm{d}\eta=\frac{1}{\langle t\rangle}\left\lVert Z\right\rVert^2_{H^N}.
		\end{align*}
		Then by \eqref{zbd}, we conclude \eqref{ubd} and complete the proof.
	\end{proof}
	
	\begin{proof}[Proof of Proposition \ref{linearest2}]
		The estimates are similar as \cite{K24}, and we just give a sketch of bound for $L_1$ and $L_6$ in \eqref{error}. When $|k|\leq C_2\mu^{-\frac{1}{2}}$, for $L_1$, noting that
		\begin{equation*}
			\nu\leq\frac{20\nu}{1+\left|\nu\left(t-\frac{\eta}{k}\right)\right|^2},\ \mathrm{if}\ \left|t-\frac{\eta}{k}\right|\leq4\nu^{-1},
		\end{equation*}
		we get
		\begin{equation*}
			L_1\leq\frac{100}{C_4|\beta|}\left\|\sqrt{-\frac{\partial_tM^4_k}{M^4_k}}M_k\hat{Q}_k\right\|_{L^2}^2+\frac{\nu}{10|\beta|}\|\sqrt{p_k}M_k\hat{Z}_k\|_{L^2}^2.
		\end{equation*}
		Turning to $L_6$, by the definition of $M^5$ and the fact
		\begin{equation*}
			\left|\frac{\partial_tp_k}{2p_k}\right|\leq\frac{\mu}{50}p_k,\ \mathrm{if}\ |t-\frac{\eta}{k}|\geq4\mu^{-\frac{1}{3}},
		\end{equation*} we obtain
		\begin{equation*}
			L_6\leq\left\|\sqrt{-\frac{\partial_tM^5_k}{M^5_k}}M_k\hat{Q}_k\right\|_{L^2}^2+\frac{1}{50}DF.
		\end{equation*}
		If $|k|>C_2\mu^{-\frac{1}{2}}$, by \eqref{bdm2}, we have
		\begin{equation*}
			L_1+L_6\leq\frac{200}{C_2}\left(\left\|\sqrt{-\frac{\partial_tM^2_k}{M^2_k}}M_k\hat{Q}_k\right\|_{L^2}^2+DF\right)+\frac{1}{4}DF.
		\end{equation*}
		The estimates of other terms are similar to the treatment in Proposition \ref{linearest}, completing the proof.
	\end{proof}
	\begin{proof}[Proof of Proposition \ref{linearest3}]
		If $|k|\leq C_2\nu^{-\frac{1}{2}}$, the proof is similar to the argument in \cite{ZZ23}, and if $|k|> C_2\nu^{-\frac{1}{2}}$, the proof is similar to treatment above, therefore, we omit further details.
	\end{proof}
	
\vspace{4 mm}
	
\noindent \textbf{Acknowledgements:} The work of FW is supported by the National Natural Science Foundation of China (No. 12101396, 12471223, 12331008, and 12161141004).
	\bibliography{ref}
\end{document}